\documentclass{article}

\usepackage{amssymb}
\usepackage{amsmath}
\usepackage{amsthm}
\usepackage{graphicx}
\usepackage{subcaption}
\usepackage{mathtools}

\usepackage[main, final]{neurips_2026}

\usepackage[utf8]{inputenc} 
\usepackage[T1]{fontenc}    
\usepackage{hyperref}       
\usepackage{url}            
\usepackage{booktabs}       
\usepackage{amsfonts}       
\usepackage{nicefrac}       
\usepackage{microtype}      
\usepackage{xcolor}         
\usepackage{tikz}
\usepackage{enumitem}

\setlist[itemize]{leftmargin=1.5em}

\newtheorem{theorem}{Theorem}[section]
\newtheorem{remark}{Remark}[section]
\newtheorem{corollary}{Corollary}[section]

\newtheorem{lemma}{Lemma}[section]

\newtheorem{proposition}{Proposition}[section]

\newcommand{\KL}{\operatorname{KL}}
\newcommand{\thetalo}{\theta_-}
\newcommand{\thetahi}{\theta_+}
\newcommand{\R}{\mathbb{R}}
\newcommand{\Pbb}{\mathbb{P}}
\newcommand{\E}{\mathbb{E}}
\newcommand{\Cov}{\operatorname{Cov}}
\newcommand{\Var}{\operatorname{Var}}
\newcommand{\supp}{\operatorname{supp}}
\newcommand{\ind}{\mathbf 1}
\newcommand{\norm}[1]{\lVert #1\rVert}
\newcommand{\abs}[1]{\lvert #1\rvert}
\newcommand{\op}{\mathrm{op}}
\newcommand{\scr}{\mathrm{scr}}

\definecolor{myblue}{RGB}{36,72,175}

\usepackage[utf8]{inputenc} 
\usepackage[T1]{fontenc}    
\usepackage{hyperref}       
\usepackage{url}            
\usepackage{booktabs}       
\usepackage{amsfonts}       
\usepackage{nicefrac}       
\usepackage{microtype}      
\usepackage{xcolor}         

\title{On Minimal Observation Time for Recovering \\ Latent Hawkes Networks}

\author{%
  Jonas Linkerhägner \\
  University of Basel\\
  \texttt{jonas.linkerhaegner@unibas.ch} 
  \And
  Michele Bortolasi\\
  University of Basel\\
  \texttt{michele.bortolasi@unibas.ch}
  \And
  Lorenzo Baldassari\\
  University of Basel\\
  \texttt{lorenzo.baldassari@unibas.ch}
  \And
  Maarten V. de Hoop\\
  Rice University \\
  \texttt{mvd2@rice.edu}
  \And
  Ivan Dokmanić\\
  University of Basel\\
  \texttt{ivan.dokmanic@unibas.ch}
}

\begin{document}

\maketitle

\begin{abstract}
Dynamics of interacting systems in engineering, society, and nature often evolve over latent networks that govern which entities can interact. We study the problem of inferring these networks from event-based observations, which arise naturally in finance, seismology, and neuroscience. While there is substantial algorithmic work addressing this important problem, theoretical results are scarce. In this paper we ask the following fundamental question: what is the minimum time that one must observe the dynamics in order to exactly recover the underlying network, as a function of the number $d$ of interacting entities? For a class of stationary Hawkes processes with sparse, weak interactions, we prove that an observation time of order $\log d$ is necessary and sufficient. For the upper bound we construct a two-stage estimator that uses clipped and binned event data for screening, followed by a least-squares refinement, and apply concentration bounds derived from the Poisson cluster representation. For the lower bound we combine Fano’s inequality with Jacod’s Girsanov formula for point processes on a suitable subclass of networks. 

\end{abstract}

\section{Introduction}\label{sec:intro}

Reconstructing the latent interaction network is a central problem in the statistical analysis of point processes.
In this paper, we ask:  
\emph{How must the observation time scale with the number of nodes $d$ in order to recover the directed interaction network exactly?}

We give an answer for Hawkes processes \cite{achab2018uncovering, bacry2020sparse, embrechts2018hawkes}, when excitations of events occur over a sparse, weakly coupled interaction network.
Such processes model a variety of phenomena, including aftershock propagation after earthquakes \cite{ogata1988statistical, zhuang2002stochastic, marsan2008extending}, high-frequency trading \cite{bowsher2007modelling, large2007measuring, bacry2012non}, neuronal spike trains \cite{truccolo2005point, pillow2008spatio}, social activity and information cascades \cite{zhou2013learning, he2015hawkestopic}, and epidemic spreading \cite{chiang2022hawkes}.

Suppose, for example, that we record earthquake occurrences in a seismic region. 
The nodes of the underlying network could be subsurface elements such as fault segments or spatial cells; the edges may describe direct triggering influence, mediated, for instance, by stress transfer: that is, how a seismic event in one element modulates the probability that events will occur in other elements.
Recovering this directed network helps understand how seismic activity is coordinated and how triggering effects propagate \cite{nandan2022are, kaya2026deep}.

Such interacting systems are high-dimensional but sparse: a seismic region may contain many fault segments, but each segment directly interacts only with a small number of nearby segments as dictated by the fault geometry.
This motivates studying how the observation time required for network reconstruction scales with the number of nodes \(d\), when each node is directly influenced by fixed \(k \ll d\) others, and these direct influences are individually moderate.
\begin{figure}
    \centering
    \includegraphics[width=0.99\linewidth]{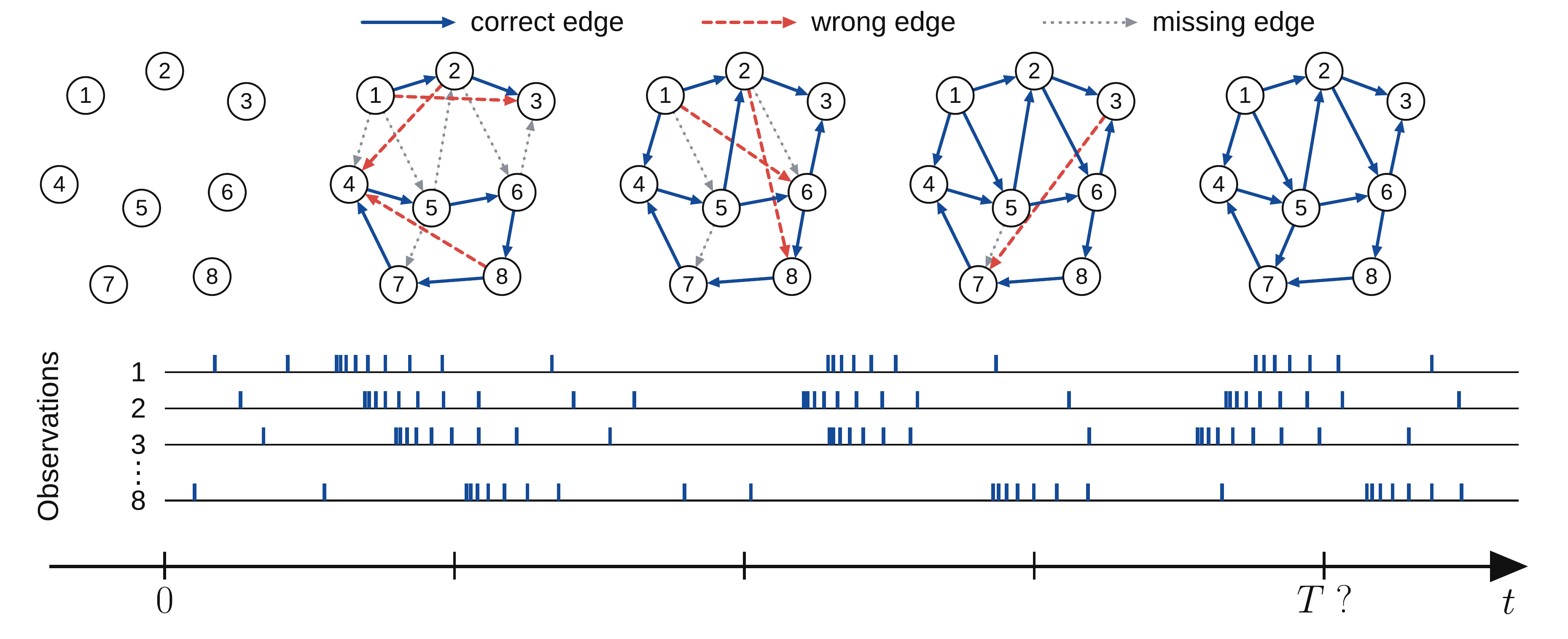}
    \vspace{-0.6em}
    \caption{We observe event streams for each node over a window $[0,T]$, and the goal is to recover the hidden directed interaction network. Independent of the choice of a specific estimator, for short observation times, different networks remain statistically indistinguishable, leading to wrong and missing edges. With sufficiently long time, the true network becomes identifiable. Our work quantifies how the observation-time requirement scales with the network size $d$.}
    \label{fig:illustration}
    \vspace{-0.6em}
\end{figure}
Figure~\ref{fig:illustration} provides an illustration. 
If the observation time $T$ is too small, the observations do not yet include enough information to recover the network correctly. When $T$ is of the correct order in $d$, exact recovery becomes possible.

Here we provide a theoretical analysis of the network recovery problem. 
We prove that for a fixed Hawkes model, with fixed sparsity $k$ and sufficiently weak coupling $\alpha$, the observation time required for exact network recovery scales as $\log d$.
The upper bound in Section~\ref{sec:upper-bound} is constructive. It uses a two-stage estimator that utilizes clipped and binned event data for screening, followed by local least-squares regression and thresholding. Combined with concentration bounds derived from the Poisson cluster representation \cite{hawkes1974cluster}, it shows that exact support recovery is achievable with high probability when
\[
T \gtrsim   \log d .
\]
In Section~\ref{sec:lower-bound}, we prove a matching information-theoretic lower bound by applying Fano's inequality \cite{gerchinovitz2020fano} to an explicit subclass of interaction networks. This reduces exact recovery to controlling the mutual information between the latent network and the observations, which we bound via the corresponding Kullback--Leibler divergences using Jacod’s Girsanov formula for point processes \cite{jacod1975multivariate}.
As a consequence, for all $d$ large enough, any estimator fails to achieve exact support recovery with non-vanishing probability when the observation time is  $T\lesssim\log d$.  

\subsection*{Related Work}
Hawkes processes are self- and mutually-exciting point processes \cite{hawkes1971spectra}, where events can increase the probability of future events. 
The cluster representation characterizes events as branching cascades \cite{hawkes1974cluster}: one event may trigger further events, which may again trigger subsequent generations of events.
Classical work on stability and stationarity conditions \cite{bremaud1996stability, bremaud2002rate} and likelihood-based estimation methods \cite{ogata1978estimators, ozaki1979maximum} provide the foundations for our analysis.
Since then, Hawkes processes have been widely used to model events such as aftershocks \cite{nandan2022are}, neuronal spike trains \cite{kobayashi2019reconstructing}, social activities \cite{crane2008robust, rizoiu2017expecting, chen2018marked}, epidemics \cite{meyer2012space, unwin2021using}, crime \cite{mohler2011self, clark2018modeling, zhuang2019semiparametric} and trades and volatility in finance \cite{bacry2015hawkes, ait2015modeling, bondi2024rough}.

Existing approaches to recover latent Hawkes networks use penalized \cite{hansen2015lasso, bacry2020sparse, cai2024latent} and nonparametric estimation \cite{bacry2016first, donnet2020hawkes}, causal graph recovery \cite{achab2018uncovering}, and graph screening \cite{chen2017nearly}.
More recent approaches include Bayesian structure learning \cite{linderman2014discovering, rasmussen2013bayesian, malem2022variational} and Granger-causal recovery \cite{xu2016learning, wu2024learning}. 
Neural relational inference uses neural networks to infer relations in dynamical systems \cite{kipf2018neural, pan2024graph, wu2025riva}, including point processes \cite{zhang2021neural}.
Other works focus on learning full intensity functions of point processes \cite{omi2019fully, ludke2023add}, often using Neural ODEs \cite{jia2019neural, zhang2024neural}. We complement this line of work with a theoretical result on how the observation time for exact graph recovery scales with network size.

The main challenge compared to other high-dimensional graphical models \cite{meinshausen2006high} is that for Hawkes processes, the network to be recovered enters the model endogenously: the interactions of the latent intensities are governed by the network we want to recover. This makes the problem challenging, since indirect excitation can mimic direct edges that are not present, and even the absence of events is informative.
It is qualitatively distinct from observation-time scaling results in Gaussian graphical models \cite{wainwright2009information, wainwright2009sharp}, Markov random fields \cite{ravikumar2010high} and Ornstein--Uhlenbeck diffusions \cite{bento2010learning, bento2011information}, where network dependence is linear and noise is additive.

\section{Model, assumptions, and main result}\label{sec:setting}
In this section, we introduce the Hawkes model under study, formulate the exact support-recovery problem, and specify the sparse, weak interaction setting in which our analysis takes place.

\subsection{Exponential Hawkes process}

We consider a $d$-dimensional exponential Hawkes process observed over a stationary time window $[0,T]$. Writing $N:=(N_1,\dots,N_d)$ for the counting process, and using the exponential kernel with decay rate $\beta>0$, let 
\[
X_j(t):=\int_{(-\infty,t]} e^{-\beta (t-s)}\, dN_j(s),
\]
denote the exponential shot-noise process associated with node $j$. Each coordinate evolves according to the SDE
\[
    dX_j(t)=-\beta X_j(t)\,dt+dN_j(t).
\]
Thus, with the exponential kernel, the Hawkes process admits a Markov representation. The state variable used throughout the proofs is \[ 
X(t)=(X_1(t),\dots,X_d(t)).
\] 
The conditional intensity of node $i$ is
\[
\lambda_i(t)=\mu_i+\sum_{j=1}^d \theta_{ij} X_j(t-), \qquad i\in[d].
\]
Here, $\mu_i>0$ is the background intensity of node $i$ and $[d]=\{1,\dots ,d\}$. $X_j(t-)$ denotes the left limit of the state just before time $t$, ensuring that the intensity is predictable. The matrix 
\[
\Theta := (\theta_{ij})_{i,j\in[d]}
\]
is the interaction matrix. We consider the excitatory case, so for all $i,j \in [d]$ we have $\theta_{ij}\geq0$. 
This model is a standard choice for multivariate event data with self- and cross-excitation, where $\theta_{ij}>0$ means that the activity at node $j$ directly increases the event rate at node $i$ through the current state $X_j(t-)$. $\theta_{ij}=0$ means that there is no edge from $j$ to $i$.
We do not consider inhibitory Hawkes processes in this work, since this would require additional constraints or a nonlinear intensity function to ensure that the intensities remain nonnegative.

To focus narrowly on the task of interaction recovery, we treat the decay rate $\beta>0$ as known throughout the paper. Also, we take the observed data to be 
\[ Y_T=(X(0),N_{[0,T]}), 
\] 
where $N_{[0,T]}$ is the counting process restricted to the window $[0,T]$. For exponential Hawkes with known $\beta$, \(X(0)\) is the Markov state summarizing the past. The fact that \(X(0)\) is treated as observable is an initialization assumption rather than additional information. In practice, \(X(0)\) is available in stationary simulations after burn-in, and can be approximated from a warm-up window of pre-\(0\) event data.

\subsection{Exact support-recovery}

Our goal is exact network recovery from the observed data $Y_T$. 
In practice, this corresponds, for example, to inferring regions with a higher aftershock risk from observing earthquakes, identifying assets that transmit volatility shocks to other assets, or identifying epidemic spreading routes from observing infections. 
Since $\theta_{ij}>0$ corresponds to a directed edge from node $j$ to node $i$, network recovery corresponds to  support recovery of the interaction matrix $\Theta$.
For each row $i$, let
\[
S_i := \{j \in [d] : \theta_{ij} > 0\}
\]
denote the parent set of node $i$, that is, the set of source nodes that directly excite node $i$. We ask how the required observation time $T$ must scale with the ambient dimension $d$ to recover the full support pattern $(S_i)_{i=1}^d$ exactly with high probability. 

\subsection{Sparse weak-interaction model class}

We now formalize the sparse, weak-interaction setting in which exact support recovery will be studied. We fix constants
\[
\beta >0, \qquad 0<\mu_-\le \mu_+<\infty,
\qquad
0<w_-\le w_+<\infty,
\]
an integer $k\ge 1$, and a coupling level $\alpha>0$. 

\paragraph{Sparse network structure.} Let $\mathcal B_{d,k}$ denote the class of nonnegative matrices $B=(b_{ij})\in\R_+^{d\times d}$ such that
\begin{itemize}
\item each row of $B$ has at most $k$ nonzero entries:
\[
\|b_{i\cdot}\|_0\le k,
\qquad i\in[d];
\]
\item every nonzero entry of $B$ lies in $[w_-,w_+]$:
\[
b_{ij}\neq 0
\quad\Longrightarrow\quad
b_{ij}\in[w_-,w_+].
\]
\end{itemize}

$\mathcal B_{d,k}$ specifies the structural part of the model class: each node has at most $k$ direct parents, and every nonzero edge weight is bounded above and below.

\paragraph{Hawkes parameters.} We then consider Hawkes models whose interaction matrix is of the form
\[
\Theta=\alpha B,
\qquad
B\in\mathcal B_{d,k},
\]
and whose background intensities satisfy
\[
\mu_i\in[\mu_-,\mu_+],
\qquad i\in[d].
\]
In addition, we define 
\[
\thetalo := \alpha w_-,
\qquad
\thetahi := \alpha w_+,
\qquad
\gamma := \frac{k\thetahi}{\beta},
\]
where \(\gamma\) is a uniform upper bound on the total excitation that
one row can receive, measured relative to the exponential decay rate \(\beta\).
Indeed, since each row has at most \(k\) nonzero entries and each interaction weight is at most \(\thetahi\),
\[
\frac{1}{\beta}\rho(\Theta)
\le
\frac{1}{\beta}\|\Theta\|_\infty
\le
\frac{k\thetahi}{\beta}
=
\gamma .
\]
We impose the subcriticality condition \(\gamma<1\), a sufficient form of the standard Hawkes stability condition \cite{hawkes1971spectra}. It guarantees existence of a unique stationary law and excludes explosive excitation cascades.
Together, this yields the class of Hawkes parameters studied in the paper:
\[
\mathcal G_{d,k}
:=
\Bigl\{
(\mu,\Theta):
\Theta=\alpha B,\ B\in\mathcal B_{d,k},\ \mu_i\in[\mu_-,\mu_+]\ \forall i,\ \gamma<1
\Bigr\}.
\]
The weak-interaction setting captures large sparse systems where individual interactions are weak, such as neuronal networks, social activity cascades, and epidemic propagation. Since excitation is sufficiently moderate in these cases, the contribution of a true parent remains detectable, while higher-order cascade effects remain controlled. This is precisely the regime in which one can hope to prove sample complexity bounds by combining support separation with concentration uniformly in $d$.

Finally, the smallest possible nonzero interaction strength in this class has to fulfill $\thetalo>0$.
This is the minimum-signal condition needed when $d$ grows: without a positive lower bound on nonzero edges, true interactions could become arbitrarily weak, and exact support recovery becomes impossible.

\subsection{Main result}
In the regime set by $\mathcal G_{d,k}$, our main result identifies the correct scaling of the observation time with dimension $d$ for exact network recovery: for fixed $k$ and fixed weak coupling level $\alpha$,\footnote{See Appendix~\ref{app:intrinsic_alpha} for a preliminary discussion of the dependence
on the remaining parameters, whose optimal scaling remains open.} the necessary and sufficient observation-time scaling for exact support recovery is
\[
    T \asymp \log d.
\]
To prove this, we combine an achievability result in Theorem~\ref{thm:upper} with an impossibility result in Theorem~\ref{thm:lower-bound}, provided  in Sections~\ref{sec:upper-bound} and~\ref{sec:lower-bound}, respectively. 

\section{Upper bound on the observation time}\label{sec:upper-bound}
In this section, we prove the upper bound by constructing an explicit estimator. Since the estimator serves as a proof tool rather than as a practical recovery algorithm, we do not discuss its computational aspects here.
Figure~\ref{fig:estimator_overview} summarizes the procedure.
The estimator is built node by node because exact support recovery is equivalent to identifying, for each node $i$, the set of nodes with a direct influence on it. The estimator proceeds in two stages.

The screening stage asks which nodes are active shortly before node $i$ has an event. 
This is measured by the empirical covariance score between the current activity of a candidate node $j$ and the occurrence of a future event of node $i$.
The activity is reconstructed from the observed event counts, sampled at time resolution $h$, and clipped at level $R$.
Screening does not require knowing the exact maximum number of parents per node $k$, but only an upper bound $m\geq k$. The $m$ nodes with the largest scores form a candidate list of possible parents. 
In the weak-interaction regime, these scores are shown to separate true parents from non-parents, ensuring that the true support is contained in the candidate set (sure screening). 

In the second stage, the estimator reevaluates this reduced set to separate genuine direct influences from multi-hop correlations.
It fits a local ordinary least-squares regression on the selected candidates and thresholds the estimated coefficients.
Row sparsity keeps this low-dimensional, while clipping and binning allow the relevant empirical quantities to concentrate.

Overall, this strategy turns the high-dimensional recovery problem into a collection of smaller, local subproblems. We now make the construction precise.
\paragraph{Definition of the estimator.}
It is defined row by row. Fix a target node $i\in[d]$.
We choose a bin size $h>0$, a clipping level $R>0$, and a candidate set size $m\ge k$.
Let $n=\lfloor T/h\rfloor$,
so that we use the discrete time points $0,h,\dots,(n-1)h$.
For each bin $r=0,\dots,n-1$, we record two quantities.
First, we sample the state on a grid and define
\[
U_{j,r}:=X_j(rh),
\qquad
U_r=(U_{1,r},\dots,U_{d,r}),
\]
and the clipped and sampled state of node $j$ as
\[
\psi_R(x):=\min\{x,R\},
\qquad
Z_{j,r}^{(R)}:=\psi_R(U_{j,r}).
\]
Second, we record the binned event indicator for each node $i$ as
\[
Y_{i,r}^{(h)}:=\ind\{N_i((rh,(r+1)h])\ge 1\}.
\]
$Z_{j,r}^{(R)}$ is the clipped and sampled candidate state, measuring the past activity of node $j$, while $Y_{i,r}^{(h)}$ measures whether node $i$ has an event shortly afterwards.

The first stage screens candidate parents using the empirical covariance between these two quantities:
\begin{equation}
\label{eq:screen-stat}
\widehat F_{ij}^{(h,R)}
:=
\frac1n\sum_{r=0}^{n-1} Z_{j,r}^{(R)}Y_{i,r}^{(h)}
-
\Bigl(\frac1n\sum_{r=0}^{n-1}Z_{j,r}^{(R)}\Bigr)
\Bigl(\frac1n\sum_{r=0}^{n-1}Y_{i,r}^{(h)}\Bigr).
\end{equation}
For the target node $i$, let $\widehat C_i$ be the set of indices corresponding to the $m$ largest values of $\{\widehat F_{ij}^{(h,R)}:j\in[d]\}$.
The set $\widehat C_i$ is the screening set: it contains the nodes whose clipped and sampled activity is 
\newpage
\begin{figure}
    \centering
    \includegraphics[width=0.99\linewidth]{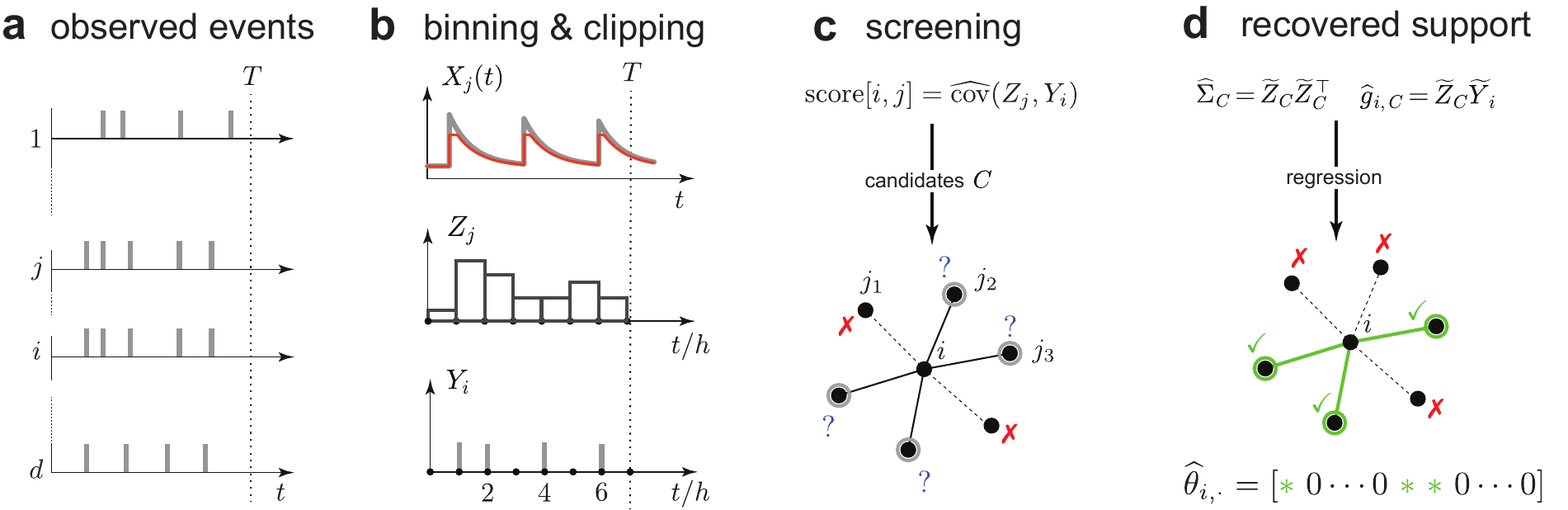}
    \caption{\textbf{{a}} We observe event streams from a Hawkes process with an underlying latent network. To recover the network, our two-stage estimator works node by node. \textbf{{b}} For each node $i$, it converts the observed event times into a binned activity trace $Y_i$. Also, for all candidates $j\in[d]$, it computes the state $X_j(t)$ and its sampled and clipped variant $Z_j$. \textbf{{c}} The first stage screens for nodes that are active shortly before an event at node $i$, retaining the nodes with the highest scores as a candidate set $C$ of possible parents (grey). \textbf{{d}} The second stage centers $Y_i$ and $Z_j$, performs a local regression on this reduced candidate set, and thresholds the fitted coefficients to recover the correct support (green).}
    \label{fig:estimator_overview}
\end{figure}
most correlated with future events of node $i$. The second stage refines this candidate set by a local least-squares regression.
We first center the clipped and sampled state and the binned event indicator,
\[
\widetilde Z_{j,r}:=Z_{j,r}^{(R)}-\frac1n\sum_{r=0}^{n-1}Z_{j,r}^{(R)},
\qquad
\widetilde Y_{i,r}:=Y_{i,r}^{(h)}-\frac1n\sum_{r=0}^{n-1}Y_{i,r}^{(h)}.
\]
For any candidate set $C\subset[d]$, we write
\[
\widetilde Z_{C,r}:=(\widetilde Z_{j,r})_{j\in C}.
\]
We then define the empirical covariance matrix and the empirical cross-covariance:
\begin{equation}
\label{eq:ols-est}
\widehat\Sigma_C^{(R)}
:=\frac1n\sum_{r=0}^{n-1}\widetilde Z_{C,r}\widetilde Z_{C,r}^{\top},
\qquad
\widehat g_{i,C}^{(h,R)}
:=\frac1n\sum_{r=0}^{n-1}\widetilde Z_{C,r}\widetilde Y_{i,r}.
\end{equation}
On the screened set $\widehat C_i$, if $\widehat\Sigma_{\widehat C_i}^{(R)}$ is invertible, we define the local least-squares coefficient vector by
\begin{equation}
\label{eq:local-ols-coeff}
\widehat y_{i,\widehat C_i}^{(h,R)}
:=
\bigl(\widehat\Sigma_{\widehat C_i}^{(R)}\bigr)^{-1}
\widehat g_{i,\widehat C_i}^{(h,R)}.
\end{equation}
The estimated parent set of node $i$ is obtained by thresholding the local least-squares coefficients. We choose the threshold $\tau_{\alpha,h}:=\alpha w_- h/2$ and define
\begin{equation}
    \label{eq:threshold}
    \widehat S_i
:=
\{j\in \widehat C_i:\widehat y_{ij}^{(h,R)}\ge \tau_{\alpha,h}\}.
\end{equation}

We can now state the exact support-recovery result for the bounded-block estimator.
\begin{theorem}
\label{thm:upper}
Fix $m\ge k$. There exists $\alpha_0>0$ such that for every fixed $0<\alpha\le \alpha_0$ there exist constants $ C_*,c_*>0$ only depending on $(\beta, \mu_\pm, w_\pm, k, m, \alpha)$, and $d_0\geq m$ such that for every $d> d_0$, every $B\in\mathcal B_{d,k}$, and every observation time satisfying
\begin{equation*}
T\ge   C_*\log d,
\end{equation*}
for every row $i$, the empirical covariance matrix in \eqref{eq:ols-est} is invertible, and the bounded-block estimator defined in \eqref{eq:screen-stat}, \eqref{eq:local-ols-coeff} and \eqref{eq:threshold} satisfies
\[
\Pbb\bigl(\widehat S_i=S_i, \ \forall i\bigr)
\ge 1-d^{-c_*}.
\]
\end{theorem}

\begin{proof}[Sketch of proof] The proof has three ingredients; the full details are in Appendix~\ref{app:upper_bound}. First, for fixed $(h,R)$, the theoretical screening scores

\[
F^{(h,R)}_{ij}:=\Cov\bigl(Z^{(R)}_{j,0},Y^{(h)}_{i,0}\bigr)
\]
admit the first-order expansion
\[
F^{(h,R)}_{ij}=hH^{(R)}_{ij}+O(Rh^2),
\qquad
H^{(R)}_{ij}\approx \frac{\mu_j\theta_{ij}}{2\beta},
\]
uniformly over the weak-interaction class. More precisely,
for suitable fixed choices
\(R=R(\alpha)\asymp \alpha^{-1}\) and
\(h=h(\alpha)\asymp \alpha^2\),
\[
\min_{j\in S_i}F_{ij}^{(h,R)}
-
\max_{j\notin S_i}F_{ij}^{(h,R)}
\ge c_{\rm scr}\,\alpha w_- h ,
\]
uniformly over \(\mathcal G_{d,k}\); see Proposition~\ref{prop:one-bin-gap}. Thus, every true parent has a strictly
larger theoretical screening score than every non-parent.

Second, the bounded block statistics concentrate uniformly. Since
\(Z_{j,r}^{(R)}\le R\) and \(Y_{i,r}^{(h)}\in\{0,1\}\), the Poisson cluster
representation gives a Bernstein-type inequality for the dependent block
averages. Consequently, with probability at least \(1-d^{-c}\), simultaneously
for all rows, all coordinates, and all candidate sets \(C\) with \(|C|\le m\),
\[
\max_{i,j}|\widehat F_{ij}^{(h,R)}-F_{ij}^{(h,R)}|
+
\max_{i,C}\|\widehat g_{i,C}^{(h,R)}-g_{i,C}^{(h,R)}\|_\infty
+
\max_C\|\widehat\Sigma_C^{(R)}-\Sigma_C^{(R)}\|_{\op}
\le C\delta_T,
\]
where
\[
\delta_T:=
\sqrt{\frac{(h^{-1}+R^2h+R^2/h)\log d}{n}},
\qquad n=\lfloor T/h\rfloor ;
\]
see Lemma~\ref{lem:block-conc-consequences}. If \(C\delta_T\le c_{\rm scr}\alpha w_- h/2\), the empirical screening scores
preserve the screening gap, and hence \(S_i\subseteq \widehat C_i\) for every
row \(i\); see Proposition~\ref{prop:sure-screening}.

Third, on this same high-probability event, the local least-squares step is
stable; see Lemma~\ref{lem:Gram-nondeg}. The theoretical Gram matrices on all sets \(C\supseteq S_i\),
\(|C|\le m\), have eigenvalues bounded away from zero, and their theoretical
regression coefficients satisfy the separation
\[
\min_{j\in S_i} y^{(h,R)}_{ij,C}\ge \frac34\alpha w_- h,
\qquad
\max_{j\in C\setminus S_i}|y^{(h,R)}_{ij,C}|
\le \frac14\alpha w_- h ;
\]
see Proposition~\ref{prop:population-beta}. The preceding concentration bounds and a standard matrix-inversion perturbation
argument imply
\[
\|\widehat y^{(h,R)}_{i,\widehat C_i}
      -y^{(h,R)}_{i,\widehat C_i}\|_\infty
\le C'\delta_T .
\]
Taking \(C'\delta_T\le \alpha w_- h/8\), thresholding at
\(\tau_{\alpha,h}=\alpha w_- h/2\) includes every parent and excludes every
non-parent; see Proposition~\ref{prop:ols-error}. Since \(h(\alpha)\) and \(R(\alpha)\) are fixed once \(\alpha\) is
fixed, the required small-deviation conditions are implied by
\(T\ge  C\log d\). This proves the theorem.
\end{proof}

\section{Information-theoretic lower bound}\label{sec:lower-bound}

Theorem~\ref{thm:upper} shows that there exists a procedure that achieves exact network recovery once the observation time $T$ is of order $\log d$. We now check whether this is the intrinsic scaling of the problem, or whether it comes from how we constructed the procedure. In this section, we prove that the logarithmic scaling of the observation time in $d$ is optimal for the problem by establishing an information-theoretic lower bound. It holds uniformly over all estimators based on the observed process, regardless of their construction or computational cost.
Rather than treating the full model class of Section~\ref{sec:setting} in complete generality, we restrict our attention to a simple subclass. This is sufficient for proving optimality: since Theorem~\ref{thm:upper} already gives an upper bound of order $\log d$ over the full class, it remains only to find a subclass with matching scaling in $d$. Any such lower bound then applies \textit{a fortiori} to the full class.

The subclass we consider consists of interaction matrices with all rows equal to zero except for one target row $i_\star$. 
Fix $i_\star\in[d]$ and background intensities $\bar\mu,\bar\mu_\star\in[\mu_-,\mu_+]$. The only unknown part of the interaction matrix is the support of size $k$ of the $i_\star$-th row. We denote the set of admissible supports by
\[
\mathcal S_{d,k}(i_\star)
:=
\{S\subset[d]\setminus\{i_\star\}: |S|=k\}.
\]
Thus, each $S\in\mathcal S_{d,k}(i_\star)$ specifies which $k$ nodes can influence node $i_\star$. For a given support $S$, we define the corresponding interaction matrix $\Theta^{(S)} = (\theta^{(S)}_{ij})$ by
\[
\theta^{(S)}_{i_\star j}
=
\begin{cases}
\theta_-, & j\in S,\\
0, & j\notin S,
\end{cases}
\qquad
\theta^{(S)}_{ij}=0
\quad\text{for } i\neq i_\star.
\]
Thus, in this subclass, exact support recovery amounts to identifying the unknown set $S$ from the observation $Y_T=(X(0),N_{[0,T]})$. 
This construction deliberately removes most interactions and indirect cascades rather than relying on them for the analysis.

We now place a uniform prior on $S$. The initial uncertainty about the correct support is then
$
\log |\mathcal S_{d,k}(i_\star)|
=
\log{\binom{d-1}{k}}$ nats (using natural logarithmic convention). Therefore, an estimator can succeed only if the observation $Y_T$ carries a comparable amount of information about $S$. This information is measured by the mutual information $I(S;Y_T)$. Fano's inequality turns this intuition into a lower bound on the error probability: if $I(S;Y_T)$ is much smaller than $\log |\mathcal S_{d,k}(i_\star)|$, then many supports remain statistically indistinguishable, and every estimator must fail with non-negligible probability. For more details, see Appendix~\ref{app:it-intro}.

Since $\bar\mu,\bar\mu_\star\in[\mu_-,\mu_+]$ and every active edge has the minimum allowed signal $\thetalo=\alpha w_-$, every $\Theta^{(S)}$, with $S\in\mathcal S_{d,k}(i_\star)$ belongs to $\mathcal G_{d,k}$. Thus, any impossibility result proved for this subclass applies \emph{a fortiori} to the full class $\mathcal G_{d,k}$. This is made precise by the following theorem, whose complete proof is given in Appendix~\ref{app:lower-bound}.

\begin{theorem}\label{thm:lower-bound}
For fixed $k\ge1$, there exist constants $c>0$ and
$\varepsilon\in(0,1)$, depending only on
$(\beta,\bar\mu,\bar\mu_\star,\thetalo,k)$, such that for all sufficiently
large $d$, if $T\le c\log d$, then
\[
\inf_{\widehat S}
\sup_{S\in\mathcal S_{d,k}(i_\star)}
\Pbb_S(\widehat S\neq S)
\ge \varepsilon .
\]
Consequently,
\[
\inf_{\widehat G}
\sup_{(\mu,\Theta)\in\mathcal G_{d,k}}
\Pbb_{\mu,\Theta}
\bigl(\widehat G(Y_T)\neq \supp(\Theta)\bigr)
\ge \varepsilon.
\]
In particular, along any sequence $T_d=o(\log d)$, the worst-case
exact-recovery error over the full class is bounded away from zero. 
Combined with Theorem~\ref{thm:upper}, this identifies the optimal dimension dependence as logarithmic for fixed $k$ and fixed coupling level.
\end{theorem}

\begin{proof}[Sketch of proof]

Let $S^\star$ be uniformly distributed on
$\mathcal S_{d,k}(i_\star)$, and let $Y_T$ be the stationary observation generated under the interaction matrix $\Theta^{(S^\star)}$ and the fixed background intensities $\bar\mu,\bar\mu_\star$. 
Fano's inequality gives, for any estimator $\widehat S=\widehat S(Y_T)$ of $S^\star$,
\[
\sup_{S\in\mathcal S_{d,k}(i_\star)}
\Pbb_S(\widehat S\neq S)
\ge
1-\frac{I(S^\star;Y_T)+\log 2}
{\log |\mathcal S_{d,k}(i_\star)|}.
\]
Since $
\log |\mathcal S_{d,k}(i_\star)|
=
\log{\binom{d-1}{k}}$, it remains to bound the mutual information $I(S^\star;Y_T)$. 
Let $P_S^T$ be the law of the observation $Y_T$ when the parent set is $S$. Let $P_0^T$ be the reference law in which all interaction coefficients are set to zero, while the background intensities $\bar\mu,\bar\mu_\star$ are kept unchanged. The mutual information $I(S^\star;Y_T)$ measures how much information the observation contains about the unknown support $S^\star$. We do not bound this mutual information directly. Instead, we compare every possible law $P_S^T$ with the same law $P_0^T$. Since $P_0^T$ does not depend on $S$, Lemma~\ref{lem:app-mi-baseline} gives
\begin{equation}
I(S^\star;Y_T)
\le
\frac{1}{|\mathcal S_{d,k}(i_\star)|}
\sum_{S\in\mathcal S_{d,k}(i_\star)}
\KL(P_S^T\|P_0^T).
\label{eq:mi-reference-bound}
\end{equation}
It is therefore enough to obtain a bound on $\KL(P_S^T\|P_0^T)$ that holds uniformly over $S\in\mathcal S_{d,k}(i_\star)$.
For each $S$, the KL divergence splits into two terms: one coming from the stationary law of $X(0)$, and one coming from the conditional law of the counting process $N_{[0,T]}$ given $X(0)$; see Lemma~\ref{lem:app-chain}. The stationary-law contribution is bounded by a constant $C_{\rm init}$, independent of both $d$ and $T$. The reason is that, after a permutation of the nodes, the stationary laws differ only on the fixed $(k+1)$-dimensional block formed by node $i_\star$ and its $k$ parents; see Proposition~\ref{prop:app-initKL-k}.

For the conditional-law contribution, the laws $P_S^T$ and $P_0^T$ differ only through the intensity of the counting process $N_{i_\star}$:
\[
\lambda_{i_\star}^{(S)}(t)-\lambda_{i_\star}^{(0)}(t)
=
\thetalo\sum_{j\in S} X_j(t-).
\]
Under $P_0^T$, all interaction coefficients are zero, so the intensity of $N_{i_\star}$ is equal to its background intensity:  $\lambda_{i_\star}^{(0)}(t)\equiv \bar\mu_\star$ for all $t$.
Jacod's point-process Girsanov formula \cite{jacod1975multivariate} therefore yields
\[
\KL(P_S^T\|P_0^T)
\le
C_{\mathrm{init}}
+
\frac{\thetalo^2}{\bar\mu_\star}
\int_0^T
\E_S\Bigl[\Bigl(\sum_{j\in S}X_j(t-)\Bigr)^2\Bigr]\,dt;
\]
see Proposition~\ref{prop:app-dynamic-KL}. Now the advantage of the subclass becomes clear: for $j\neq i_\star$, the processes $X_j$ are independent stationary Poisson shot-noise processes with background intensity $\bar\mu$. Hence,
\[
\E_S\Bigl[\Bigl(\sum_{j\in S}X_j(t-)\Bigr)^2\Bigr]
=
\frac{k^2\bar\mu^2}{\beta^2}
+
\frac{k\bar\mu}{2\beta}
=: C_{\mathrm{path}},
\]
uniformly in $S$, $d$, and $t$; see Lemma~\ref{lem:app-moment-k}. Therefore,
\[
\KL(P_S^T\|P_0^T)
\le
C_{\mathrm{init}}
+
\frac{\thetalo^2}{\bar\mu_\star} C_{\mathrm{path}}\,T.
\]
Substituting this bound into Fano's inequality gives
\begin{equation}
\label{eq:fano-thm}
\inf_{\widehat S}
\sup_{S\in\mathcal S_{d,k}(i_\star)}
\Pbb_S(\widehat S\neq S)
\ge
1-
\frac{
C_{\mathrm{init}}
+
\frac{\thetalo^2}{\bar\mu_\star}C_{\mathrm{path}}T
+
\log 2
}{
\log \binom{d-1}{k}
};
\end{equation}
see Proposition~\ref{prop:app-finite-fano}. Since $k$ is fixed,
$\log\binom{d-1}{k}\asymp\log d$, so the right-hand side is bounded away from zero whenever $T\le c\log d$ for $c>0$ small enough.

Finally, any estimator of the full support matrix over $\mathcal G_{d,k}$ induces, on this subclass, an estimator of the unknown support $S$ by restricting the estimated support to the $i_\star$-th row. Hence the lower bound obtained on the subclass also applies to exact support recovery over the full class $\mathcal G_{d,k}$. 
\end{proof}

Based on Theorem~\ref{thm:lower-bound}, we can make a few remarks:
\begin{itemize}
\item The term $C_{\mathrm{init}}$ is the bound on the information contained in stationary initial state $X(0)$. It does \emph{not} depend on $d$ and does not grow with $T$. Therefore, in the dimension-asymptotic regime with all other parameters fixed, it contributes only an additive $O(1)$ term. The part of the KL bound that grows with the observation time is
\[
\frac{\thetalo^2}{\bar\mu_\star}\,C_{\mathrm{path}}\,T,
\qquad
C_{\mathrm{path}}
=
\frac{k^2\bar\mu^2}{\beta^2}
+
\frac{k\bar\mu}{2\beta}.
\]
Thus the lower-bound argument identifies the necessary information scale
\[
T
\gtrsim
\frac{\bar\mu_\star}{\thetalo^2\,C_{\mathrm{path}}}
\log \binom{d-1}{k},
\]
up to dimension-independent additive constants depending on
$(\beta,\bar\mu,\bar\mu_\star,\thetalo,k)$.
\item The logarithmic dependence on \(d\) is the combinatorial part of the required observation time. The constant prefactor reflects the amount of information generated per unit time. For $d$ sufficiently large, it can be read informally as follows: a larger minimum edge strength $\thetalo$ makes recovery easier, since distinguishability scales quadratically in $\thetalo$; a larger background intensity at node $i_\star$ makes recovery harder, because it increases the background activity against which parental influence must be detected; a larger background activity $\bar\mu$ of the parent process makes recovery easier, as it increases the aggregated signal power $C_{\mathrm{path}}$ of the parents; and a smaller decay rate $\beta$ makes recovery easier, since longer memory strengthens the cumulative effect of parent events.
\end{itemize}

\begin{figure}[t]
    \centering
    \begin{subfigure}[t]{0.48\linewidth}
        \centering
        \includegraphics[width=\linewidth]{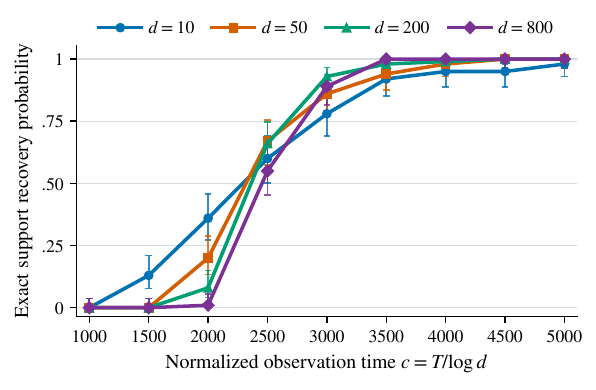}
        \caption{Empirical recovery probability versus normalized
        observation time $c=T/\log d$ for representative $d$.}
        \label{fig:hawkes-recovery-c}
    \end{subfigure}
    \hfill
    \begin{subfigure}[t]{0.48\linewidth}
        \centering
        \includegraphics[width=\linewidth]{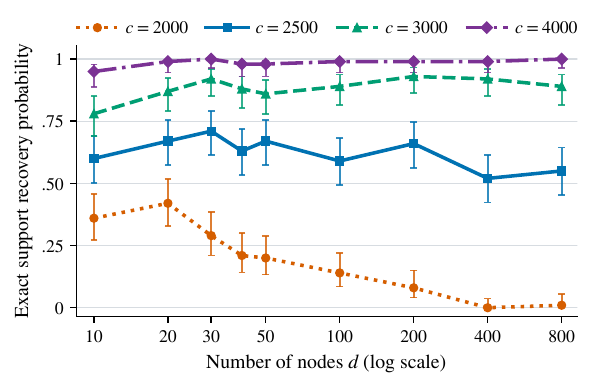}
        \caption{Recovery probability versus number of nodes $d$
        at representative $c$, using a logarithmic horizontal axis.}
        \label{fig:hawkes-recovery-d}
    \end{subfigure}

    \caption{Empirical illustrations of the observation time scaling for finite dimensions. Error bars indicate 95\% Wilson confidence intervals
    based on 100 independent repetitions per configuration.}
    \label{fig:hawkes-recovery-scaling}
\end{figure}

\section{An illustration}
Figure~\ref{fig:hawkes-recovery-scaling}~numerically illustrates that the proven $\log d$ scaling is already visible at finite dimensions. We simulate stationary exponential Hawkes networks with $\beta=1$, $\mu_i=1$, at most $k=5$ parents per node, and homogeneous interaction weights $\theta_{ij}=\alpha=0.1$. 
In each repetition, we sample a new random network support.
The estimator uses $m=10$, $h=0.02$, $R=80$, and $\tau=10^{-3}$. We consider different $d$ at observation times $T=c\log d$, discard a burn-in period of length $50$, and run $100$ independent repetitions. 
Each point reports the fraction of repetitions in which the entire support is recovered exactly, together with its $95\%$ Wilson confidence interval. The smaller dimensions are included to illustrate finite-size effects.
The code is publicly available on \href{https://github.com/jlinki/HawkesObservationTime}{GitHub}.

Figure~\ref{fig:hawkes-recovery-c}~shows that the recovery curves approximately align when the observation time is normalized by $\log d$. 
Figure~\ref{fig:hawkes-recovery-d}~provides a complementary perspective
by plotting recovery probability against network dimension at fixed
values of $c$. Even for large network dimensions $d$,
comparable recovery performance is achieved at approximately the same normalized observation times.

\section{Conclusion and future work}
\label{sec: conclusion}
We proved that the observation time required for exact network recovery in Hawkes processes scales logarithmically with the network size $d$.
This puts the success of empirical work on a theoretical basis: a large network size does not make exact recovery prohibitive, because $d$ only enters through the combinatorial cost of identifying a sparse parent set. 
To the best of our knowledge, this is the first result of its kind.
We also identify the key factors that make our proof possible.
The known exponential kernel provides a Markov state representation, the intensity depends linearly on this state, and stationarity removes burn-in effects while providing uniform distributional control over the observation window.
Weak interactions play an equally important role. 
Indirect excitation cascades may create correlations between non-adjacent nodes and thereby mimic direct causal influence. 
The weak-interaction condition ensures that such effects remain smaller than the contribution of a true parent.
This separation is the central statistical mechanism behind both the screening gap and the stability of the local regression step.

A natural direction for future work is to sharpen the dependence on other parameters, especially the coupling strength.
This is particularly important for applications where the available observation window is usually constrained. 
Other extensions include signed or inhibitory interactions, unknown or coordinate-dependent decay rates, marked Hawkes processes, and spatially structured networks. 
Signed interactions are challenging because cancellations may destroy separation arguments and the intensity may approach zero. 
Marked and spatial models would bring the theory closer to applications.

One speculative but promising application area is earthquake dynamics. Earthquake catalogs are naturally event data, and classical aftershock models already use Hawkes processes to describe triggering. 
It would be interesting to develop a multiscale (marked and spatial) extension of the theory and connect it to emerging network-based forecasting approaches.

\section*{Acknowledgments} This work was funded by the Swiss National Science Foundation (SNSF) under Grant No. 10007161.

\bibliography{references}
\bibliographystyle{plain}



\appendix

\section{Upper bound proof details}
\label{app:upper_bound}
Appendix~\ref{app:upper_bound} contains the proof of the upper bound. 
We organize it into five sections.
We begin with an intuitive description and a short overview before turning to the technical details. 

As presented in Section~\ref{sec:upper-bound}, the whole proof for the upper bound is based on the construction of an estimator composed by a first screening stage and a second Ordinary Least Squares regression stage. The choice for this two-steps procedure is not unique: a method that optimizes over the remaining parameters, or a more ``compressed'' one, likely exists. Nevertheless, we believe these choices simplify the study of the scaling in $d$, which is the main goal of this work. In fact, it is a natural approach to first reduce the search space by choosing a candidate superset and, after guaranteeing that it contains the information about the true interactions, refine the search via regression. Note also that screening first is necessary\footnote{Screening alone could potentially be sufficient. However, since it would not change the result obtained for the scaling in $d$, investigating this possibility further is beyond the scope of this work.} because, at $n = O(log d)\ll d$, the unrestricted empirical Gram matrix for the regression step is singular.

Instead of using the raw and unbounded process, we clip the covariates (i.e., the states or specific combinations of them) at a finite level $R$, and then discretize time into small bins of size $h$. This yields a discrete formulation of the problem, which simplifies the analysis.

The immediate next step is to show that this preprocessing preserves the signal: the covariance used to distinguish true parents of events (trigger nodes) from non-parents must survive clipping and binning, and the resulting regression problem must remain well posed. Once this foundation is established, the next challenge is statistical: the empirical averages computed from the data must be shown to concentrate around their true values. This is nontrivial because the binned observations are not independent, as events in a Hawkes process can trigger later events. The Poisson cluster representation, however, isolates the independent cluster roots, making it possible to control dependence across bins and derive explicit exponential tail bounds. Finally, we combine the theoretical separation between parents and non-parents with these concentration bounds to show that the two-stage estimator succeeds with high probability.

The contents of the sections are outlined in the following list:
\begin{itemize}
    \item \textbf{Section A1} contains some lemmas resulting from the choice of the framework and the assumptions made; these results are not specific to the estimator but rather reflect fundamental properties of stable multivariate Hawkes processes, and they serve as the basic toolkit invoked repeatedly throughout the subsequent sections
    \item \textbf{Section A2} introduces clipping and binning and analyzes the effect on covariances and, consequently, on screening gaps;
    \item \textbf{Section A3} shows that the least-squares regression on clipped covariates, run on any candidate set containing the true parents, produces slope estimates that are meaningful and that the associated Gram matrix remains well-conditioned uniformly;
    \item \textbf{Section A4} derives concentration inequalities for the empirical averages of the bounded observables used by the estimator, by decomposing the Hawkes process into clusters with independent roots, bounding the impact of each cluster on the statistics of interest, and combining these bounds into explicit tail probabilities that decay exponentially in the sample size;
    \item \textbf{Section A5} combines the theoretical screening gap and regression separation with the empirical concentration to show that, with probability tending to one, the two-stage estimator first includes all true parents in the screened superset and then recovers them exactly by thresholding the bounded least-squares coefficients.
\end{itemize}


\subsection{Preliminary Lemmas}
\label{appendix:1}

First of all we prove the uniform bounds for any fixed order moments of $X_j(0)$: 

\begin{proposition}
\label{prop:moments}
For every $P\in\mathbb N$ there exists $\alpha_P>0$, depending only on $(P,\beta,\mu_+,k,w_+)$, such that for all $0<\alpha\le \alpha_P$,
\[
\sup_{d\ge 1}\sup_{B\in\mathcal B_{d,k}}\sup_{1\le j\le d}\E[X_j(0)^p]<\infty,
\qquad p=1,\dots,P.
\]
\end{proposition}
\begin{proof}
Fix $p\ge 1$ and $j$. For $M\ge 1$, let $\phi_M(x)=\min\{x, M\}$ and set $
F_{p,M}(x):=\phi_M(x_j)^p$. Define the approximation $F_{p,M,\epsilon}$ using a smooth mollifier $\phi_{M,\epsilon}$ and apply Dynkin's formula in stationarity with the coherent infinitesimal generator of the process.
Then letting the smoothing parameter tend to zero leads to
\begin{equation*}
\label{eq:moment-recursion-trunc}
\beta p\,\E\!\Big[X_j \min\{X_j, M\}^{p-1}\mathbf 1_{\{X_j<M\}}\Big]
=
\E\Big[\lambda_j\bigl(X(0)\bigr)\bigl(\min\{X_j+1, M\}^p-\min\{X_j, M\}^p\bigl)\Big].
\end{equation*}
Both sides are nonnegative and increase to their untruncated counterparts as $M\uparrow\infty$, so monotone convergence gives
\begin{equation*}
\label{eq:moment-recursion-base}
\beta p\,\E[X_j^p]
=
\E\Big[\lambda_j\bigl(X(0)\bigr)\big((X_j+1)^p-X_j^p\big)\Big].
\end{equation*}
Now, since $(x+1)^p-x^p\le c_p(1+x^{p-1})$ for $x\ge 0$ and
$ \lambda_j\bigl(X(0)\bigr)\le \mu_+ + \theta_+\sum_{\ell\in S_j,\,|S_j|\le k}X_\ell(0)$, we obtain from the previous equation
\begin{equation}
\label{eq:step intermediate}
\beta p\,\E[X_j^p]
\le c_p\mu_+\bigl(1+\E[X_j^{p-1}]\bigr)
+ c_p\theta_+\sum_{\ell\in S_j,\,|S_j|\le k}\E\bigl[X_\ell(1+X_j^{p-1})\bigr].
\end{equation}
Using Young's inequality for the mixed term gives
\[
\E\bigl[X_\ell(1+X_j^{p-1})\bigr]
\le \E[X_\ell] + \E[X_\ell^p]+\E[X_j^p] \le M_1 + M_p + M_p,
\]
where, with the convention $M_0:=1$, $M_p:=\sup_{d,B,j}\E[X_j(0)^p]$.
Therefore taking the supremum over $(d,B,j)$ in \eqref{eq:step intermediate}, after enlarging the constant if necessary, gives:
\[
\beta p\,M_p
\le C_p\mu_+\bigl(1+M_{p-1}\bigr)+C_p\beta\gamma\bigl(M_1+M_p\bigr),
\]
Equivalently, the recursion equation
\[
\beta (p-C_p\gamma)\,M_p
\le C_p\bigl(\mu_+ + \beta \gamma M_1 + \mu_+ M_{p-1}\bigr).
\]
For $p=1$ this yields $\beta(1-2C_p\gamma)M_1\le 2C_p\mu_+$ and for $\alpha$ small enough this implies $M_1<\infty$. Choosing then $\alpha$ small enough in order to have $\beta (p-C_p\gamma)>0$ for $p=1,\dots,P$ proves the claim by induction.
\end{proof}
Before establishing bounds on the covariance of $X_j(0)$, we prove the following lemma regarding the stationary mean.
\begin{lemma}
\label{lem:stability_stationarymean_momen}
Under the stability condition, let $m:=\E[X(0)]$ and $\bar\lambda:=\E[\lambda(0)]$. Then 
\[
    m=(\beta I-\Theta)^{-1}\mu, \qquad \bar \lambda =\beta m.
\]
Moreover, uniformly in $d$ and $B\in\mathcal B_{d,k}$
    \begin{equation}
        \label{eq:mean-bounds}
        \frac{\mu_-}{\beta}
        \le
        \min_j m_j
        \le
        \max_j m_j
        \le
        \frac{\mu_+}{\beta(1-\gamma)},
    \end{equation}
    \begin{equation}
    \label{eq:mean-first-order}
        \max_j\left|m_j-\frac{\mu_j}{\beta}\right|
        \le
        \frac{\gamma\mu_+}{\beta(1-\gamma)}.
    \end{equation}

\end{lemma}
\begin{proof}
Stationarity of the process leads to $0=-\beta m+\bar\lambda$ and $\bar\lambda=\mu+\Theta m$, therefore $(\beta I-\Theta)m=\mu>0$ and $\bar\lambda=\beta m$. The lower bound in \eqref{eq:mean-bounds} follows from $\beta m_j\ge \mu_j\ge \mu_-$, while for the upper bound:
\[
\beta m_j=\mu_j+\sum_\ell\theta_{j\ell}m_\ell
\le \mu_+ + \gamma \beta\max_\ell m_\ell,
\]
especially $\beta(1-\gamma)\max_j m_j\le \mu_+$. Finally, exploiting the stability condition in the inequality,
\[
    \left|m_j-\frac{\mu_j}{\beta}\right|=\frac{1}{\beta}\sum_\ell \theta_{j\ell}m_\ell \leq \gamma \max_\ell m_\ell ,
\]
whence \eqref{eq:mean-first-order} by the upper bound in \eqref{eq:mean-bounds}.
\end{proof}

We are now in the position of proving bounds on the covariance of $X_j(0)$. 

\begin{proposition}
\label{prop:cov-perturb}
There exists $\alpha_{\rm cov}>0$, depending only on $(\beta,\mu_+,w_+,k)$, such that whenever $0<\alpha\le \alpha_{\rm cov}$,
\begin{equation}
\label{eq:cov-bounds-main}
\sup_{j\neq \ell}|\Cov(X_j(0),X_\ell(0))|\le C_{\rm off}\alpha,
\qquad
\sup_j \left|\Var(X_j(0))-\frac{\mu_j}{2\beta}\right|\le C_{\rm diag}\alpha,
\end{equation}
with constants depending only on $(\beta,\mu_-,\mu_+,w_+,k)$.
\end{proposition}
\begin{proof}
Set $m=\E[X(0)]$, $\bar\lambda=\E[\lambda(0)]$ and $Y_t=X_t-m$. Then
\[
dY_t=(\Theta-\beta I)Y_t\,dt+dM_t,
\qquad
M_t:=N_t-\int_0^t\lambda_s\,ds.
\]
Denote $\Sigma:=\Cov(X(0),X(0))$, $ D:=\sup_j \Sigma_{jj}$ and $S:=\sup_{j\neq\ell}|\Sigma_{j\ell}|$: applying the semimartingale product rule to $Y_tY_t^\top$ and taking expectations in stationarity yield the (entrywise-)Lyapunov equation
\begin{equation}
\label{eq:lyapunov-entrywise}
2\beta\Sigma_{j\ell}
=
\sum_r \theta_{jr}\Sigma_{r\ell}
+
\sum_r \theta_{\ell r}\Sigma_{jr}
+
\bar\lambda_j\ind_{\{j=\ell\}}.
\end{equation}
From the Lyapunov equation for $j=\ell$ and using $|\Sigma_{jr}|\le \sqrt{\Sigma_{jj}\Sigma_{rr}}\le D$ and $\sum_r\theta_{jr}\le \gamma \beta$, we get
\[
2\beta\Sigma_{jj}\le 2\gamma \beta D+\bar\lambda_j.
\] 
Taking the supremum over $j$ and using the previous Lemma
\begin{equation}
\label{eq:D-bound}
2\beta(1-\gamma)D\le \max_j \bar\lambda_j \le \frac{\mu_+}{1-\gamma} \qquad \implies \qquad D\le \frac{\mu_+}{2\beta(1-\gamma)^2}=:C_D.
\end{equation}
In a similar manner, if $j\neq \ell$, \eqref{eq:lyapunov-entrywise} yields $2\beta|\Sigma_{j\ell}|\le 2\gamma\beta D$, whence $S\le \gamma D \le (C_D k w_+/\beta)\,\alpha =: C_{\rm off}\alpha$ (which is the stated off-diagonal bound).\\
From the diagonal equation, subtracting $\mu_j/2\beta$ in both sides and taking the absolute value lead to
\[
\left|\Sigma_{jj}-\frac{\mu_j}{2\beta}\right|
\le \frac{1}{\beta}\sum_r \theta_{jr}|\Sigma_{jr}| + \frac{|\bar\lambda_j-\mu_j|}{2\beta}.
\]
By \eqref{eq:D-bound}, the first term in right side is at most $\gamma D=O(\alpha)$, while the second term is $O(\alpha)$ by the previous Lemma because $\bar\lambda_j=\beta m_j$. This proves the diagonal bound in \eqref{eq:cov-bounds-main}. 
\end{proof}

We conclude this section by showing that the screening scores effectively distinguish between the nodes $j\in S_i$ and $j\notin S_i$.
\begin{proposition}
\label{prop:Gexpand}
Define
\[
G_{ij}:=\Cov(X_j(0),\lambda_i(0)).
\]
Assume $0<\alpha\le \alpha_{\rm cov}$. Then there exists  $C_G>0$, depending only on $(\beta,\mu_-,\mu_+,w_+,k)$, such that:
\[
\left|G_{ij}-\frac{\mu_j}{2\beta}\theta_{ij}\right|\le C_G\alpha^2 \quad \text{if } \,j\in S_i \qquad \text{and} \qquad |G_{ij}|\le C_G\alpha^2 \quad \text{if } \,j\notin S_i
\]
Consequently there exists $\alpha_{scr}\in(0,\alpha_{\rm cov}]$ such that for every $0<\alpha\le \alpha_{scr}$,
\begin{equation}
\label{eq:screen-gap-G}
\min_{j\in S_i} G_{ij}-\max_{j\notin S_i}G_{ij}
\ge \Delta_{scr}(\alpha):=\frac{\mu_-w_-}{4\beta}\alpha
\end{equation}
for every row $i$ with $S_i\neq\varnothing$, while $G_{ij}=0$ for all $j$ whenever $S_i=\varnothing$.
\end{proposition}

\begin{proof}
From the definition of $\lambda_i(0)$, if $j\in S_i$, then
\[
G_{ij}=\theta_{ij}\Sigma_{jj}+\sum_{\ell\in S_i\setminus\{j\}}\theta_{i\ell}\Sigma_{j\ell}.
\]
Hence using Proposition~\ref{prop:cov-perturb} and $\theta_{ij}\leq \theta_+$:
\begin{equation}
    \label{eq:minG}
    \left|G_{ij}-\frac{\mu_j}{2\beta}\theta_{ij}\right|
    \le \theta_{ij}\left|\Sigma_{jj}-\frac{\mu_j}{2\beta}\right| + \sum_{\ell\in S_i\setminus\{j\}} \theta_{i\ell}|\Sigma_{j\ell}|
    \le C_G\alpha^2.
\end{equation}
If $j\notin S_i$, then
\begin{equation}
    \label{eq:maxG}
    |G_{ij}| = \left|\sum_{\ell\in S_i}\theta_{i\ell}\Sigma_{j\ell}\right|
    \le k\theta_+ \sup_{j\neq \ell}|\Sigma_{j\ell}|
    \le C_G\alpha^2.
\end{equation}
Finally, using \eqref{eq:minG} and \eqref{eq:maxG}, for $\alpha$ small enough the gap \eqref{eq:screen-gap-G} follows.
\end{proof}

\subsection{Clipping and binning}
\label{appendix:2}
Before we show that the screening signal survives clipping, we need to quantify how much the covariances change when we replace the unbounded covariates $X_j(0)$ by their clipped versions $Z_j^{(R)} = \min\{X_j(0), R\}$. Intuitively, clipping only affects the tails, i.e. the rare events where $X_j(0)$ is very large, but since these events become rarer as $R$ grows, the ``distortion'' introduced should vanish as R$\to \infty$. The following lemma makes this precise, showing that the covariances change by at most $C_{\rm clip}/R$.\\
We fix $h>0$ and $R>0$ and denote $Z_j^{(R)}:=\psi_R(X_j(0))$ and $F_{ij}^{(h,R)}:=\Cov(Z_j^{(R)},Y_{i,0}^{(h)})$.
\begin{lemma}
\label{lem:clip-bias}
There exist a constant $C_{\rm clip}>0$, depending only on $(\beta,\mu_-,\mu_+,w_+,k)$, such that uniformly in $d,B,i,j,\ell$,
\begin{align}
\label{eq:clip-cov1}
\abs{\Cov(Z_j^{(R)},\lambda_i(0)) - \Cov(X_j(0),\lambda_i(0))}
&\le \frac{C_{\rm clip}}{R},\\
\label{eq:clip-cov2}
\abs{\Cov(Z_j^{(R)},Z_\ell^{(R)}) - \Cov(X_j(0),X_\ell(0))}
&\le \frac{C_{\rm clip}}{R}.
\end{align}
\end{lemma}

\begin{proof}
Writing $\delta_j^{(R)}:=X_j(0)-Z_j^{(R)}=(X_j(0)-R)_+$,
\[
\abs{\Cov(Z_j^{(R)},\lambda_i)-\Cov(X_j,\lambda_i)}=\abs{\Cov(\delta_j^{(R)},\lambda_i)}
\le
\E[\delta_j^{(R)}\lambda_i]+
\E[\delta_j^{(R)}]\E[\lambda_i].
\]
Because $\delta_j^{(R)}\le X_j\ind_{\{X_j>R\}}$, Cauchy--Schwarz and Markov give
\[
\E[\delta_j^{(R)}]
\le
\sqrt{\E[X_j^2]\Pbb(X_j>R)}
\le
\frac{\sqrt{\E[X_j^2]\E[X_j^2]}}{R}
\le \frac{C_{\rm clip}}{R}.
\]
Moreover, for every $x\ge 0$, $(x-R)_+^2\le x^2\ind\{x>R\}\le \frac{x^4}{R^2}$. Proposition~\ref{prop:moments} with $P=4$ yields
\begin{equation}
\label{eq:delta-second-moment-main}
\E[(\delta_j^{(R)})^2]\le \frac{\E[X_j^4]}{R^2}\le \frac{C_{\rm clip}}{R^2}.
\end{equation}
Since $\lambda_i\le \mu_+ + \alpha w_+\sum_{\ell\in S_i}X_\ell$ and $|S_i|\le k$, Proposition~\ref{prop:moments} with $P=4$ also implies $\sup_{i,d,B}\E[\lambda_i^2]<\infty$. Therefore
\[
\E[\delta_j^{(R)}\lambda_i]
\le
\sqrt{\E[(\delta_j^{(R)})^2]\E[\lambda_i^2]}
\le
\frac{C_{\rm clip}}{R}.
\]
This proves \eqref{eq:clip-cov1}. The proof of \eqref{eq:clip-cov2} is identical after replacing $\lambda_i$ by $X_\ell$.
\end{proof}
\begin{proposition}
\label{prop:clipped-screen-gap}
Define
\[
H_{ij}^{(R)}:=\Cov\bigl(Z_j^{(R)},\lambda_i(0)\bigr).
\]
There exist $\alpha_1>0$, $R_1:(0,\alpha_1]\to[1,\infty)$, and constants $c_H,C_H>0$, depending only on $(\beta,\mu_-,\mu_+,w_-,w_+,k)$, such that for every fixed $0<\alpha\le \alpha_1$ and every $R\ge R_1(\alpha)$,
\begin{align}
\label{eq:H-parent}
\min_{j\in S_i}H_{ij}^{(R)} &\ge c_H\alpha,\\
\label{eq:H-nonparent}
\max_{j\notin S_i}H_{ij}^{(R)} &\le C_H\alpha^2 + \frac{C_H}{R},
\end{align}
for every row $i$, and therefore
\begin{equation}
\label{eq:H-gap}
\min_{j\in S_i}H_{ij}^{(R)} - \max_{j\notin S_i}H_{ij}^{(R)}
\ge \frac{\mu_-w_-}{8\beta}\alpha.
\end{equation}
And $H_{ij}^{(R)}=0$ for all $j$ whenever $S_i=\varnothing$.
\end{proposition}

\begin{proof}
By Lemma~\ref{lem:clip-bias} and Proposition~\ref{prop:Gexpand}, $H_{ij}^{(R)} = G_{ij} + O(R^{-1})$ uniformly in $(i,j,d,B)$. Using again the bounds in Proposition~\ref{prop:Gexpand}, choose $\alpha_1$  such that $C_G\alpha\le \mu_-w_-/(16\beta)$ for $0<\alpha\le \alpha_1$, and then choose $R_1(\alpha)$ such that $C_{\rm clip}/R\le \mu_-w_-\alpha/(16\beta)$ for all $R\ge R_1(\alpha)$. Using these bounds proves \eqref{eq:H-parent}--\eqref{eq:H-gap}.
\end{proof}

Since we are using discrete bin indicators $Y_{i,0}^{(h)}$, which record whether at least one event of type-$i$ occurred in the first bin, the next Lemmas show that the covariance between $Z_j^{(R)}$ and $Y_{i,0}^{(h)}$ is well approximated by $h$ times the covariance between $Z_j^{(R)}$ and $\lambda_i(0)$, up to an error of order $Rh^2$. Intuitively this is true because, over a short bin of width $h$, the probability of observing at least one event is approximately $h\lambda_i(0)$, so the indicator $Y_{i,0}^{(h)}$ behaves roughly like $h\lambda_i(0)$. We first present the following auxiliary Lemma about the approximation on one bin.
\begin{lemma}
\label{lem:one-bin-expansion}
Then there exists a constant $C_{\rm bin}>0$, depending only on $(\beta,\mu_+,w_+,k)$, such that uniformly in $d,B,i,j$,
\begin{equation}
\label{eq:one-bin-expansion}
\abs{F_{ij}^{(h,R)} - h H_{ij}^{(R)}}
\le C_{\rm bin} R h^2
\qquad (0<h\le 1,\ R\ge 1).
\end{equation}
\end{lemma}

\begin{proof}
Fix $(i,j)$ and write $Z:=\psi_R(X_j(0))$, $N_i^{(h)}:=N_i((0,h])
$ and $\Lambda_i^{(h)}:=\int_0^h \lambda_i(t)\,dt$ (note that $Z$ is $\mathcal F_0$-measurable). By the compensator identity for $N_i$ and using $\ind\{N_i((0,t))\ge 1\}\le N_i((0,t])$ for the inequality, we have
\begin{equation}
    \label{eq:step in lemma9}
    \E\bigl[N_i^{(h)}-Y_{i,0}^{(h)}\bigr]= \E\int_0^h \ind\{N_i((0,t))\ge 1\}\lambda_i(t)\,dt\le \int_0^h \E\bigl[N_i((0,t])\lambda_i(t)\bigr]dt.
\end{equation}
Now, for $t\le 1$, using $X_\ell(t)\le X_\ell(0)+N_\ell((0,t])$
\[
\lambda_i(t)
\le
\mu_+ + \theta_+\sum_{\ell\in S_i}X_\ell(t)
\le
C_1\left(1+\sum_{\ell\in S_i}X_\ell(0)+\sum_{\ell\in S_i}N_\ell((0,t])\right),
\]
Hence the expectation in the previous integral satisfies 
\begin{equation}
\label{eq:bound general}
\begin{split}
    \mathbb{E} \bigl[ N_i((0,t]) \lambda_i(t) \bigr] 
    \le{} & C_1 \mathbb{E} [N_i((0,t])] + C_1 \sum_{\ell \in S_i} \mathbb{E} \bigl[ N_i((0,t]) X_\ell(0) \bigr] \\
    & + C_1 \sum_{\ell \in S_i} \mathbb{E} \bigl[ N_i((0,t]) N_\ell((0,t]) \bigr].
\end{split}
\end{equation}
The first term satisfies $\E[N_i((0,t])]\le C_2t$ by compensator identity and Proposition~\ref{lem:stability_stationarymean_momen}. For the mixed term, using the compensator identity multiple times, $\lambda_i(s)\le C_3(1+\sum_{m\in S_i}X_m(s))$, Cauchy Schwarz, stationarity, and Proposition~\ref{prop:moments} with $P=2$, it holds $\E[N_i((0,t])X_\ell(0)]\le C_3t$. For the last term, we use the martingale decomposition for $N_q((0,t])$ (for any $q$), It\^{o} isometry, and stationarity to state that $M_q$ is a square-integrable martingale, and then for $t\le 1$ it holds 
\[
\E\bigl[N_q((0,t])^2\bigr]
\le
2\E[M_q(t)^2] + 2\E\left[\left(\int_0^t \lambda_q(s)\,ds\right)^2\right]
\le C_4t \qquad \text{uniformly in } (q,d,B),
\]
where again the elementary inequality $(x+y)^2\le C(x^2+y^2)$, the It\^{o} isometry, Cauchy Schwarz for integrals, and Proposition~\ref{prop:moments} with $P=2$ have been used. Then it is easy to bound the term $\E\bigl[N_i((0,t])N_\ell((0,t])\bigr]$ by Cauchy Schwarz. Hence, using $|S_i|\le k$, combining all the bounds obtained above for the right-hand terms in \eqref{eq:bound general} and defining $C_{bin,1}$ as a convenient combination of the constants in \eqref{eq:bound general} lead to $\E\bigl[N_i((0,t])\lambda_i(t)\bigr]\le C_{bin,1}t$,  and substituting this bound in \eqref{eq:step in lemma9} gives:
\[
\E\bigl[N_i^{(h)}-Y_{i,0}^{(h)}\bigr]\le \int_0^h \E\bigl[N_i((0,t])\lambda_i(t)\bigr]dt\le \frac{C_{bin,1}}{2}h^2
\]
Exploiting the facts that $0\le Z\le R$ and that  $\Cov(Z,N_i^{(h)})=\Cov(Z,\Lambda_i^{(h)})$ (this follows from substituting the martingale decomposition in the expectations, because $Z$ is $\mathcal F_0$-measurable and $N_i(t)-\int_0^t\lambda_i(s)ds$ is an $(\mathcal F_t)$-martingale), then it is possible to show the following:
\begin{equation}
\label{eq:indicator-to-comp}
\abs{\Cov(Z,Y_{i,0}^{(h)})-\Cov(Z,N_i^{(h)})}
\le 2\|Z\|_\infty\E\bigl[N_i^{(h)}-Y_{i,0}^{(h)}\bigr]
\le C_{bin,1} R h^2.
\end{equation}
\[
\Cov(Z,N_i^{(h)}) - h\Cov(Z,\lambda_i(0))=\Cov(Z,\Lambda_i^{(h)}) - h\Cov(Z,\lambda_i(0))
=
\int_0^h \Cov\bigl(Z,\lambda_i(t)-\lambda_i(0)\bigr)dt.
\]
As before, it is easy to show that $\abs{\Cov\bigl(Z,\lambda_i(t)-\lambda_i(0)\bigr)}
\le 2R\,\E\abs{\lambda_i(t)-\lambda_i(0)}$. Now:
\[
\lambda_i(t)-\lambda_i(0)=\sum_{\ell\in S_i} \theta_{i\ell}(X_\ell(t)-X_\ell(0)),
\]
Writing down the integral form of SDE for $X_\ell(t)$, taking the absolute value, and using stationarity and Lemma~\ref{lem:stability_stationarymean_momen}, we get 
\[
\E\abs{X_\ell(t)-X_\ell(0)}
\le \beta\int_0^t \E[X_\ell(s)]ds + \E[N_\ell(t)-N_\ell(0)]
\le C_5t.
\]
Therefore, we can define conveniently $C_{bin,2}$ such that $\E|\lambda_i(t)-\lambda_i(0)|\le C_{bin,2}t$ uniformly in $(i,d,B)$, and thus
\[
\abs{\Cov(Z,\Lambda_i^{(h)}) - hH_{ij}^{(R)}}
\le 2R\int_0^h C_{bin,2}t\,dt
\le C_{bin,2} R h^2.
\]
Combining this with \eqref{eq:indicator-to-comp} and defining $C_{bin}:=C_{bin,1}+C_{bin,2}$ proves \eqref{eq:one-bin-expansion}.
\end{proof}
Lemma~\ref{lem:one-bin-expansion}, together with the screening gap already established for the clipped covariances, yields Proposition~\ref{prop:one-bin-gap} below: the separation between true parents and non-parents is preserved by the observable quantities used by the estimator, provided $R$ is sufficiently large and $h$ is sufficiently small relative to $\alpha$.
\begin{proposition}
\label{prop:one-bin-gap}
There exist $\alpha_2>0$, a function $R_2:(0,\alpha_2]\to[1,\infty)$, and a constant $c_{\scr}>0$ depending only on $(\beta,\mu_{\pm},w_{\pm},k)$, such that for every fixed $0<\alpha\le \alpha_2$, every $R\ge R_2(\alpha)$, and every
\[
0<h\le \min\Bigl\{\frac{\mu_-w_-}{32\beta C_{\rm bin}}\frac{\alpha}{R}, 1\Bigr\},
\]
one has, for every row $i$ with $S_i\neq\varnothing$,
\begin{equation}
\label{eq:F-gap}
\min_{j\in S_i}F_{ij}^{(h,R)} - \max_{j\notin S_i}F_{ij}^{(h,R)}
\ge c_{\scr}\alpha h ,
\end{equation}
while $F_{ij}^{(h,R)}=0$ for all $j$ whenever $S_i=\varnothing$.
\end{proposition}

\begin{proof}
If $S_i=\varnothing$, then $\lambda_i(t)\equiv \mu_i$ and $Y_{i,0}^{(h)}$ is independent of $X(0)$, so $F_{ij}^{(h,R)}=0$ for all $j$.
Assume $S_i\neq\varnothing$: by Lemma~\ref{lem:one-bin-expansion}, $F_{ij}^{(h,R)} = hH_{ij}^{(R)} + O(Rh^2)$ and by Proposition~\ref{prop:clipped-screen-gap}, uniformly in $i$,
\[
\min_{j\in S_i}F_{ij}^{(h,R)}
-\max_{j\notin S_i}F_{ij}^{(h,R)}
\ge h\cdot \frac{\mu_-w_-}{8\beta}\alpha - 2C_{\text{bin}} R h^2.
\]
Choose $\alpha_2\le \alpha_1$ from Proposition~\ref{prop:clipped-screen-gap}, and then choose $R_2(\alpha)\ge R_1(\alpha)$. If $R\ge R_2(\alpha)$ and
\[
0<h\le \min\Bigl\{\frac{\mu_-w_-}{32\beta C_{\rm bin}}\frac{\alpha}{R}, 1\Bigr\},
\]
then
\[
C_{\rm bin}Rh\le \frac{\mu_-w_-}{32\beta}\alpha,
\]
so combining this with the previous inequality proves \eqref{eq:F-gap}.
\end{proof}

\subsection{Bounded local regression}
\label{appendix:3}
Fix $R\ge 1$ and a candidate set $C\subset\{1,\dots,d\}$ with $|C|\le m$. Define:
\[
\Sigma_C^{(R)}:=\Cov\bigl(Z_C^{(R)}(0),Z_C^{(R)}(0)\bigr),
\qquad
Z_C^{(R)}(0):=(Z_j^{(R)})_{j\in C}, \qquad g_{i,C}^{(h,R)}:=\Cov\bigl(Z_C^{(R)}(0),Y_{i,0}^{(h)}\bigr).
\]
Before stating the uniform nondegeneracy of the clipped Gram matrix $\Sigma_C^{(R)}$, we define, whenever $\Sigma_C^{(R)}$ is invertible, the centered theoretical least-squares slope as
\begin{equation}
\label{eq:population-slope}
y_{i,C}^{(h,R)}:=\bigl(\Sigma_C^{(R)}\bigr)^{-1}g_{i,C}^{(h,R)}
\end{equation}

\begin{lemma}
\label{lem:Gram-nondeg}
There exist $\alpha_3>0$, $R_3:(0,\alpha_3]\to[1,\infty)$, and $\kappa>0$, depending only on $(\beta,\mu_-,\mu_+,w_+,k,m)$, such that for every fixed $0<\alpha\le \alpha_3$, every $R\ge R_3(\alpha)$, every $d$, every $B\in\mathcal B_{d,k}$, and every $C\subset\{1,\dots,d\}$ with $|C|\le m$, we have $\lambda_{\min}(\Sigma_C^{(R)})\ge \kappa$.
\end{lemma}
\begin{proof}
Under $\alpha=0$ the coordinates are independent shot-noise processes. Hence, for any fixed $C$:
\[
\Sigma_{C,0}^{(R)} = \text{diag}\Bigl(\Var_0\bigl(\psi_R(X_j(0))\bigr)\Bigr)_{j\in C}.
\]
Now $\Var_0(X_j(0))=\mu_j/(2\beta)\ge \mu_-/(2\beta)$, then Lemma~\ref{lem:clip-bias} with $i=j$ and $\ell=j$ and choosing $R$ large enough gives
$\Var_0\bigl(\psi_R(X_j(0))\bigr)\ge \mu_-/4\beta$ for all $j$ . Thus $\lambda_{\min}(\Sigma_{C,0}^{(R)})\ge \mu_-/4\beta$ for all $|C|\le m$.
Next compare $\Sigma_C^{(R)}$ to $\Sigma_{C,0}^{(R)}$. Fix $j,\ell\in C$. By adding and subtracting the unclipped covariances at couplings $\alpha$ and $0$,
\begin{align*}
&\Cov_\alpha\bigl(Z_j^{(R)},Z_\ell^{(R)}\bigr)-\Cov_0\bigl(Z_j^{(R)},Z_\ell^{(R)}\bigr) \\
&\qquad= \Bigl[\Cov_\alpha\bigl(Z_j^{(R)},Z_\ell^{(R)}\bigr)-\Cov_\alpha(X_j,X_\ell)\Bigr]
+ \Bigl[\Cov_\alpha(X_j,X_\ell)-\Cov_0(X_j,X_\ell)\Bigr] \\
&\qquad\phantom{=}+ \Bigl[\Cov_0(X_j,X_\ell)-\Cov_0\bigl(Z_j^{(R)},Z_\ell^{(R)}\bigr)\Bigr].
\end{align*}
The first and third brackets are $O(R^{-1})$ uniformly by Lemma~\ref{lem:clip-bias}. For the middle bracket, if $j=\ell$ then Proposition~\ref{prop:cov-perturb} gives
\[
\abs{\Var_\alpha(X_j)-\Var_0(X_j)}\le C\alpha,
\]
while if $j\neq \ell$ then independence at $\alpha=0$ gives $\Cov_0(X_j,X_\ell)=0$ and Proposition~\ref{prop:cov-perturb} yields
\[
\abs{\Cov_\alpha(X_j,X_\ell)-\Cov_0(X_j,X_\ell)}=\abs{\Cov_\alpha(X_j,X_\ell)}\le C\alpha.
\]
Hence every entry of $\Sigma_C^{(R)}-\Sigma_{C,0}^{(R)}$ is $O(\alpha)+O(R^{-1})$, uniformly over all $(d,B,C)$. Recalling the bound for operator norm given by the Frobenius norm, since the matrix dimension is at most $m$, it holds that
\[
\sup_{|C|\le m}\norm{\Sigma_C^{(R)}-\Sigma_{C,0}^{(R)}}_\op
\le C_m\bigl(\alpha+R^{-1}\bigr).
\]
We take $\alpha_3>0$ small and then $R_3(\alpha)$ large enough so that the right-hand side is at most $\mu_-/(8\beta)$. Weyl's inequality and bound $\lambda_{\min}(\Sigma_{C,0}^{(R)})\ge \mu_-/4\beta$ yield the desired claim with $\kappa:=\mu_-/8\beta$
\end{proof}
The next proposition shows how the theoretical OLS slope computed on any candidate set $C \supset S_i$ of bounded size is close to the rescaled true coefficient vector $\theta_{i,C}$ up to small errors from clipping and binning and that this approximation is tight enough to correctly separate true parents from non-parents:
\begin{proposition}
\label{prop:population-beta}
There exist $\alpha_4>0$, a function $R_4:(0,\alpha_4]\to[1,\infty)$, and a constant $C_y>0$ depending only on $(\beta,\mu_{\pm},w_{\pm},k,m)$ such that for every fixed $0<\alpha\le \alpha_4$, every $R\ge R_4(\alpha)$, and every
\[
0<h\le \min\left\{\frac{w_-}{8C_y}\frac{\alpha}{R}, 1\right\},
\]
one has, for every row $i$ and every candidate set $C\supset S_i$ with $|C|\le m$,
\begin{equation}
\label{eq:population-beta-bound}
\norm{y_{i,C}^{(h,R)} - h\theta_{i,C}}_\infty \le C_y\left(\frac{\alpha h}{R} + R h^2\right).
\end{equation}
In particular,
\begin{equation}
\label{eq:population-beta-threshold}
\min_{j\in S_i} y_{ij,C}^{(h,R)} \ge \frac{3}{4}\theta_- h,
\qquad
\max_{j\in C\setminus S_i} \abs{y_{ij,C}^{(h,R)}} \le \frac14 \theta_- h.
\end{equation}
\end{proposition}

\begin{proof}
Fix $(i,C)$. Recalling the definition of $g_{i,C}^{(h,R)}$, Lemma~\ref{lem:one-bin-expansion} applied coordinatewise gives
\begin{equation}
\label{eq:g-expand-1}
g_{i,C}^{(h,R)} = h\Cov\bigl(Z_C^{(R)}(0),\lambda_i(0)\bigr)+r_{i,C}^{(h,R)},
\qquad
\norm{r_{i,C}^{(h,R)}}_\infty\le C R h^2.
\end{equation}
Now, because $C\supset S_i$, using the definition for $\lambda_i(0)$:
\[
\Cov\bigl(Z_C^{(R)},\lambda_i\bigr)
=
\sum_{\ell\in C}\theta_{i\ell}\Cov\bigl(Z_C^{(R)},X_\ell\bigr).
\]
Writing $X_\ell = Z_\ell^{(R)} + \delta_\ell^{(R)}$ with $\delta_\ell^{(R)}=(X_\ell-R)_+$ gives $\Cov\bigl(Z_C^{(R)},\lambda_i\bigr)=\Sigma_C^{(R)}\theta_{i,C}+q_{i,C}^{(R)}$, where  $q_{i,C}^{(R)}:=\sum_{\ell\in C}\theta_{i\ell}\Cov\bigl(Z_C^{(R)},\delta_\ell^{(R)}\bigr)$, then 
\begin{equation}
    \label{eq:grewritten}
    g_{i,C}^{(h,R)} = h\Sigma_C^{(R)}\theta_{i,C} + h q_{i,C}^{(R)} + r_{i,C}^{(h,R)}.
\end{equation}
To bound $q_{i,C}^{(R)}$, fix $j,\ell\in C$. Since $0\le Z_j^{(R)}\le X_j$ and $\delta_\ell^{(R)}=(X_\ell-R)_+$,
\[
\abs{\Cov\bigl(Z_j^{(R)},\delta_\ell^{(R)}\bigr)}
\le \E\bigl[Z_j^{(R)}\delta_\ell^{(R)}\bigr] + \E[Z_j^{(R)}]\E[\delta_\ell^{(R)}]
\le \E\bigl[X_j\delta_\ell^{(R)}\bigr] + \E[X_j]\E[\delta_\ell^{(R)}].
\]
By Cauchy--Schwarz, Proposition~\ref{prop:moments} with $P=4$ and \eqref{eq:delta-second-moment-main},
\[
\E\bigl[X_j\delta_\ell^{(R)}\bigr]
\le \sqrt{\E[X_j^2]\E[(\delta_\ell^{(R)})^2]}
\le \frac{C}{R},
\qquad
\E[X_j]\E[\delta_\ell^{(R)}]\le \frac{C}{R}.
\]
Hence each coordinate of $\Cov(Z_C^{(R)},\delta_\ell^{(R)})$ is $O(R^{-1})$, uniformly in $(d,B,i,C,\ell)$. Since $|C|\le m$ and $\theta_{i\ell}\le \alpha w_+$, 
\[ \norm{q_{i,C}^{(R)}}_\infty\le C \alpha/R. 
\]
Now, rewriting \eqref{eq:population-slope} using \eqref{eq:grewritten}, gives
\[
y_{i,C}^{(h,R)} - h\theta_{i,C}
=
\bigl(\Sigma_C^{(R)}\bigr)^{-1}\Bigl(h q_{i,C}^{(R)} + r_{i,C}^{(h,R)}\Bigr).
\]
Therefore, using Lemma~\ref{lem:Gram-nondeg} implication $\norm{(\Sigma_C^{(R)})^{-1}}_\op\le \kappa^{-1}$ and the obtained bound for $\norm{q_{i,C}^{(R)}}_\infty$ gives
\[
\norm{y_{i,C}^{(h,R)} - h\theta_{i,C}}_\infty
\le C_y\left(\frac{\alpha h}{R} + R h^2\right)
\]
for a constant $C_y$ depending only on the model parameters. Choose $\alpha_4\le \alpha_3$ and then choose $R_4(\alpha)\ge \max\left\{R_3(\alpha), 8C_y/w_-\right\}$ and $0<h\le \min\{1,\theta_-/(8C_yR)\}$. If $R\ge R_4(\alpha)$ then
\[
C_y\frac{\alpha h}{R}\le \frac18\theta_-h,
\qquad
C_y R h^2\le \frac18\theta_- h,
\]
so the right-hand side of \eqref{eq:population-beta-bound} is at most $h \theta_-/4$. Since $\theta_{ij}\ge \theta_-$ on $S_i$ and $\theta_{ij}=0$ on $C\setminus S_i$, \eqref{eq:population-beta-threshold} follows.
\end{proof}

\subsection{Bounded-block concentration via Poisson clusters}
\label{appendix:4}
First of all, set $0<h\le h_0:=1$ and recall $\gamma<1$. The observables we are focusing on are the following four:
\begin{equation}
    \label{eq:explicit-family}
    Z_{j,r}^{(R)},
    \qquad
    Y_{i,r}^{(h)},
    \qquad
    Z_{j,r}^{(R)}Y_{i,r}^{(h)},
    \qquad
Z_{a,r}^{(R)}Z_{b,r}^{(R)}
\end{equation}
Now, we introduce the terminology used in the context of the cluster representation of Hawkes processes:
\begin{itemize}
    \item\textbf{Root}: these are events that arrive according to a independent Poisson process with background intensity $\mu_v$. They are ``exogenous'' because they are not caused by previous events;
    \item\textbf{Type-$j$}: since the process $X$ has $d$ components indexed by $j=1,\dots, d$, each event in the system carries a label $j$ indicating which component it belongs to. This label is called its ``type'';
    \item \textbf{Offspring}: each root can trigger direct offspring, which in turn can trigger their own offspring (called multiple generations). Thus, each cluster $\mathcal C$ has an independent root and each event of type-$j$ produces type-$i$ children according to an independent Poisson point process on $(0,\infty)$ with intensity $\theta_{ij} \exp\{-\beta t\}\,dt$. Therefore the mean offspring matrix is $K:=\Theta/\beta$ and consequently $\|K\|_{\infty}\le \gamma < 1$;
    \item\textbf{Progeny}: while offspring refers only to the events that are immediately triggered by one specific root, the progeny of one root includes every single event in the cluster, across all generations;
    \item\textbf{Age}: in the cluster representation, each cluster has a root event that is ``born'' at some time $t$, namely the time at which the root arrives. The age of any event in the cluster is the time that passed since the root was born. Thus if an event occurs at time $s$, its age is $s - t$. Hence, ``age'' is not absolute time, but rather time measured relative to the birth time of the root of the cluster to which the event belongs.
\end{itemize}
Now, set $a>1$ s.t. $a\gamma<1$, fix a finite coordinate set $J\subset[d]$ with $|J|\le 2$ and denote the type-$j$ events at time $t$ in a cluster $\mathcal C$ as $(j,t)\in\mathcal C$. Define the type weights $q_v^{(a)}$ as
\begin{equation}
\label{eq:q-weights-main}
q_v^{(a)}:=\sum_{\ell=0}^{\infty} a^{\ell}\sum_{j\in J}(K^{\ell})_{jv},
\qquad v=1,\dots,d.
\end{equation}
Note that $q_v^{(a)}=\ind\{v\in J\}+a\sum_{i=1}^d K_{iv}q_i^{(a)}$, thus $q_v^{(a)}\ge 1$ for every $v\in J$. Define also:
\begin{equation}
\label{eq:ah-bounds-main}
a_h(t):=\sum_{r=0}^{\infty} e^{-\beta(rh-t)}\textbf{1}\{t\le rh\}\le C_h\times
\begin{cases}
e^{\beta t}, & t<0,\\
1, & t\ge 0,
\end{cases}
\qquad\quad\quad
C_h:=\frac{e^{\beta h}}{1-e^{-\beta h}},
\end{equation}
where the inequality follows from a simple calculation. Lastly, the weighted local cluster size is:
\begin{equation}
\label{eq:H-cluster-main}
H_J(\mathcal C):=\sum_{(j,t)\in \mathcal C,\, j\in J} a_h(t)+\sum_{(j,t)\in \mathcal C,\, j\in J} \textbf{1}\{t\ge 0\}.
\end{equation}
Because $q_j^{(a)}\ge 1$ on $J$ and $a_h(t)\le C_h$, defining $W_J^{(a)}(\mathcal C):=\sum_{(j,t)\in \mathcal C} q_j^{(a)}$, it holds
\begin{equation}
\label{eq:H-vs-W-main}
H_J(\mathcal C)
\le (C_h+1)\sum_{(j,t)\in \mathcal C,\,j\in J}1
\le (C_h+1)W_J^{(a)}(\mathcal C).
\end{equation}
The following lemma shows that the exponential moment of the weighted cluster size is uniformly bounded, regardless of the dimension $d$ and the network configuration.
\begin{lemma}
\label{lem:A}
There exist constants $\eta_A>0$ and $C_A>0$, depending only on $(h,R,\beta,\gamma)$, such that for every $d$, every $B\in\mathcal B_{d,k}$, every $J\subset[d]$ with $|J|\le 2$, and every root type $v$,
\begin{equation}
\label{eq:lemmaA-main}
    \mathbb E_v\big[e^{\eta_A H_J(\mathcal C)}\big]\le C_A .
\end{equation}
\end{lemma}
\begin{proof}
Choose $a\in(1,\gamma^{-1})$ and abbreviate $q_v:=q_v^{(a)}$. By \eqref{eq:H-vs-W-main}, it suffices to prove the Lemma for $W_J^{(a)}(\mathcal C)$. First of all note that: 
\begin{equation}
    \label{eq:Qqv}
    0\le q_v^{(a)}\le \sum_{\ell=0}^{\infty} a^{\ell}\sum_{j\in J}\|K\|_{\infty}^\ell\le |J|\sum_{\ell=0}^{\infty} (a\gamma)^{\ell}= \frac{|J|}{1-a\gamma}\le \frac{2}{1-a\gamma}=:Q.
\end{equation}
Let $W_{J,v}^{(a)}$ denote the weighted total progeny when the root has type $v$, and let $W_{J,v}^{(a,m)}$ be the same quantity truncated at generation depth at most $m\ge 0$. Similarly, if root $v$ has $M_i$ type-$i$ offspring (note that $M_i$ is Poisson with parameter $K_{iv}$), denote $W_{J,i,k}^{(a,m)}$ the truncated weighted total progeny generated by type-$i$ offspring. Then defining $F_v^{(m)}(\eta):=\mathbb E_v\big[e^{\eta W_{J,v}^{(a,m)}}\big]$, 
\begin{align*}
    F_v^{(m+1)}(\eta)={ }&  \mathbb E_v\left[ e^{\eta \left(q_v+\sum_{i=1}^d\sum_{k=1}^{M_i}W_{J,i,k}^{(a,m)} \right)} \right]=e^{\eta q_v}\prod_{i=1}^d\mathbb E_v\left[ \mathbb E\left[ \prod_{k=1}^{M_i}e^{\eta W_{J,i,k}^{(a,m)}}\biggl|M_i=m_i  \right] \right]=\\
    ={ }& e^{\eta q_v}\prod_{i=1}^d\mathbb E_v\left[ \left(F^{(m)}_i(\eta)\right)^{M_i}\right] =  e^{\eta q_v}\prod_{i=1}^d e^{-K_{iv}}\sum_{m_i=0}^{\infty} \frac{\left(F_i^{(m)} K_{iv}\right)^{m_i}}{m_i!}=\\
    ={ }& \exp\left(\eta q_v + \sum_{i=1}^d K_{iv}\bigl(F_i^{(m)}(\eta)-1\bigr)\right),
\end{align*}
where $F_v^{(0)}(\eta)=e^{\eta q_v}$. Set $A:=2/(1-a^{-1})>1$, then $A(1-a^{-1})>1$ and $\exists x_0>0$ depending only on $a$, such that
\begin{equation}
\label{eq:Ax-main}
e^x-1\le A\frac{x}{1+A/a}
\qquad\text{for every }0\le x\le x_0.
\end{equation}
Now choose $\eta_A>0$ such that $\eta_A Q(1+A/a)\le x_0$. This choice, for $m=0$ and \eqref{eq:Qqv}, implies $\eta_A q_v\le x_0$ and therefore $F_v^{(0)}(\eta_A)-1=e^{\eta_A q_v}-1\le A\eta_A q_v$. Assume that the same inequality holds for a fixed $m>0$ and $\forall v$ as inductive step; using the obtained expression for $F_v^{(m+1)}(\eta)$,  $\sum_{i=1}^d K_{iv}q_i\le  q_v/a$ and the same reasoning used for the zero-step, it holds
\[
F_v^{(m+1)}(\eta_A)
\le
\exp\!\left(\eta_A q_v + A\eta_A\sum_{i=1}^d K_{iv}q_i\right)
\le
\exp\!\left(\eta_A q_v\Bigl(1+\frac{A}{a}\Bigr)\right)\le 1+A\eta_Aq_v,
\]
concluding the induction. Since $W_{J,v}^{(a,m)}\uparrow W_{J,v}^{(a)}$ almost surely as $m\to\infty$, monotone convergence yields $\mathbb E_v\big[e^{\eta_A W_J^{(a)}(\mathcal C)}\big] \le 1+A\eta_A Q =:C_A'<\infty$,
uniformly in $(d,B,J,v)$. Combining with \eqref{eq:H-vs-W-main} and shrinking $\eta_A$ by the factor $(C_h+1)$ and renaming it proves \eqref{eq:lemmaA-main}.
\end{proof}
The proof of Lemma~\ref{lem:A} gives a little more than the uniform bound stated there: monotone convergence in previous proof's inductive step yields
\begin{equation}
\label{eq:exp-minus-one-q-main}
\E_v\big[e^{\eta_A W_{J,v}^{(a)}}-1\big]\le A\eta_A q_v^{(a)}.
\end{equation}
Moreover, by the row-sum bound,
\begin{equation}
\label{eq:q-sum-main}
\sum_{v=1}^d \mu_v q_v^{(a)}
\le \mu_+\sum_{\ell=0}^{\infty} a^{\ell}\sum_{j\in J}\sum_{v=1}^d (K^{\ell})_{jv} \le  \mu_+\sum_{\ell=0}^{\infty} a^{\ell}\sum_{j\in J}\|K\|_{\infty}^\ell\le 2\mu_+\sum_{\ell=0}^{\infty}(a\gamma)^{\ell}
\le C_q.
\end{equation}
Consequently, for every integer $m\ge 1$ there exists $\eta_m>0$ such that
\begin{equation}
\label{eq:agg-W-moment-main}
\sum_{v=1}^d \mu_v\,\E_v\big[(W_{J,v}^{(a)})^m e^{\eta_m W_{J,v}^{(a)}}\big]\le C_m,
\end{equation}
where $C_m>0$ depends only on $(m,\beta,\mu_+,k,w_+,\gamma)$.
Indeed, choose $\eta_m>0$ such that $2\eta_m\le \eta_A$ and use the pointwise bound $x^m e^{\eta_m x}\le C_{m,\eta_m}(e^{\eta_A x}-1)$, valid for every $m\ge 1$, together with \eqref{eq:exp-minus-one-q-main} and \eqref{eq:q-sum-main}. For $m=0$ one simply uses \eqref{eq:lemmaA-main} directly. We now use this result to prove the following lemma, which controls how long a cluster keeps producing events of the types we are interested in (for example $J$). Specifically, it shows that the probability of a cluster still having type-$J$ events at a late age $u$ decays exponentially in $u$, uniformly over all dimensions and all network configurations.
\begin{lemma}
\label{lem:relevant-lifetime-tail}
Fix $J\subset[d]$ with $|J|\le 2$. For one cluster started from a root of type $v$ at time $0$, let
\[
L_J:=\sup\{u\ge 0:\ \text{there exists a type-}j\in J\text{ event in the full cluster at age }u\},
\]
with the convention $\sup\varnothing=0$. Thus the root itself is counted at age $0$ when $v\in J$. Then there exist constants $a\in(1,\gamma^{-1})$, $c_L>0$, and $C_L>0$, depending only on $(\beta,k,w_+,\gamma)$, such that for every $u\ge 0$,
\begin{align}
\Pbb_v(L_J\ge u)&\le q_v^{(a)} e^{-c_L u}, \label{eq:lifetime-tail-pointwise-main}\\
\sum_{v=1}^d \mu_v\Pbb_v(L_J\ge u)&\le C_L e^{-c_L u}. \label{eq:lifetime-tail-agg-main}
\end{align}
\end{lemma}
\begin{proof}
Set $c_L:=\beta(1-a^{-1})>0$. Let $N_J([u,\infty))$ denote the number of type-$J$ events in the full cluster whose ages lie in $[u,\infty)$, again counting the root when applicable. Then
\[
\Pbb_v(L_J\ge u)=\mathbb E _v[\textbf{1}\{L_J\geq u\}]\le \mathbb E_v\big[N_J([u,\infty))\big].
\]
For $u=0$, the root contributes exactly when $v\in J$, so
\[
\E_v\big[N_J([0,\infty))\big]
=\ind\{v\in J\}+\sum_{\ell\ge 1}\sum_{j\in J}(K^{\ell})_{jv}
\le q_v^{(a)},
\]
because $a^\ell\ge 1$ and $q_v^{(a)}$ contains the $\ell=0$ root term explicitly. For $u>0$, the root does not contribute, and a type-$j\in J$ offspring in generation $\ell\ge 1$ has age distributed as $S_\ell:=E_1+\cdots+E_\ell$ with $E_m\stackrel{\mathrm{i.i.d.}}{\sim}\mathrm{Exp}(\beta)$ and expected multiplicity $(K^{\ell})_{jv}$. Hence
\[
\E_v\big[N_J([u,\infty))\big]
=\sum_{\ell\ge 1}\sum_{j\in J}(K^{\ell})_{jv}\Pbb(S_\ell\ge u).
\]
By Chernoff's bound,
\[
\Pbb(S_\ell\ge u)
\le e^{-c_L u}\Bigl(\frac{\beta}{\beta-c_L}\Bigr)^{\ell}
=e^{-c_L u} a^{\ell}.
\]
Therefore, for $u>0$,
\[
\Pbb_v(L_J\ge u)
\le e^{-c_L u}\sum_{\ell\ge 1}a^{\ell}\sum_{j\in J}(K^{\ell})_{jv}
\le q_v^{(a)} e^{-c_L u}.
\]
Together with the $u=0$ case this proves \eqref{eq:lifetime-tail-pointwise-main}. Summing against $\mu_v$ and using \eqref{eq:q-sum-main} gives \eqref{eq:lifetime-tail-agg-main}.
\end{proof}
We are now ready to prove the next lemma, which combines (i) the exponential decay in $u$ of the probability that a cluster contains a type-$J$ event at age larger than $u$ (Lemma~\ref{lem:relevant-lifetime-tail}) and (ii) the uniform bound on the exponential moments of the total cluster size (equation \eqref{eq:agg-W-moment-main}). These two ingredients give finite moment bounds on the local cluster counts that do not depend on the dimension $d$. For real $u$, let $N_{J,u}^{\rm tail}$ be the number of type-$J$ events in the full cluster with ages in $[u,\infty)$, and let $N_{J,u}^{\rm win}$ be the number of type-$J$ events in the truncated unit window $[u,u+1)\cap[0,\infty)$.
\begin{lemma}
\label{lem:local-count-moments}
Fix $J\subset[d]$ with $|J|\le 2$. Then there exist constants $\eta_C>0$, $c_{cl}>0$, and $C_{cl}>0$, depending only on $(\beta,\mu_+,k,w_+,\gamma)$, such that for every $u\ge 0$ and every $s\ge -1$,
\begin{align}
\sum_{v=1}^d \mu_v\,\mathbb E_v\big[(N_{J,u}^{\rm tail})^2 e^{\eta_C N_{J,u}^{\rm tail}}\big]&\le C_{cl} e^{-c_{cl} u}, \quad \sum_{v=1}^d \mu_v\,\mathbb E_v\big[(N_{J,s}^{\rm win})^2 e^{\eta_C N_{J,s}^{\rm win}}\big]\le C_{cl} e^{-c_{cl} (s)_+}. 
\label{eq:tail-count-moment-main} 
\end{align}
\end{lemma}
\begin{proof}
Let $N_J^{\rm tot}$ denote the total number of type-$J$ events in the whole cluster, root included when $v\in J$. Since $N_J^{\rm tot}=\sum_{(j,t)\in\mathcal C,j\in J}1\le\sum_{(j,t)\in\mathcal C}q_j^{(a)}= W_{J,v}^{(a)}$, \eqref{eq:agg-W-moment-main} with $m=4$ yields:
\begin{equation}
\label{eq:agg-total-count-fourth-main}
\sum_{v=1}^d \mu_v\,\E_v\big[(N_J^{\rm tot})^4 e^{2\eta_C N_J^{\rm tot}}\big]\le C
\end{equation}
for some sufficiently small $\eta_C>0$.

Now $N_{J,u}^{\rm tail}\le N_J^{\rm tot}$ and the event $\{N_{J,u}^{\rm tail}>0\}\subset \{L_J\ge u\}$. Hence, by Cauchy--Schwarz,
\begin{align*}
\sum_{v=1}^d \mu_v\,\mathbb E_v\big[(N_{J,u}^{\rm tail})^2 e^{\eta_C N_{J,u}^{\rm tail}}\big]
&\le \sum_{v=1}^d \mu_v\,\mathbb E_v\big[(N_J^{\rm tot})^2 e^{\eta_C N_J^{\rm tot}}\,\textbf{1}\{L_J\ge u\}\big]\\
&\le \left(\sum_{v=1}^d \mu_v\,\mathbb E_v\big[(N_J^{\rm tot})^4 e^{2\eta_C N_J^{\rm tot}}\big]\right)^{1/2}
\left(\sum_{v=1}^d \mu_v\Pbb_v(L_J\ge u)\right)^{1/2}\\
&\le C e^{-c_{cl} u},
\end{align*}
using \eqref{eq:lifetime-tail-agg-main} and after decreasing $c_{cl}>0$ if necessary. For the window counts, first take $s\ge 0$. Then $N_{J,s}^{\rm win}\le N_J^{\rm tot}$ and $\{N_{J,s}^{\rm win}>0\}\subset\{L_J\ge s\}$, so the same argument gives
\[
\sum_{v=1}^d \mu_v\,\E_v\big[(N_{J,s}^{\rm win})^2 e^{\eta_C N_{J,s}^{\rm win}}\big]\le C e^{-c_{cl} s}.
\]
If $-1\le s<0$, then $N_{J,s}^{\rm win}\le N_{J,0}^{\rm win}$ because $[s,s+1)\cap[0,\infty)=[0,s+1)\subset[0,1)$, so the same bound follows from the already proved case $s=0$. 
\end{proof}
From now on, we consider independent root time-blocks: for every integer $q\in\mathbb Z$, let $I_q:=[qh,(q+1)h)$ and $U_q:=\text{all clusters rooted in }I_q$. The sequence $(U_q)_{q\in\mathbb Z}$ is independent (note that the observed bin variables $(\xi_0,\dots,\xi_{n-1})$, where $\xi_r$ is any in \eqref{eq:explicit-family}, are \emph{not} independent) and consequently define the filtration $\mathcal F_q:=\sigma(U_p:\ p\le q)$. Thus, for any statistic from \eqref{eq:explicit-family}, the random variable
\begin{equation}
    \label{eq:s_nstatistic}
    S_n:=\sum_{r=0}^{n-1}(\xi_r-\mathbb E\xi_0)
\end{equation}
is measurable with respect to $(U_q)_{q\le n-1}$. For one root born at $x=(v,t)$, let $\mathcal C_x$ denote its full cluster, including the root itself. For $r=0,\dots,n-1$, the family-specific one-cluster envelopes are:
\begin{itemize}
    \item for a type-$j$, $B_j^x(r):=\sum_{\substack{e\in \mathcal C_x, \mathrm{type}(e)=j,\, s_e\le rh}} e^{-\beta(rh-s_e)}$,
    \item for a type-$i$, $A_i^x(r):=\textbf{1}\{\mathcal C_x \text{ contributes at least one type-}i\text{ event in }(rh,(r+1)h]\}.$
\end{itemize}
Set also:
\begin{align}
\Gamma_{Z_j}(x)   &\coloneqq \sum_{r=0}^{n-1} B_j^x(r), & \Gamma_{Y_i}(x)   &\coloneqq \sum_{r=0}^{n-1} A_i^x(r) \label{eq:Gamma-Z-Y-main} \\
\Gamma_{Z_jY_i}(x) &\coloneqq 2\Gamma_{Z_j}(x)+R\Gamma_{Y_i}(x), & \Gamma_{Z_aZ_b}(x) &\coloneqq 2R\Gamma_{Z_a}(x)+2R\Gamma_{Z_b}(x). \label{eq:Gamma-ZY-ZZ-main}
\end{align}
\begin{lemma}
\label{lem:one-cluster-add-one}
Fix one sequence $(\xi_r)$ from \eqref{eq:explicit-family}. Let $x=(v,t)$ and let $\mathcal C_x$ be one cluster rooted at $x$. For every fixed configuration of all other blocks, adding $\mathcal C_x$ changes the centered statistic by at most
\[
\bigl|S_n(\mathcal C_x)-S_n(0)\bigr|\le \Gamma_\xi(x,\mathcal C_x),
\]
where $S_n(0)$ denotes the value obtained when that one cluster is absent.
\end{lemma}
\begin{proof}
First of all, note that the centering is deterministic and cancels when taking the difference, so we will not consider it. We can prove the claim for each of the element from \eqref{eq:explicit-family}:
\begin{itemize}
\item \emph{$\xi=Z_{j,r}^{(R)}$.} $\psi_R$ is 1-Lipschitz, so adding one type-$j$ event at time $s_e \le rh$ changes $X_j(rh)$ by exactly $e^{-\beta(rh-s_e)}$, and after clipping the change is at most this. Thus
\begin{align*}
    \left|\Delta_{\mathcal C_x}\sum_{r=0}^{n-1}Z_{j,r}^{(R)}\right| \le \sum_{e\in\mathcal C_x,\ \mathrm{type}(e)=j} \sum_{r=0}^{n-1} e^{-\beta(rh-s_e)}\cdot \textbf{1}\{s_e\le rh\} = \sum_{r=0}^{n-1} B_j^x(r)= \Gamma_{Z_j}(x,\mathcal C_x). 
\end{align*}
\item \emph{$\xi=Y_{i,r}^{(h)}$.} Note that $A_i^x(r)=0$ iff $\mathcal C_x$ has no type-$i$ offspring in $(rh,(r+1)h]$. Therefore removing $\mathcal C_x$ gives $|\Delta _{\mathcal C_x}Y_{i,r}^{(h)}|=0\le A_i^x(r)$. Now, note that $A_i^x(r)=1$ in two cases: the first one, if $\mathcal C_x$ is the only cluster having type-$i$ offspring in $(rh,(r+1)h]$. Therefore removing $\mathcal C_x$ gives $|\Delta _{\mathcal C_x}Y_{i,r}^{(h)}|=1\le A_i^x(r)$. The second one, if  $\mathcal C_x$ has type-$i$ offspring in $(rh,(r+1)h]$, but there exists $\mathcal C_y$ for some $y=(v',t')$ with type-$i$ offspring in $(rh,(r+1)h]$. Then, removing $\mathcal C_x$ has no impact and therefore $|\Delta _{\mathcal C_x}Y_{i,r}^{(h)}|=0\le A_i^x(r)$. Then the result follows immediately. 
\item \emph{$\xi=Z_{j,r}^{(R)}Y_{i,r}^{(h)}$.} Use the inequality $|(z+\delta)(y+a)-zy|\le 2|\delta|+R|a|$ for $y\ge0$, $y+a\le1$, $z\ge0$ and $z+\delta\le R$. Then $z+\delta$ plays the role of $Z_{j,r}^{(R)}$ when considering $\mathcal C_x$, while $z$ plays that role when it is not considered. Similarly, $y+a$  plays the role of $Y_{i,r}^{(h)}$ when considering $\mathcal C_x$, while $y$ plays that role when it is not considered. Therefore the result follows taking $\delta$ as $\Delta_{\mathcal C_x}Z_{j,r}^{(R)}$ and $a$ as $\Delta _{\mathcal C_x}Y_{i,r}^{(h)}$, and using the previous results.
\item \emph{$\xi=Z_{a,r}^{(R)}Z_{b,r}^{(R)}$.} Same as the previous step, but using $|(z_a+\delta_a)(z_b+\delta_b)-z_az_b|\le 2R|\delta_a| + 2R|\delta_b|$.
\end{itemize}
\end{proof}
Now we show an auxiliary lemma useful for proving the next one about conditional mgf of one block increment: 
\begin{lemma}
\label{lem:poisson-bd-external}
Let $(E,\mathcal E)$ be a standard Borel space, let $\eta$ be a Poisson point process on $E$ with intensity measure $\nu$, and let $F=F(\eta)$ be a measurable functional with $\E[F]=0$. Assume there exists a deterministic measurable function $g:E\to[0,\infty)$ such that for every configuration $\eta$ and every point $z\in E$,
\[
|D_zF(\eta)|:=|F(\eta+\delta_z)-F(\eta)|\le g(z).
\]
Assume moreover that for some $\eta_*>0$,
\[
\int_E g(z)^2 e^{\eta_* g(z)}\,\nu(dz)<\infty.
\]
Then there exist a constant $\lambda_\mathrm{P}>0$, depending only on $\eta_*$, such that
\begin{equation}
\label{eq:poisson-bd-external-main}
\E\big[e^{\lambda F}\big]\le \exp\!\left(\lambda^2 \int_E g(z)^2 e^{\eta_* g(z)}\,\nu(dz)\right)
\qquad (|\lambda|\le \lambda_\mathrm{P}).
\end{equation}
\end{lemma}

\begin{proof}
For $K\ge 1$, let $F^{(K)}:=(-K)\vee(F\wedge K)-\mathbb E[(-K)\vee(F\wedge K)]$.
The clipping map is $1$-Lipschitz, so $|D_zF^{(K)}|\le g(z)$. It therefore suffices to prove the claim for bounded centered functionals and then let $K\uparrow\infty$. \\
Assume for the moment that $F$ is bounded. Take $B\subseteq[0,1]$ with $card(B)=card (E)$ and s.t. $B\in \mathcal B([0,1])$. Since $([0,1],\mathcal B([0,1]))$ is standard Borel, then Corollary (13.4) in \cite{kechris2012classical} holds and $(B, \mathcal B([0,1])|_{B})$ is standard Borel. By assumption $(E,\mathcal E)$ is standard Borel, then by Theorem (15.6) in \cite{kechris2012classical} there exists an isomorphism $\phi:E\to B\subseteq [0,1]$ such that $\phi$ and $\phi^{-1}$ are Borel measurable. Being $\eta$ a Poisson point process with intensity $\nu$ and being $\phi$ Borel measurable, by Theorem (5.1) in \cite{last2018lectures}, $\phi_{\#}\eta$ is again Poisson with intensity $\phi_{\#}\nu$ and distributional properties are preserved. 
The induced functional $\widetilde F(\phi_\#\eta):=F(\eta)$ is well defined because $\phi$ is a bijection and $\E[\tilde F (\phi_{\#}\eta)]=\E[F(\eta)]=0$. Moreover, for any $\lambda$ it holds $\E[e^{\tilde F(\phi_{\#}\eta)}]=\E[e^{\lambda F(\eta)}]$, thus any mgf bound for $\tilde F(\phi_{\#}\eta)$ is exactly equivalent for $F(\eta)$. From the definition of pushforward:
\[
|D_w\tilde F| = |\tilde F(\phi_{\#}\eta+\delta_w)-\tilde F(\phi_{\#}\eta)| = |F(\eta+\delta_z)-F(\eta)|\le g(z)=(g \,\circ\,\phi^{-1})(w)=\tilde g(w)
\]
so the add-one bound is preserved by redefining $g$. Lastly note that:
\[
\int_B \tilde g(w)^2e^{\eta_*\tilde g(w)} (\phi_{\#}\nu)(dw) = \int_E \tilde g(\phi(z))^2 e^{\eta_*\tilde g(\phi(z))} \nu(dz) = \int_E g(z)^2 e^{\eta_* g(z)} \nu(dz)  <\infty .
\]
So it is possible to prove Lemma \ref{lem:poisson-bd-external} over the transported system $(B, \mathcal B([0,1])|_B, \phi_{\#}\eta,\tilde F, \tilde g, \phi_{\#}\nu)$, renaming it as the original one $(E,\mathcal E, \eta, F, g,\nu)$. \\
Now, realize $\eta$ as the time-$1$ section of a marked Poisson process $\bar\eta$ on $E\times[0,1]$ with intensity $\nu(dz)\,ds$. Write
\[
\eta_s:=\bar\eta(\cdot\times[0,s]),
\qquad
\eta^s:=\bar\eta(\cdot\times(s,1]),
\qquad
\mathcal G_s:=\sigma\!\bigl(\bar\eta|_{E\times[0,s]}\bigr),
\]
and define the bounded Doob martingale $M_s:=\E[F(\eta_s+\eta^s)\mid \mathcal G_s]=\E[F(\eta)\mid \mathcal G_s]$, then $M_0=\E[F]=0$ and $M_1=F$. If a marked point $(z,s)$ appears at time $s$, the jump of $M$ is
\[
J_s(z):=
\mathbb E\Big[F(\eta_{s-}+\delta_z+\eta^s)-F(\eta_{s-}+\eta^s)\,\Big|\,\mathcal G_{s-},z\Big],
\]
hence $|J_s(z)|\le g(z)$ and, because $F$ is bounded, the martingale $M$ is square-integrable. For the completed natural filtration $(\mathcal G_s)$ of the Poisson random measure $N(ds,dz):=\bar\eta(ds,dz)$, the Theorem (6.7) II in \cite{ikeda2014stochastic} yields a predictable integrand $H\in L^2(\Omega\times(0,1]\times E, d\mathbb P\,ds\,\nu(dz))$ such that
\[
M_t=\int_{(0,t]\times E} H_s(z)\,\widetilde N(ds,dz),
\qquad
\widetilde N(ds,dz):=N(ds,dz)-\nu(dz)\,ds.
\]
At a jump time $s$ of $N$ with mark $z$, from Doob martingale definition
\[
\Delta M_s=\mathbb E\Big[F(\eta_{s-}+\delta_z+\eta^s)-(\eta_{s-}+\eta^s)\,\Big|\,\mathcal G_{s-},z\Big]=J_s(z).
\]
At same time, the stochastic integral representation has jump size $H_s(z)$. By uniqueness of the predictable representation, $H_s(z)=J_s(z)$ for $d\mathbb P\,ds\,\nu(dz)$-a.e. $(\omega,s,z)$, and therefore
\[
M_t=\int_{(0,t]\times E} J_s(z)\,\widetilde N(ds,dz).
\]
Apply It\^o formula to $e^{\lambda M_t}$ (Proposition 8.19 \cite{cont2003financial}), take expectations and use $|J_s(z)|\le g(z)$: 
\[
\E[e^{\lambda M_t}]
\le
1+\int_0^t \E[e^{\lambda M_{s-}}]
\left(\int_E \bigl(e^{|\lambda|g(z)}-1-|\lambda|g(z)\bigr)\,\nu(dz)\right) ds.
\]
Note that for $|\lambda|\le \eta_*/2$,
\[
0\le e^{|\lambda|g(z)}-1-|\lambda|g(z)
\le \frac{\lambda^2}{2} g(z)^2 e^{|\lambda|g(z)}
\le \frac{\lambda^2}{2} g(z)^2 e^{\eta_* g(z)},
\]
so the compensator term is finite by the integrability assumption $\int_E g(z)^2 e^{\eta_* g(z)}\,\nu(dz)<\infty$. Since $F$ is bounded, also $|M_t|\le \|F\|_\infty$ almost surely for every $t$, hence $0\le \E[e^{\lambda M_t}]\le e^{|\lambda|\|F\|_\infty}<\infty$. Thus the function $f(t):=\E[e^{\lambda M_t}]$ is finite on $[0,1]$, and Gronwall's inequality applies to give
\[
\mathbb E[e^{\lambda F}]
\le
\exp\!\left(\int_E \bigl(e^{|\lambda|g(z)}-1-|\lambda|g(z)\bigr)\,\nu(dz)\right)
\qquad (\lambda\in\mathbb R).
\]
Now choose $\lambda_\mathrm{P}:=\eta_*/2$. For $y\ge 0$, again $e^y-1-y\le y^2 e^y$, so for $|\lambda|\le \lambda_\mathrm{P}$,
\[
e^{|\lambda|g(z)}-1-|\lambda|g(z)
\le \lambda^2 g(z)^2 e^{|\lambda|g(z)}
\le \lambda^2 g(z)^2 e^{\eta_* g(z)}.
\]
Substituting this into the previous display gives \eqref{eq:poisson-bd-external-main} for bounded centered $F$. Finally apply the bounded case to $F^{(K)}$. Because $F^{(K)}\to F$ almost surely and in $L^1$, Fatou's lemma for $|\lambda|\le \lambda_\mathrm{P}$ yields
\[
\E[e^{\lambda F}]
\le
\liminf_{K\to\infty}\E[e^{\lambda F^{(K)}}]
\le
\exp\!\left(\lambda^2 \int_E g(z)^2 e^{\eta_* g(z)}\,\nu(dz)\right).
\]
\end{proof}
\begin{lemma}
\label{lem:block-conditional-mgf}
Fix one family $(\xi_r)$ from \eqref{eq:explicit-family}. For a root birth at $x=(v,t)$, define
\[
b_\xi(v,t):=\E\Big[\Gamma_\xi(x)^2 e^{\eta_1\Gamma_\xi(x)}\Big],
\]
where $\eta_1=\eta_1(h,R)>0$ is chosen so that all these expectations are finite. Let
\[
V_q(\xi):=\int_{I_q}\sum_{v=1}^d \mu_v\, b_\xi(v,t)\,dt,
\qquad I_q=[qh,(q+1)h),\ q\in\mathbb Z.
\]
Then there exist a constant $c_\lambda>0$ depending only on $(\beta,\mu_+,k,w_+,\gamma,h_0)$, such that for every finite $q\le n-1$, the Doob martingale difference $D_q:=\E[S_n\mid \mathcal F_q]-\E[S_n\mid \mathcal F_{q-1}]$ satisfies
\begin{equation}
\label{eq:block-conditional-mgf-main}
\E[e^{\lambda D_q}\mid \mathcal F_{q-1}]\le \exp\big(\lambda^2 V_q(\xi)\big)
\qquad (|\lambda|\le c_\lambda h/R).
\end{equation}
\end{lemma}
\begin{proof}
Fix one finite block $q\le n-1$. Conditional on $\mathcal F_{q-1}$, the random block $U_q$ is an independent marked Poisson point process, its state space is the space of marked triples $(v,t,\mathcal C)$, where $v\in[d]$, $t\in I_q$, and $\mathcal C$ is a full cluster rooted at $(v,t)$. By Marking Theorem (see \cite{last2018lectures}), the intensity measure is $\nu_q(dv\,dt\,d\mathcal C)=\mu_v\,dt\,\Pbb_{v,t}(d\mathcal C)$. Define the centered block functional 
\[
F_q(U_q):=\mathbb E[S_n\mid \mathcal F_{q-1},U_q]-\mathbb E[S_n\mid \mathcal F_{q-1}].
\]
Since $\mathcal F_q=\sigma(\mathcal F_{q-1},U_q)$, we have $F_q(U_q)=D_q$. Now fix one marked cluster $z=(x,\mathcal C)$ with $x=(v,t)\in[d]\times I_q$. Let $\eta^+$ denote the future root blocks $(U_{q+1},\dots,U_{n-1})$. 
For every fixed realization of $\mathcal F_{q-1}$, every block configuration $U=(\dots,U_{q-1},U_q,\eta_+)$, and every realization of $\eta^+$, Lemma~\ref{lem:one-cluster-add-one} gives
\[
\bigl|S_n(U+\delta_z) - S_n(U)\bigr|\le \Gamma_\xi(x,\mathcal C).
\]
Taking conditional expectation over $\eta^+$ and using Jensen therefore yields
\[
\bigl|F_q(U+\delta_z)-F_q(U)\bigr|
=\Bigl|\mathbb E\bigl[S_n(U+\delta_z) - S_n(U)\mid \mathcal F_{q-1},U_q\bigr]\Bigr|
\le \Gamma_\xi(x,\mathcal C).
\]
Hence Lemma~\ref{lem:poisson-bd-external}, conditional to $\mathcal F_{q-1}$, applied with $g(z)=\Gamma_{\xi}(z)$ and $\eta_*=\eta_1(h,R)$, gives:
\[
\E[e^{\lambda D_q}\mid \mathcal F_{q-1}]\le \exp\!\left(\lambda^2 \int g(z)^2 e^{\eta_1 g(z)}\,\nu_q(dz)\right)
\qquad (|\lambda|\le \lambda_\mathrm{P}).
\]
By the definition of $\nu_q$ and of $b_\xi(v,t)$,
\[
\int g(z)^2 e^{\eta_1 g(z)}\,\nu_q(dz)
=\int_{I_q}\sum_{v=1}^d \mu_v\,\E\Big[\Gamma_\xi(v,t)^2 e^{\eta_1\Gamma_\xi(v,t)}\Big]dt
=V_q(\xi).
\]
Since Lemma~\ref{lem:poisson-bd-external} gives $\lambda_\mathrm{P}=\eta_1(h,R)/2=c_\eta h/(2R)$, \eqref{eq:block-conditional-mgf-main} holds with $c_\lambda:=c_\eta/2$, i.e. the admissible range is $|\lambda|\le c_\lambda h/R$.
\end{proof}
We now bound the aggregated second-order exponential moments of the one-cluster envelopes $\Gamma_{\xi}$:
\begin{lemma}
\label{lem:envelope-b-profile}
Assume $0<h\le h_0$, and $R\ge 1$. Then there exist constants $c_\eta>0$, $c_\Gamma>0$, and $C_\Gamma>0$, depending only on $(\beta,\mu_+,k,w_+,\gamma,h_0)$, such that with $\eta_1(h,R):=c_\eta h/R$, 
one has, uniformly over all $d\ge 1$, all $B\in\mathcal B_{d,k}$, all indices, and all $t\in\R$,
\begin{align}
\sum_{v=1}^d \mu_v b_{Z_j}(v,t) &\le C_\Gamma h^{-2} e^{-c_\Gamma(-t)_+}, & \sum_{v=1}^d \mu_v b_{Y_i}(v,t)  &\le  C_\Gamma e^{-c_\Gamma(-t)_+}, \label{eq:b-profile-Z-Y-main} \\
\sum_{v=1}^d \mu_v b_{Z_jY_i}(v,t) &\le C_\Gamma \bigl(h^{-2}+R^2\bigr)e^{-c_\Gamma(-t)_+}, & \sum_{v=1}^d \mu_v b_{Z_aZ_b}(v,t) &\le C_\Gamma R^2 h^{-2} e^{-c_\Gamma(-t)_+}. \label{eq:b-profile-ZY-ZZ-main}
\end{align}
\end{lemma}
\begin{proof}
For $Y_i$, if the root is born at time $t$, then every contributing type-$i$ event must occur at age at least $u:=(-t)_+$. Because $N_{\{i\},u}^{\rm tail}$ counts events in the full cluster, including the root when $v=i$ and $u=0$, this bound also covers the case $t\ge 0$ in which the only contribution may come from the root itself. Thus $\Gamma_{Y_i}(x)\le N_{\{i\},u}^{\rm tail}$. Applying \eqref{eq:tail-count-moment-main} gives the first in \eqref{eq:b-profile-Z-Y-main}.\\
For $Z_j$, first suppose $t\ge 0$. Then, recall the first definition in \eqref{eq:Gamma-Z-Y-main}, then every type-$j$ event contributes at most $C_h\le C/h$ to the sampled sum (see definition \eqref{eq:H-cluster-main}), so considering every type-$j$ event $\Gamma_{Z_j}(x)\le C_h N_{\{j\},0}^{\rm tail}$. Here again $N_{\{j\},0}^{\rm tail}$ counts the root itself when $v=j$, so the bound remains valid even if the cluster has no offspring at all. Then \eqref{eq:tail-count-moment-main} yields the bound $\sum_v\mu_v b_{Z_j}(v,t)\le C h^{-2}$. Now suppose $t<0$ and write $u:=-t>0$. Split the type-$j$ events of the cluster according to their ages: events with age at least $u$ satisfy $t+s\ge 0$, so we only use $a_h(t+s)\le C_h$. For the remaining events, ages in the strip
\[
[u-m-1,u-m)\cap[0,\infty),\qquad m\ge 0,
\]
satisfy $t+s\in[-m-1,-m)$ and therefore $a_h(t+s)\le C_h e^{-\beta m}$. This includes the first partial window $[0,u)$ when $0<u<1$, namely the case $m=0$. Hence
\begin{equation}
\label{eq:Gamma-Z-strip-main}
\Gamma_{Z_j}(x)
\le C_h N_{\{j\},u}^{\rm tail}
+ C_h\sum_{m\ge 0} e^{-\beta m} N_{\{j\},u-m-1}^{\rm win}=:\sum_{r\ge 0} w_r A_r.
\end{equation}
where in the last definition we set
\[
A_0:=N_{\{j\},u}^{\rm tail},
\qquad
A_{m+1}:=N_{\{j\},u-m-1}^{\rm win}, \quad w_0:=C_h, \qquad w_{m+1}:=C_h e^{-\beta m}\quad (m\ge 0).
\]
Then let
\[
W_h:=\sum_{r\ge 0} w_r = C_h\Bigl(1+\sum_{m\ge 0}e^{-\beta m}\Bigr)\le C_\beta h^{-1}.
\]
For $\phi_{\eta,W}(y):=y^2 e^{\eta W y}$, convexity of $\phi_{\eta,W}$ on $[0,\infty)$, and Jensen's inequality for the probability weights $p_r:=w_r/W_h$ give
\[
\Bigl(\sum_{r\ge 0} w_r A_r\Bigr)^2 e^{\eta \sum_{r\ge 0} w_r A_r}
= W_h^2\, \phi_{\eta,W_h}\!\Bigl(\sum_{r\ge 0} p_r A_r\Bigr)
\le W_h^2 \sum_{r\ge 0} p_r \phi_{\eta,W_h}(A_r)
= W_h \sum_{r\ge 0} w_r A_r^2 e^{\eta W_h A_r}.
\]
Applying this with $\eta=\eta_1$ and using $W_h\le C_\beta h^{-1}$ yields
\[
\Gamma_{Z_j}(x)^2 e^{\eta_1\Gamma_{Z_j}(x)}
\le C h^{-2}\left((N_{\{j\},u}^{\rm tail})^2 e^{c_h\eta_1 N_{\{j\},u}^{\rm tail}} + \sum_{m\ge 0} e^{-\beta m}(N_{\{j\},u-m-1}^{\rm win})^2 e^{c_h\eta_1 N_{\{j\},u-m-1}^{\rm win}}\right),
\]
for some $c_h\asymp W_h\lesssim h^{-1}$. Let $\eta_{\rm C}>0$ be the exponential-moment parameter from Lemma~\ref{lem:local-count-moments}. In order to apply Lemma~\ref{lem:local-count-moments}, we need $c_h\eta_1$ sufficiently small. With $\eta_1(h,R):=c_\eta h/R$ and $W_h\le C_\beta h^{-1}$, one has $c_h\eta_1(h,R)\le C_\beta\eta_1(h,R)/h = C_\beta c_\eta/R\le C_\beta c_\eta$ uniformly in $R\ge 1$ and $h\in(0,h_0]$. It therefore suffices to choose
$c_\eta\le \min\!\left\{\frac{\eta_{\rm C}}{4C_\beta},\,\frac{\eta_{\rm C}}{2h_0}\right\}$, which depends only on $(\beta,\mu_+,k,w_+,\gamma,h_0)$. With this choice, summing over $v$, using Lemma~\ref{lem:local-count-moments} and decreasing $c_\Gamma>0$ if necessary lead to
\[
\sum_{v=1}^d \mu_v b_{Z_j}(v,t)
\le C h^{-2}\left(e^{-c_\Gamma u}+\sum_{m\ge 0} e^{-\beta m} e^{-c_\Gamma (u-m-1)_+}\right)
\le C h^{-2} e^{-c_\Gamma u}
\]
For the mixed $Z_jY_i$, write $a:=2\Gamma_{Z_j}$ and $b:=R\Gamma_{Y_i}$ and apply $(a+b)^2e^{\eta_1(a+b)}\le C(a^2e^{2\eta_1 a}+b^2e^{2\eta_1 b})$. The first term $4\Gamma_{Z_j}^2e^{4\eta_1\Gamma_{Z_j}}$ is handled exactly as in the $Z_j$ case, since the effective parameter $4c_h\eta_1 = 4C_\beta c_\eta/R\le \eta_{\rm C}$. The second term $R^2\Gamma_{Y_i}^2e^{2R\eta_1\Gamma_{Y_i}}$  has exponent $2R\eta_1 = 2R\cdot c_\eta h/R = 2c_\eta h\le \eta_{\rm C}$; since $\Gamma_{Y_i}\le N_{\{i\},u}^{\rm tail}$, Lemma~\ref{lem:local-count-moments} applies and $R^2$ gives the bound $C_\Gamma(h^{-2}+R^2)e^{-c_\Gamma(-t)_+}$ in \eqref{eq:b-profile-ZY-ZZ-main}. For $Z_aZ_b$, set $a:=2R\Gamma_{Z_a}$ and $b:=2R\Gamma_{Z_b}$; the effective parameter for each $\Gamma_{Z_a}$ term is $4Rc_h\eta_1 = 4R\cdot(C_\beta/h)\cdot(c_\eta h/R) = 4C_\beta c_\eta\le\eta_{\rm C}$, so Lemma~\ref{lem:local-count-moments} applies and yields the bound $C_\Gamma R^2h^{-2}e^{-c_\Gamma(-t)_+}$ in \eqref{eq:b-profile-ZY-ZZ-main}.
\end{proof}
We now sum the profile bounds of Lemma~\ref{lem:envelope-b-profile} over all root blocks to obtain the total variance proxy $\sum_q V_q(\xi)$ needed by the next bounded-block concentration lemma.
\begin{proposition}
\label{prop:square-profile}
Assume $0<h\le h_0$, $R\ge 1$, and $T\ge 1$. Then there exists a constant $C>0$, depending only on $(h,R,\beta,\mu_+,k,w_+,\gamma,h_0)$, such that uniformly over all $d\ge 1$, all $B\in\mathcal B_{d,k}$, and all indices,
\begin{align}
\sum_{q=-\infty}^{n-1}V_q(Z_j) &\le C\,\frac{n}{h}, &\sum_{q=-\infty}^{n-1}V_q(Y_i) &\le C\,n h,\label{eq:square-profile-Z-Y-main}\\
\sum_{q=-\infty}^{n-1}V_q(Z_jY_i) &\le C\,n\Bigl(\frac1h+R^2 h\Bigr), &\sum_{q=-\infty}^{n-1}V_q(Z_aZ_b) &\le C\,n\,\frac{R^2}{h}\label{eq:square-profile-ZY-ZZ-main}
\end{align}
\end{proposition}
\begin{proof}
Recall the definition for $V_q(\xi)$ given in Lemma~\ref{lem:block-conditional-mgf}. For $q=0,\dots,n-1$, by Lemma~\ref{lem:envelope-b-profile}:
\[
V_q(Z_j)\le C h^{-1},
\qquad
V_q(Y_i)\le C h,
\qquad
V_q(Z_jY_i)\le C(h^{-1}+R^2 h),
\qquad
V_q(Z_aZ_b)\le C R^2 h^{-1}.
\]
For $q<0$, write $t\in[qh,(q+1)h)$ so that $(-t)_+\asymp |q|h$ and therefore
\begin{align*}
    V_q(Z_j)&\le C h^{-1}e^{-c|q|h},
\quad
&V_q(Y_i)&\le C h e^{-c|q|h}, \\
V_q(Z_jY_i)&\le C(h^{-1}+R^2 h)e^{-c|q|h},
\quad
&V_q(Z_aZ_b)&\le C R^2 h^{-1}e^{-c|q|h}.
\end{align*}
Summing the geometric tails and using $T\ge 1$, hence $n=\lfloor T/h\rfloor\gtrsim h^{-1}$, gives
\begin{align*}
\sum_{q=-\infty}^{n-1}V_q(Z_j)
&\le C\left(\frac{n}{h}+\sum_{m\ge 1} h^{-1}e^{-cmh}\right)
\le C\left(\frac{n}{h}+\frac1{h^2}\right)
\le C\frac{n}{h},\\
\sum_{q=-\infty}^{n-1}V_q(Y_i)
&\le C\left(nh+\sum_{m\ge 1} he^{-cmh}\right)
\le C(nh+1)
\le Cnh,
\end{align*}
and the two product bounds follow in the same way.
\end{proof}
Now, we are ready to state the central result for this section: the lemmas above show that each individual cluster has a limited impact on the empirical averages (Lemma~\ref{lem:one-cluster-add-one} and Lemma~\ref{lem:block-conditional-mgf}), and that this impact decays as the cluster ages (Lemma~\ref{lem:relevant-lifetime-tail} and Lemma~\ref{lem:envelope-b-profile}). Combining these bounds across all clusters and all time blocks, the following proposition shows that every empirical average used by the estimator concentrates around its true value, with a tail probability that decays exponentially in $n$. This is the key statistical guarantee that will allow the estimator to succeed with high probability once $n$ is large enough.
\begin{proposition}
\label{prop:custom-bernstein}
Assume $0<h\le h_0$, $R\ge 1$, and $T\ge 1$. Then there exist constants $c,u_0>0$, depending on $(h,R,\beta,\mu_+,k,w_+,\gamma,h_0)$, such that uniformly over all $d\ge 1$, all $B\in\mathcal B_{d,k}$, and all choices of $(\xi_r)$ from the family \eqref{eq:explicit-family}, the following holds for every $u$ with $0<u\le u_0$:
\begin{align}
\Pbb\left(\left|\frac1n\sum_{r=0}^{n-1}(Z_{j,r}^{(R)}-\mathbb E Z_{j,0}^{(R)})\right|>u\right)
&\le 2\exp\left(-c n h u^2\right), \label{eq:conc-Z-main}\\
\Pbb\left(\left|\frac1n\sum_{r=0}^{n-1}(Y_{i,r}^{(h)}-\mathbb E Y_{i,0}^{(h)})\right|>u\right)
&\le 2\exp\left(-c n \frac{u^2}{h}\right), \label{eq:conc-Y-main}\\
\Pbb\left(\left|\frac1n\sum_{r=0}^{n-1}(Z_{j,r}^{(R)}Y_{i,r}^{(h)}-\mathbb E[Z_{j,0}^{(R)}Y_{i,0}^{(h)}])\right|>u\right)
&\le 2\exp\left(-c n\frac{u^2}{h^{-1}+R^2 h}\right), \label{eq:conc-ZY-main}\\
\Pbb\left(\left|\frac1n\sum_{r=0}^{n-1}(Z_{a,r}^{(R)}Z_{b,r}^{(R)}-\mathbb E[Z_{a,0}^{(R)}Z_{b,0}^{(R)}])\right|>u\right)
&\le 2\exp\left(-c n\frac{u^2}{R^2 h^{-1}}\right). \label{eq:conc-ZZ-main}
\end{align}
\end{proposition}
\begin{proof}
Let $S_n$ as in \eqref{eq:s_nstatistic} and $D_q$ as in Lemma~\ref{lem:block-conditional-mgf} with $q\le n-1$. For $M\ge 1$, set
\[
S_n^{(M)}:=\sum_{q=-M}^{n-1} D_q=\mathbb E[S_n\mid \mathcal F_{n-1}]-\mathbb E[S_n\mid \mathcal F_{-M-1}].
\]
By Lemma~\ref{lem:block-conditional-mgf}, $\mathbb E[e^{\lambda D_q}\mid \mathcal F_{q-1}]\le \exp\big(\lambda^2 V_q(\xi)\big)$ with $|\lambda|\le c_\lambda h/R$. Iterating for $q=-M,\dots,n-1$ gives
\[
\E[e^{\lambda S_n^{(M)}}]\le \exp\Big(\lambda^2\sum_{q=-M}^{n-1}V_q(\xi)\Big).
\]    
By Proposition~\ref{prop:square-profile}, $\sum_{q=-\infty}^{n-1}V_q(\xi)<\infty$. Moreover, differentiating \eqref{eq:block-conditional-mgf-main} at $\lambda=0$ shows $\E[D_q^2]\le V_q(\xi)$, so $\sum_{q=-\infty}^{n-1}\E[D_q^2]<\infty$. In particular $(S_n^{(M)})_{M\ge 1}$ is Cauchy in $L^2$. To identify its limit, note that $S_n$ is $\mathcal F_{n-1}$-measurable and $\E[S_n]=0$, hence
\[
S_n^{(M)}=\E[S_n\mid \mathcal F_{n-1}]-\E[S_n\mid \mathcal F_{-M-1}]=S_n-\E[S_n\mid \mathcal F_{-M-1}].
\]
Note that $(\mathcal F_{-M-1})_{M\ge 1}$ decrease to the tail sigma-field of the independent root blocks $(U_q)_{q\in\mathbb Z}$
\[
\mathcal F_{-\infty}:=\bigcap_{M\ge 1}\mathcal F_{-M-1},
\]
By the reverse martingale convergence theorem,
$\mathbb E[S_n\mid \mathcal F_{-M-1}]\to \E[S_n\mid \mathcal F_{-\infty}]$ in $L^2$ and by Kolmogorov's $0$-$1$ law the tail sigma-field $\mathcal F_{-\infty}$ is trivial, so $\E[S_n\mid \mathcal F_{-\infty}]=\E[S_n]=0$. Taking $|2\lambda|\le c_\lambda h/R$, the same mgf bound applied with the sign-correct parameter $2\lambda$ gives
\[
\sup_M \mathbb E[e^{2\lambda S_n^{(M)}}]\le \exp\Big(4\lambda^2\sum_{q=-\infty}^{n-1}V_q(\xi)\Big)<\infty.
\]
where, for simplicity on exposition, we choose to introduce a suitable positive constant $C$ instead of a numerical factor.
Thus, the family $\{e^{\lambda S_n^{(M)}}\}_M$ is bounded in $L^2$ and therefore uniformly integrable. Therefore
\[
\E[e^{\lambda S_n}] = \lim_{M\to\infty} \E[e^{\lambda S_n^{(M)}}]
\le \exp\Big(\lambda^2\sum_{q=-\infty}^{n-1}V_q(\xi)\Big).
\]
In order to have $\E[e^{\lambda S_n}]\le e^{K_\xi\lambda^2}$ with $|\lambda|\le c_\lambda h/(2R)$, write
\[
K_\xi:=\sum_{q=-\infty}^{n-1}V_q(\xi).
\]
For any $\lambda\in[0,c_\lambda h/(2R)]$, Chernoff's bound gives $\Pbb\!\left(S_n/n>u\right)\le \exp\!\bigl(-\lambda n u+K_\xi\lambda^2\bigr)$. Choose $\lambda=nu/2K_\xi$ (we can do this whenever $u\le u_0$, as we will see in a moment), then
\[
\Pbb\!\left(\frac{S_n}{n}>u\right)\le \exp\!\left(-\frac{n^2u^2}{4K_\xi}\right),
\]
and the same bound holds for the lower tail. Proposition~\ref{prop:square-profile} gives
\[
K_{Z_j}\lesssim \frac{n}{h},
\qquad
K_{Y_i}\lesssim nh,
\qquad
K_{Z_jY_i}\lesssim n(h^{-1}+R^2 h),
\qquad
K_{Z_aZ_b}\lesssim \frac{nR^2}{h},
\]
with constants depending only on $(h,R,\beta,\mu_+,k,w_+,\gamma,h_0)$. Thus the four displayed tail bounds \eqref{eq:conc-Z-main}--\eqref{eq:conc-ZZ-main} follow once
\[
u\le u_0:=\frac{c_\lambda h}{2R}
\min\!\left\{\frac1h,\ h,\ h^{-1}+R^2 h,\ \frac{R^2}{h}\right\},
\]
after absorbing fixed factors into the constant $c$.
\end{proof}
Next lemma translates Proposition~\ref{prop:custom-bernstein} into simultaneous high-probability control over all empirical screening statistics $\widehat{F}_{ij}^{(h,R)}$, Gram matrices $\widehat{\Sigma}_C^{(R)}$, and score vectors $\widehat{g}_{i,C}^{(h,R)}$ across all $d^2$ index pairs and all candidate sets of size at most $m$, via a union bound costing a $\log d$ factor.
\begin{lemma}
\label{lem:block-conc-consequences}
Fix $0<h\le1$, $R\ge 1$, and $m\in\mathbb N$. There exist constants $C_{\rm prob},c_{\rm prob},C_0>0$, depending only on $(m,h,R,\beta,\mu_+,k,w_+,\gamma,h_0)$, such that for every $d$, every $B\in\mathcal B_{d,k}$, and every $T\ge 1$ satisfying
\begin{equation}
\label{eq:n-conc-admissible}
n\ge C_0\max\!\left\{\frac{h^{-1}+R^2 h}{u_0(h,R)^2},\ \frac{R^2}{h\,u_0(h,R)^2}\right\}\log d,
\end{equation}
with probability at least $1-d^{-c_{\rm prob}}$ the following hold simultaneously:
\begin{enumerate}[label=(\roman*),itemsep=0.25em]
\item for all rows $i$ and all coordinates $j$,
\begin{equation}
\label{eq:screen-conc}
\abs{\widehat F_{ij}^{(h,R)}-F_{ij}^{(h,R)}}
\le C_{\rm prob}\sqrt{\frac{(h^{-1}+R^2 h)\log d}{n}};
\end{equation}
\item for all candidate sets $C$ with $|C|\le m$,
\begin{equation}
\label{eq:Gram-conc}
\norm{\widehat\Sigma_C^{(R)}-\Sigma_C^{(R)}}_\op
\le C_{\rm prob}R\sqrt{\frac{\log d}{n h}};
\end{equation}
\item for all rows $i$ and all candidate sets $C$ with $|C|\le m$,
\begin{equation}
\label{eq:g-conc}
\norm{\widehat g_{i,C}^{(h,R)}-g_{i,C}^{(h,R)}}_\infty
\le C_{\rm prob}\sqrt{\frac{(h^{-1}+R^2 h)\log d}{n}}.
\end{equation}
\end{enumerate}
\end{lemma}
\begin{proof}
(i) Under \eqref{eq:n-conc-admissible}, the screening/score deviation level $\sqrt{(h^{-1}+R^2 h)(\log d)/n}$ and the Gram deviation level $R\sqrt{(\log d)/(nh)}$ are both at most $u_0(h,R)$ after enlarging $C_0$ if necessary. Hence Proposition~\ref{prop:custom-bernstein} is applicable in each instance. For screening, write
\[
\hat m_{ZY}:=\frac1n\sum_{r=0}^{n-1} Z_{j,r}^{(R)}Y_{i,r}^{(h)},
\qquad
\hat m_Z:=\frac1n\sum_{r=0}^{n-1} Z_{j,r}^{(R)},
\qquad
\hat m_Y:=\frac1n\sum_{r=0}^{n-1} Y_{i,r}^{(h)},
\]
and similarly $m_{ZY},m_Z,m_Y$ for the expectations. Then by the definition at the beginning of Appendix~\ref{appendix:2}, $\widehat F_{ij}^{(h,R)}-F_{ij}^{(h,R)}
=(\hat m_{ZY}-m_{ZY})-(\hat m_Z\hat m_Y-m_Zm_Y)$. By Proposition~\ref{prop:custom-bernstein},
\[
|\hat m_{ZY}-m_{ZY}|\le C\sqrt{\frac{(h^{-1}+R^2 h)\log d}{n}},
\qquad
|\hat m_Z-m_Z|\le C\sqrt{\frac{\log d}{n h}},
\qquad
|\hat m_Y-m_Y|\le C\sqrt{\frac{h\log d}{n}}.
\]
Also $m_Y=\E[Y_{i,0}^{(h)}]\le Ch$ and $|\hat m_Z|\le R$. Hence summing and subtracting $\hat m _Zm_Y$ in $|\hat m_Z\hat m_Y-m_Zm_Y|$, using the bounds just obtained and the fact that $h\le1$ prove \eqref{eq:screen-conc}. 

(ii) Now for the Gram matrix, write
\[
\hat M_{ab}:=\frac1n\sum_{r=0}^{n-1} Z_{a,r}^{(R)}Z_{b,r}^{(R)},
\qquad M_{ab}:=\E[Z_{a,0}^{(R)}Z_{b,0}^{(R)}],
\]
and let $\hat m_a,m_a$ be the empirical and theoretical means of $Z_a^{(R)}$. Then as before, $(\widehat\Sigma_C^{(R)}-\Sigma_C^{(R)})_{ab}=(\hat M_{ab}-M_{ab})-(\hat m_a\hat m_b-m_am_b)$. Now Proposition~\ref{prop:custom-bernstein} yields
\[
|\hat M_{ab}-M_{ab}|\le C R\sqrt{\frac{\log d}{n h}},
\qquad
|\hat m_a-m_a|\le C\sqrt{\frac{\log d}{n h}}.
\]
Therefore
\[
|(\widehat\Sigma_C^{(R)}-\Sigma_C^{(R)})_{ab}|
\le C R\sqrt{\frac{\log d}{n h}} + 2R\cdot C\sqrt{\frac{\log d}{n h}}
\le C R\sqrt{\frac{\log d}{n h}}.
\]
Since $|C|\le m$ is fixed, \eqref{eq:Gram-conc} follows. 

(iii) For the score vector, write
\[
\hat s_a:=\frac1n\sum_{r=0}^{n-1} Z_{a,r}^{(R)}Y_{i,r}^{(h)},
\qquad
s_a:=\E[Z_{a,0}^{(R)}Y_{i,0}^{(h)}].
\]
Then
\[
(\widehat g_{i,C}^{(h,R)}-g_{i,C}^{(h,R)})_a=(\hat s_a-s_a)- (\hat m_a\hat m_Y-m_am_Y),
\]
so the same bounds as in the screening case give \eqref{eq:g-conc}. Finally, there are at most $d^2$ screening pairs $(i,j)$, at most $\sum_{r=0}^m\binom{d}{r}\le C_m d^m$ candidate sets $C$, and at most $d\sum_{r=0}^m\binom{d}{r}\le C_m d^{m+1}$ row-model pairs $(i,C)$. A union bound completes the proof.   
\end{proof}

\subsection{Sure screening and exact recovery}
\label{appendix:5}

\begin{proposition}
\label{prop:sure-screening}
Under the assumptions of Proposition~\ref{prop:one-bin-gap} and Lemma~\ref{lem:block-conc-consequences}, there exists a constant $C_{\scr}>0$ such that if
\begin{equation}
\label{eq:n-screen}
n\ge C_{\scr}\frac{h^{-1}+R^2 h}{\alpha^2 h^2}\log d,
\end{equation}
then with probability at least $1-d^{-c}$ one has $S_i\subset \widehat C_i$ for every $i=1,\dots,d$.
\end{proposition}
\begin{proof}
On the event \eqref{eq:screen-conc}, adding and subtracting $\min_{j\in S_i}F_{ij}$ and $\max_{j\notin S_i}F_{ij}$ and using conveniently the elementary inequality $|a_j-b_j|\ge a_j-b_j$, we get
\[
\min_{j\in S_i}\widehat F_{ij}^{(h,R)} - \max_{j\notin S_i}\widehat F_{ij}^{(h,R)}
\ge
\min_{j\in S_i}F_{ij}^{(h,R)} - \max_{j\notin S_i}F_{ij}^{(h,R)} - 2C_{\rm prob}\sqrt{\frac{(h^{-1}+R^2 h)\log d}{n}}.
\]
By Proposition~\ref{prop:one-bin-gap}, the theoretical screening gap is at least $c_{\scr}\alpha h$. Thus if
\[
2C_{\rm prob}\sqrt{\frac{(h^{-1}+R^2 h)\log d}{n}}\le \frac12 c_{\scr}\alpha h,
\]
then every support index outranks every non-support index. Since $m\ge k\ge |S_i|$, the top-$m$ empirical screening set contains $S_i$. This is implied by \eqref{eq:n-screen} after enlarging the constant.
\end{proof}
\begin{proposition}
\label{prop:ols-error}
Under the assumptions in Lemma~\ref{lem:Gram-nondeg}, Proposition~\ref{prop:population-beta}, and Lemma~\ref{lem:block-conc-consequences}, there exists a constant $C_{ols}>0$ such that if
\begin{equation}
\label{eq:n-ols}
n\ge C_{ols}\max\!\left\{\frac{R^2}{h},\frac{h^{-1}+R^2 h}{\alpha^2 h^2}\right\}\log d,
\end{equation}
then with probability at least $1-d^{-c}$ the following holds simultaneously for all rows $i$ and all candidate sets $C\supset S_i$ with $|C|\le m$:
\begin{equation}
\label{eq:ols-inf-error}
\|\widehat y_{i,C}^{(h,R)}-y_{i,C}^{(h,R)}\|_\infty
\le C_{ols}\sqrt{\frac{(h^{-1}+R^2 h)\log d}{n}}.
\end{equation}
In particular, on the sure-screening event, the same bound holds for $C=\widehat C_i$ for every $i$.
\end{proposition}
\begin{proof}
Fix $(i,C)$. On the event \eqref{eq:Gram-conc}, Lemma~\ref{lem:Gram-nondeg} leads to $\|\widehat\Sigma_C^{(R)}-\Sigma_C^{(R)}\|_{\text{op}}\le \kappa/2$. Then using Weyl's inequality $\widehat\Sigma_C^{(R)}$ is invertible and $\|(\widehat\Sigma_C^{(R)})^{-1}\|_{\text{op}}\le 2/\kappa$. Moreover, \[
y_{i,C}^{(h,R)}-y_{i,C}^{(h,R)}=(\widehat\Sigma_C^{(R)})^{-1}\Bigl[\widehat g_{i,C}^{(h,R)}-g_{i,C}^{(h,R)} + \Sigma_C^{(R)}y_{i,C}^{(h,R)} -\widehat\Sigma_C^{(R)}y_{i,C}^{(h,R)}\Bigr],
\]
then rearranging:
\[
\widehat y_{i,C}^{(h,R)}-y_{i,C}^{(h,R)}=
(\widehat\Sigma_C^{(R)})^{-1}\Bigl[(\widehat g_{i,C}^{(h,R)}-g_{i,C}^{(h,R)}) - (\widehat\Sigma_C^{(R)}-\Sigma_C^{(R)})y_{i,C}^{(h,R)}\Bigr].
\]
Proposition~\ref{prop:population-beta} implies $\|y_{i,C}^{(h,R)}\|_\infty\le C\alpha h$ uniformly, hence $\|y_{i,C}^{(h,R)}\|_2\le C_{m}\alpha h$. Therefore \eqref{eq:Gram-conc} and \eqref{eq:g-conc} give
\[
\|\widehat y_{i,C}^{(h,R)}-y_{i,C}^{(h,R)}\|_\infty
\le
\frac{2}{\kappa}
\Bigl(
\|\widehat g_{i,C}^{(h,R)}-g_{i,C}^{(h,R)}\|_\infty + C_m\alpha h\|\widehat\Sigma_C^{(R)}-\Sigma_C^{(R)}\|_{\text{op}}\Bigr)
\le C_{ols}\sqrt{\frac{(h^{-1}+R^2 h)\log d}{n}}.
\]
A union bound over all row/model pairs is already built into Lemma~\ref{lem:block-conc-consequences}. The requirement \eqref{eq:n-ols} guarantees both the invertibility condition and the final bound above.
\end{proof}

We can now proof the upper bound in Theorem~\ref{thm:upper} from the main text.
\begin{proof}[Proof of Theorem~\ref{thm:upper}]
Choose $\alpha_0>0$ small enough that Propositions~\ref{prop:one-bin-gap}, \ref{prop:population-beta}, and Lemma~\ref{lem:Gram-nondeg} all apply for every $0<\alpha\le \alpha_0$. From the proofs of Proposition~\ref{prop:clipped-screen-gap}, Lemma~\ref{lem:Gram-nondeg}, and Proposition~\ref{prop:population-beta}, choose
\[
A_R:=\max\left\{\frac{16\beta C_{\rm clip}}{\mu_-w_-},\,\frac{16\beta C_m}{\mu_-},\,\frac{8C_y}{w_-},\,1\right\},
\]
and set $R(\alpha):=A_R\alpha^{-1}$. Next choose
\[
A_h:=\frac12\min\left\{1,\frac{\mu_-w_-}{32\beta C_{\rm bin}A_R},\frac{w_-}{8C_yA_R}\right\},
\]
and set $h(\alpha):=A_h\alpha^2$. With this pair, the screening gap \eqref{eq:F-gap}, the theoretical least-squares slope separation \eqref{eq:population-beta-threshold}, and the Gram nondegeneracy all hold simultaneously.

Let $n=\lfloor T/h(\alpha)\rfloor$. Once $\alpha$ is fixed, the substituted values $h(\alpha)$ and $R(\alpha)$ are also fixed, so the constants from Lemma~\ref{lem:block-conc-consequences}, Proposition~\ref{prop:sure-screening}, and Proposition~\ref{prop:ols-error} become finite numbers depending on $\alpha$. Besides the screening and OLS requirements from Propositions~\ref{prop:sure-screening} and \ref{prop:ols-error}, we must also satisfy the admissible-range condition \eqref{eq:n-conc-admissible}. For $\alpha\le \alpha_0$ we have $h(\alpha)\le h_0=1$ and $R(\alpha)\ge 1$, so $u_0(h(\alpha),R(\alpha))\ge c h(\alpha)^2/R(\alpha)$. Therefore, for this fixed
\(\alpha\), all remaining requirements are implied by
\[
n\ge A_\alpha \log d,
\]
where
\[
A_\alpha
:=
C_0\max\left\{
\frac{h(\alpha)^{-1}+R(\alpha)^2 h(\alpha)}{\alpha^2 h(\alpha)^2},
\frac{R(\alpha)^2}{h(\alpha)},
\frac{\bigl(h(\alpha)^{-1}+R(\alpha)^2 h(\alpha)\bigr)R(\alpha)^2}{h(\alpha)^4},
\frac{R(\alpha)^4}{h(\alpha)^5}
\right\}.
\]
Now choose \(C_\ast>0\), depending only on
\((\beta,\mu_\pm,w_\pm,k,m,\alpha)\), large enough that, for every integer
\(d\ge2\),
\[
T\ge C_\ast\log d
\]
implies both \(T\ge1\) and $n=\lfloor T/h(\alpha)\rfloor\ge A_\alpha\log d$.
Hence Lemma~\ref{lem:block-conc-consequences},
Proposition~\ref{prop:sure-screening}, and Proposition~\ref{prop:ols-error}
apply simultaneously.
Now, with probability at least $1-d^{-c_*}$ one has simultaneously
\[
S_i\subset \widehat C_i\qquad\text{for all }i,
\]
and
\[
\norm{\widehat y_{i,\widehat C_i}^{(h,R)}-y_{i,\widehat C_i}^{(h,R)}}_\infty
\le \frac18 \alpha w_- h.
\]
Then for every row $i$, because $\widehat C_i\supset S_i$ and \eqref{eq:population-beta-threshold} applies with $C=\widehat C_i$,
\[
\min_{j\in S_i}\widehat y_{ij}
\ge
\min_{j\in S_i}y_{ij,\widehat C_i}^{(h,R)} - \frac18\alpha w_- h
\ge \frac58 \alpha w_- h > \frac12 \alpha w_- h = \tau_{\alpha,h},
\]
while for $j\in \widehat C_i\setminus S_i$,
\[
\abs{\widehat y_{ij}}
\le
\abs{y_{ij,\widehat C_i}^{(h,R)}} + \frac18\alpha w_- h
\le \frac38 \alpha w_- h < \frac12 \alpha w_- h = \tau_{\alpha,h}.
\]
Hence thresholding at $\tau_{\alpha,h}$ recovers $S_i$ exactly for every row. The overall success probability is at least $1-d^{-c_*}$ for some $c_*>0$.\qedhere
\end{proof}

\begin{remark}
The main goal of this work is to characterize the scaling of the observation time, which we prove to scale logarithmically in the dimension of the network. Naturally, the dependence of the multiplicative constant on other parameters may be loose. In particular, a formal power counting of the coupling parameter $\alpha$ suggests a scaling of $\alpha^{-12}$, which is at great variance to the scaling of $\alpha^{-2}$ of the lower bound, leaving an optimality gap of $\alpha^{-10}$. As the proof requires $\alpha$ to be small, this may seem to limit the practical applicability of the result. First, we would like to emphasize that the proposed procedure is first and foremost a proof tool and is not intended to be a practical algorithm to replace established methods. Second, we will show in the following analysis that the $\alpha^{-12}$ scaling is more an artifact of the proof rather than the intrinsic scaling of the problem.
\end{remark}

\subsection{On the intrinsic $\alpha$-scaling of the problem}
\label{app:intrinsic_alpha}
Here, we show that replacing Lemma~\ref{lem:clip-bias}, which requires one new Lemma and an updated Corollary, already improves the $\alpha$-dependence to $\alpha^{-4}\text{ polylog}(1/\alpha)$, thereby improving the optimality gap from $\alpha^{-10}$ to $\alpha^{-2}\text{ polylog}(1/\alpha)$. This does not yet reveal the intrinsic $\alpha$-dependence of the problem, but may help future work to characterize it in full detail, which is why we report it here.

The starting point is again the cluster representation, in which we previously defined the intensity measure $\nu(dv,dt,d\mathcal C)=\mu_v \, dt \,\mathbb{P}_v(d\mathcal C)$ for the Poisson point process $(v,t,\mathcal C)$. Recall also the definition given before for $N_{J,s}^{win}(\mathcal C)$, i.e. the number of type-$J$ events in the truncated unit window $[s, s + 1) \cap  [0, \infty)$, and consider the result at Lemma~\ref{lem:local-count-moments}. The main changes needed in order to obtain the stated power counting are the following:
\begin{lemma}
    There exists $\eta_X>0$ and $C_X<\infty$ depending only on $(\beta,\mu_+,k,w_+,\gamma_0)$ such that $$\sup_{d,B,j}\mathbb{E}[e^{\eta_X X_j(0)}]\leq C_X$$
\end{lemma}
\begin{proof}
    Following the usual cluster representation and identifying $a$ as the age of the offspring w.r.t. the root $v$, write
\begin{equation*}
    X_j(0)=\sum_{(v,t,\mathcal C)}f_j(t,\mathcal C) \quad,\quad\quad f_j(t,\mathcal C)=\sum_{(j,a)\in\mathcal C, \, t+a\leq0}e^{\beta(t+a)}
\end{equation*}  
Apply Campbell's theorem for the Laplace functional with the function $\eta f_j(t,\mathcal C)>0$ (with $\eta>0$) in order to obtain $$\log\mathbb{E}[e^{\eta X_j(0)}]=\sum_v \mu_v \int_0^{\infty} \mathbb{E}_v[e^{\eta f_j(-u,\mathcal C)}-1]du$$ where $t=-u$ is substituted in the integral. Using $e^x-1\leq x e^x$ for $x\geq 0$ leads to
\begin{equation}
    \log\mathbb{E}[e^{\eta X_j(0)}]\leq \eta \sum_v\mu_v \int_0^{\infty} \mathbb{E}_v[ f_je^{\eta f_j}]du.
    \label{eq: baseline for proof bis}
\end{equation}
In order to bound $\sum_v\mu_v \mathbb{E}_v[fe^{\eta f}]$, consider again the cluster representation: fixing a root time $t=-u\leq 0$, partition the contributing ages in $a\in[u-m-1,u-m)$ (with $m\geq 0$) and $a=u$. For the former, the contributing weights in the representation are $e^{-\beta(u-a)}<e^{-\beta m}$ and their number is at most $N^{win}_{\{j\},\, u-m-1}$; in fact $-u+a\in[-m-1,-m)$ and therefore $u-a > m$. For the latter, the weight is $1$ and the number is $N^{win}_{\{j\},\,u}$. Thus, since $[0,u]=\{u\} \cup \bigcup_{m\geq0}[u-m-1,u-m)$, the pathwise bound is: $$f\leq1\cdot N^{win}_{\{j\},\, u}+\sum_{m\geq 0} e^{-\beta m}\cdot N^{win}_{\{j\},\, u-m-1}=\sum_{r\geq 0}w_rA_r$$ and, taking $w_0=1$ and $w_{m+1}=e^{-\beta m}$, the total weight is: $$W=\sum_{r\geq 0}w_r=1+\sum_{m\geq0}e^{-\beta m} =1+\frac{1}{1-e^{-\beta}}=: C_{\beta}<\infty.$$ Defining $p_r=w_r/W$, apply Jensen in the second inequality (noting that a function as $\phi(y)=y \exp\{\eta W y\}$ is increasing and convex on $[0,\infty)$):
\begin{equation*}
        fe^{\eta f}\leq  W \biggl ( \sum_{r\geq 0}p_rA_r \biggr) e^{\eta W\sum_{r\geq 0}p_rA_r} \leq W \sum_{r\geq 0}p_r A_r e^{\eta W A_r}= \sum_{r\geq 0}w_r A_r e^{\eta W A_r}
\end{equation*}
Then apply $\mathbb{E}_v[\cdot]$ and sum against $\mu_v$ to have:
\begin{equation}
    \sum_v \mu_v \mathbb{E}_v [fe^{\eta f}]\leq \sum_{r\geq 0}w_r \sum _v \mu _v\mathbb{E}_v[ A_r e^{\eta W A_r}].
    \label{eq: previous_to proof bis}
\end{equation}
Choose $\eta:=\frac{\eta_C}{C_{\beta}}$ so that $\eta W = \eta_C$. It trivially holds that $A_r\leq A_r^2$ (in fact $A_r$ takes values in $\{0,1,2,...\}$), then also $A_r e^{\eta_c A_r}\leq A_r^2 e^{\eta_c A_r}$ pointwise. Using this inequality and Lemma~\ref{lem:local-count-moments} with the convention that $N^{win}_{J,s}=0$ if $s<-1$, we get:
$$\sum_v\mu_v\mathbb{E}_v[A_r e^{\eta WA_r}]\leq \sum_v \mu_v \mathbb{E}_v[A_r^2 e^{\eta_C A_r}] \leq  C_{cl} e^{-c_{cl}(s_r)_+} $$ 
with $s_0=u$ and $s_{m+1}=u-m-1$. Substituting this into the inequality \eqref{eq: previous_to proof bis} leads to: $$\sum_v \mu_v \mathbb{E}_v[fe^{\eta f}]\leq C_{cl}[e^{-c_{cl}u}+\sum_{m\geq 0}e^{-\beta m}e^{-c_{cl}(u-m-1)_+}].$$ Finally, integrate in $u\in[0,\infty)$ term by term; then substituting the result in \eqref{eq: baseline for proof bis} gives $$\log \mathbb{E}[e^{\eta X_j(0)}]\leq \eta\,C_{\mathrm{cl}}\Big[\frac{1}{c_{\mathrm{cl}}}+\sum_{m\ge0}e^{-\beta m}\Big(m+1+\frac{1}{c_{\mathrm{cl}}}\Big)\Big] =: \eta C'<\infty.$$ 
\end{proof}
Notice that for every integer $p\geq1$ and $x\geq0$, $x^p\leq p! e^x$; hence $\mathbb{E}[X_j(0)^p]=\eta_X^{-p}\mathbb{E}[(\eta_X X_j(0))^p] \leq p! \eta_X^{-p} \mathbb{E}[e^{\eta_X X_j(0)}]\leq p! \eta_X^{-p} C_X=:M'_p$ uniformly over $(d,B,j)$. This translates to the fact that Proposition~\ref{prop:moments} can be substituted with known constants only depending on $(\beta, \mu_+, k, w_+)$, without any threshold on $\alpha$ (i.e. without any $\alpha_P$). \\
Now, using this Lemma, we can state a new clipping corollary.
\begin{corollary}
    There exist $c_0>0$, $C_0<\infty$ such that, for all $R\ge1$, uniformly
over $(d,B,i,j)$, with $\delta^{(R)}_j:=(X_j(0)-R)_+$, $$\mathbb E\big[\delta^{(R)}_j\big]\;\le\;C_0e^{-c_0R},
\qquad
\mathbb E\big[(\delta^{(R)}_j)^{2}\big]^{1/2}\;\le\;C_0e^{-c_0R},
\qquad
\mathbb E\big[\delta^{(R)}_j\,\lambda_i(0)\big]\;\le\;C_0e^{-c_0R}.$$
\end{corollary}
\begin{proof}
    For $x>0$, using $(x-R)_+\le x\,\mathbf 1\{x>R\}\le x\,e^{\eta_X(x-R)/2}$ and $ye^{y}\le e^{2y}$, we get the chain of inequalities: $$\mathbb{E}[\delta^{(R)}]\le e^{-\eta_XR/2}\,\mathbb{E}\big[X e^{\eta_XX/2}\big]
\le e^{-\eta_XR/2}\cdot\frac{2}{\eta_X}\,\mathbb{E}\big[e^{\eta_XX}\big]
\le \frac{2C_X}{\eta_X}\,e^{-\eta_XR/2}$$ where in the last inequality the previous Lemma is used.
Similarly for the second: use $(x-R)_+^2\leq x^2e^{\eta_X(x-R)/2}$ and $y^2e^y\leq C e^{2y}$; this gives $\mathbb{E}[(\delta^(R))^2]\leq C C_X e^{-\eta_X R/2}$. The third follows by Cauchy--Schwarz from the second and $\sup\mathbb{E}[\lambda_i(0)^{2}]<\infty$ (see how this is done in the proof for Lemma~\ref{lem:clip-bias} as a reference). Take $c_0:=\eta_X/4$ and $C_0$ the largest of the constants.
\end{proof}
The following Lemma would now substitute Lemma~\ref{lem:clip-bias}:
\begin{lemma}[Lemma~\ref{lem:clip-bias} changed]
\label{lem: new lemma a2}
    For all $R\ge1$, uniformly over $(d,B,i,j,\ell)$:
    \begin{equation*}
        \mathbb{E}[\delta^{(R)}_j]\le C_0e^{-c_0R},\qquad
        \mathbb{E}[(\delta^{(R)}_j)^{2}]\le C_0^{2}e^{-2c_0R},
     \end{equation*}
     \begin{equation*}
    \big|\Cov\big(Z^{(R)}_j,\delta^{(R)}_\ell\big)\big|\ \vee\
    \big|\Cov\big(\delta^{(R)}_j,\lambda_i(0)\big)\big|\le C_1e^{-c_0R}.
    \end{equation*}
\end{lemma}
\begin{proof}
    The first two are already proven. For the third: $|\Cov(Z_j ,\delta_{\ell})|\leq \mathbb{E}[Z_j\delta_\ell] + \mathbb{E}[Z_j]\mathbb{E}[\delta_\ell]\leq \mathbb{E}[X_j \delta_\ell] +M'_1\mathbb{E}[\delta_\ell]$ and $\mathbb{E}[X_j\delta_\ell]\leq (M'_{2})^{1/2}\mathbb{E}[\delta_\ell^2]^{1/2}\leq C e^{-c_0 R}$ by Cauchy-Schwarz and the previous result. Similarly, using some results from the proof for Lemma~\ref{lem:stability_stationarymean_momen} and similarly to some considerations in the proof for Lemma~\ref{lem:clip-bias}, we get $|\Cov(\delta_j,\lambda_i)|\leq \mathbb{E}[\delta_j \lambda_i] + \mathbb{E}[\delta_j]\mathbb{E}[\lambda_i]\leq Ce^{-c_0 R}+ C_0 e^{-c_0R} \beta \bar m$ with $\bar m := \sup_j m_j \leq \mu_+/(\beta(1-\gamma))$. Since $e^{-c_0 R}\leq C_{c_0}/R$, any results whose proof uses Lemma~\ref{lem:clip-bias} remain true.
\end{proof}

Choose now $R=A_R \log(1/\alpha)$ (see Proposition~\ref{prop:clipped-screen-gap} and Proposition~\ref{prop:population-beta} revisited via Lemma~\ref{lem: new lemma a2}) and consequently $h=A_h\,\alpha/\log(1/\alpha)$ (see Proposition~\ref{prop:one-bin-gap} and Proposition~\ref{prop:population-beta}). The same formal power counting as before, now results in $\alpha^{-4} \text{ polylog}(\alpha)$, which reduces the optimality gap w.r.t. the lower bound to $\alpha^{-2} \text{ polylog}(\alpha)$. This does not require any restriction on the model class considered.

\section{An information-theoretic view of support recovery}
\label{app:it-intro}

We briefly introduce the information-theoretic perspective of support recovery that is used in the lower bound.
We do not derive any new statements, but instead aim to explain why entropy, mutual information, KL divergence, and Fano's inequality naturally appear in exact support recovery.

Information theory studies how much can be learned about an unknown object from noisy observations. Let \(V\) be a discrete random variable taking values in a finite set \(\mathcal V\), and suppose that, conditional on \(V=v\), the observation \(Y\) has law \(P_v\). In communication language, \(V\) is the message, the stochastic mechanism \(v\mapsto P_v\) is a noisy channel, and an estimator \(\widehat V=\widehat V(Y)\) is a decoder. In a statistical problem,
the same viewpoint applies when \(V\) indexes a finite collection of hypotheses or models.

Support recovery fits naturally into this framework because the target is discrete. In the hard subclass used below, the unknown object is the parent set
\[
  S\in\mathcal S_{d,k}(i_\star),
  \qquad
  |\mathcal S_{d,k}(i_\star)|=\binom{d-1}{k}.
\]
Thus \(S\) can be viewed as a message chosen from \(\mathcal S_{d,k}(i_\star)\). The Hawkes process acts as the noisy channel: conditional on \(S\), it produces the observation
\[
  Y_T=(X(0),N_{[0,T]})
\]
with law \(P_S^T\). A support estimator \(\widehat S=\widehat S(Y_T)\) is a decoder trying to infer which message was sent.

The entropy of a discrete random variable \(V\) is
\[
  H(V):=-\sum_{v\in\mathcal V}\Pbb(V=v)\log \Pbb(V=v),
\]
where logarithms are natural, so information is measured in nats.
Entropy measures the uncertainty in \(V\) before observing any data.
If \(V\) is uniform on \(\mathcal V\), then \(H(V)=\log|\mathcal V|\).
Hence choosing one object among \(M\) equally likely alternatives requires \(\log M\) nats of information. 
In our support-recovery problem, a uniformly chosen parent set satisfies
\[
  H(S)=\log |\mathcal S_{d,k}(i_\star)|
  =
  \log\binom{d-1}{k}.
\]
For fixed \(k\), this quantity is of order \(\log d\). This is the basic combinatorial uncertainty that the data must resolve.

After observing \(Y\), the remaining uncertainty about \(V\) is measured by the conditional entropy \(H(V\mid Y)\). The reduction in uncertainty caused by the observation is the mutual information
\[
  I(V;Y):=H(V)-H(V\mid Y).
\]
If \(Y\) is independent of \(V\), then \(I(V;Y)=0\). If \(Y\) determines \(V\) exactly, then \(I(V;Y)=H(V)\). Thus mutual information measures how much of the initial uncertainty about the message is actually transmitted through the noisy channel.

Mutual information can also be expressed using KL divergence. For probability measures \(P\) and \(Q\),
\[
  \KL(P\|Q)
  :=
  \E_P\left[\log\frac{dP}{dQ}\right],
\]
whenever \(P\) is absolutely continuous with respect to \(Q\). KL divergence is not a metric, but it measures statistical distinguishability: small KL divergence means that observations from \(P\) are hard to distinguish from observations from \(Q\). If \(V\) has prior distribution \(\pi\), and
\[
  \bar P:=\sum_{v\in\mathcal V}\pi_v P_v
\]
is the mixture law of \(Y\), then
\[
  I(V;Y)
  =
  \sum_{v\in\mathcal V}\pi_v \KL(P_v\|\bar P).
\]
Thus the information carried by \(Y\) about \(V\) is the average distinguishability between the law generated by a fixed message and the overall mixture law. 
To get a lower-bound, it is sufficeint to compare against a convenient reference law \(Q\):
\[
  I(V;Y)
  \le
  \sum_{v\in\mathcal V}\pi_v \KL(P_v\|Q).
\]

Fano's inequality is the step that turns an information bound into an error lower bound. 
The idea is the following. 
Suppose \(V\) is uniform on a finite set \(\mathcal V\) with \(M=|\mathcal V|\), and let \(\widehat V=\widehat V(Y)\) be any estimator. 
If the estimator is usually correct, then after seeing \(Y\) there cannot be much uncertainty left about \(V\): the value \(\widehat V(Y)\) almost identifies the true message. 
The only substantial uncertainty comes from the event \(\{\widehat V\neq V\}\). 
On the event of no error, the message is known from the decoder's output. On the error event, the true message may still be one of many alternatives. 
Therefore a small error probability forces the conditional
entropy \(H(V\mid Y)\) to be small.

Fano's inequality formalizes the converse implication. If the observation leaves large conditional entropy, then no estimator can have small error probability. 
Equivalently, since
\[
  H(V\mid Y)=H(V)-I(V;Y),
\]
if the mutual information \(I(V;Y)\) is much smaller than the initial uncertainty \(H(V)=\log M\), then the observation has not conveyed enough information to identify the message. The finite-hypothesis form used here is
\[
  \Pbb(\widehat V\neq V)
  \ge
  1-\frac{I(V;Y)+\log 2}{\log M}.
\]
The additive \(\log 2\) term accounts for the binary uncertainty of whether an error occurred, and is negligible when \(M\) is large. 
The main message of the inequality is that reliable identification among \(M\) nearly hidden possibilities requires nearly \(\log M\) nats of mutual information.

This average-error statement immediately gives a minimax lower bound. For any prior \(\pi\),
\[
  \sup_{v\in\mathcal V}\Pbb_v(\widehat V\neq v)
  \ge 
  \sum_{v\in\mathcal V}\pi_v \Pbb_v(\widehat V\neq v).
\]
Hence, if Fano's inequality shows that the prior-average error is bounded away from zero, then every estimator must fail with non-negligible probability for at least one hypothesis in the class.

In our Hawkes setting, the message set is
\[
  \mathcal V=\mathcal S_{d,k}(i_\star),
  \qquad
  M=\binom{d-1}{k}
\]
and for fixed \(k\), \(\log M\asymp \log d\). We show that the observation up to time \(T\) carries at most
\[
  C_{\mathrm{init}}+C_{\mathrm{path}}T
\]
nats of information about this support, with constants independent of \(d\) and \(T\). Fano's inequality then compares the information available from the data to the information needed to identify the support. 
If \(T=o(\log d)\), the data contains too little information to distinguish among the possible parent sets, and the probability of an error is bounded away from zero.

\section{Lower bound proof details}\label{app:lower-bound}
Here, we provide the full proof of Theorem~\ref{thm:lower-bound}. 
As before, the unknown parent set is $S\in\mathcal S_{d,k}(i_\star)$, the active edge weight in the subclass is $\thetalo$, and the target and source background rates are $\bar\mu_\star$ and $\bar\mu$. Let $P_S^T$ be the law of the stationary observation $Y_T=(X(0),N_{[0,T]})$ under support $S$, and let $P_0^T$ denote the law of the reference model with no interactions and the same background rates. We write $\pi_S$ and $\pi_0$ for the corresponding stationary laws of $X(0)$, and $K_S^T(x,\cdot)$ and $K_0^T(x,\cdot)$ for the conditional laws of the path $N_{[0,T]}$ given $X(0)=x$.

The proof is structured in four parts. In Section~\ref{app:lower-prelim}, we collect two standard entropy facts and properties of one-dimensional Poisson shot-noise density function needed for the stationary initial
condition. In Section~\ref{app:lower-initial}, we reduce the initial KL term to a fixed $(k+1)$-dimensional
problem. In Section~\ref{app:lower-dynamic}, we bound the dynamic path KL by a second moment of the source
block. Finally, in Section~\ref{app:lower-fano} we apply Fano's inequality to prove Theorem~\ref{thm:lower-bound}.

\subsection{Preparatory lemmas}\label{app:lower-prelim}
First, we prove the chain-rule for the Kullback--Leibler divergence.
\begin{lemma}\label{lem:app-chain}
For every $S\in\mathcal S_{d,k}(i_\star)$,
\[
\KL(P_S^T\|P_0^T)
=
\KL(\pi_S\|\pi_0)
+
\int \KL\!\bigl(K_S^T(x,\cdot)\|K_0^T(x,\cdot)\bigr)\,\pi_S(dx),
\]
with equality in $[0,\infty]$.
\end{lemma}

\begin{proof}
This is the standard chain rule for relative entropy of probability kernels. When $P_S^T\ll P_0^T$, write
\[
P_S^T(dx,d\eta)=\pi_S(dx)K_S^T(x,d\eta),
\qquad
P_0^T(dx,d\eta)=\pi_0(dx)K_0^T(x,d\eta),
\]
and expand the Radon--Nikodym derivative. The extended-real case follows by the usual approximation argument for relative entropy of
kernels.
\end{proof}
The next lemma bounds the mutual information by the average KL to a baseline.
\begin{lemma}
\label{lem:app-mi-baseline}
Let $V$ be uniform on a finite set $\mathcal V$, and let $Y$ be an
observation whose law under $V=v$ is $P_v$. Then for any probability
measure $Q$,
\[
I(V;Y)\le \frac{1}{|\mathcal V|}\sum_{v\in\mathcal V}\KL(P_v\|Q).
\]
\end{lemma}

\begin{proof}
Let $\bar P:=|\mathcal V|^{-1}\sum_{v\in\mathcal V}P_v$ be the mixture law.
Then
\[
\frac1{|\mathcal V|}\sum_v \KL(P_v\|Q)
=\frac1{|\mathcal V|}\sum_v \KL(P_v\|\bar P)+\KL(\bar P\|Q)
=I(V;Y)+\KL(\bar P\|Q),
\]
follows by writing
$\log(dP_v/dQ)=\log(dP_v/d\bar P)+\log(d\bar P/dQ)$,
integrating under $P_v$, and averaging over $v$. The claim follows from
$\KL(\bar P\|Q)\ge 0$.
\end{proof}

The next result collects several properties of the density of the stationary law of
one-dimensional Poisson shot noise. It is later used to show that translated stationary densities have finite KL cost.

For $r>0$, let $\pi_r^{\mathrm P}$ be the stationary law of
$dU(t)=-\beta U(t)\,dt+dM(t)$, where $M$ is Poisson with rate $r$.
Equivalently,
\[
U_r\stackrel{d}{=}\int_{(-\infty,0]} e^{\beta s}\,dM_s
\stackrel{d}{=}\int_0^\infty e^{-\beta t}\,dM_t
\]
and its characteristic function is
\begin{equation}\label{eq:app-1d-CF}
\E\bigl[e^{itU_r}\bigr]
=\exp\!\left(\frac{r}{\beta}\int_0^1\frac{e^{itu}-1}{u}\,du\right),
\qquad t\in\R.
\end{equation}
Thus $U_r$ is infinitely divisible with L\'evy density $\nu_r(du)=(r/\beta)\mathbf1_{(0,1)}(u)u^{-1}\,du$.

\begin{lemma}
\label{lem:app-1d-density}
Fix $r>0$ and set $\alpha=r/\beta$. The law $\pi_r^{\mathrm P}$ has a density
$q_r$ on $(0,\infty)$ such that:
{\renewcommand{\labelenumi}{\textnormal{(\roman{enumi})}}
\begin{enumerate}
\item $q_r$ is strictly positive on $(0,\infty)$ and unimodal.
\item There exists a constant $c_r>0$ such that
\[
q_r(x)=c_r x^{\alpha-1}
\quad\text{for Lebesgue-a.e. }x\in(0,1).
\]
\item There exists a finite constant $C_r$ such that for Lebesgue-a.e.
$x>0$,
\begin{equation}\label{eq:app-global-log-q}
-\log q_r(x)\le C_r\bigl(1+x\log(1+x)+|\log x|\bigr).
\end{equation}
\item $U_r$ has exponential moments of every order,
\begin{equation}\label{eq:app-mgf-Ur}
\E[e^{\eta U_r}]
=\exp\!\left(\frac{r}{\beta}\int_0^1\frac{e^{\eta u}-1}{u}\,du\right)
<\infty,
\qquad \eta>0,
\end{equation}
and hence
$\E[U_r\log(1+U_r)]<\infty$, $\E[|\log U_r|]<\infty$, and
$\int q_r|\log q_r|<\infty$.
\end{enumerate}
}
\end{lemma}

\begin{proof}
The L\'evy measure has density $k(u)u^{-1}\,du$ with
$k(u)=\alpha\mathbf 1_{(0,1)}(u)$ nonincreasing. Sato
\cite[Theorem~15.10 and Corollary~15.11]{sato1999levy} therefore implies
that $U_r$ is self-decomposable. Absolute continuity, strict positivity on
$(0,\infty)$, and unimodality then follow from Sato
\cite[Theorem~27.13, Lemma~53.1, and Theorem~53.1]{sato1999levy}. This proves
(i).

For (ii), stationarity gives, for every $f\in C_c^1(0,\infty)$,
\[
\int_0^\infty \{ -\beta x f'(x)+r(f(x+1)-f(x)) \}q_r(x)\,dx=0.
\]
Taking $f\in C_c^1(0,1)$ removes the jump-in term and yields
$\beta(xq_r)'-rq_r=0$ in $\mathcal D'((0,1))$. Hence
$(x^{1-\alpha}q_r)'=0$ distributionally, so
$q_r(x)=c_r x^{\alpha-1}$ for a.e. $x\in(0,1)$; positivity gives $c_r>0$.
After redefining \(q_r\) on a null set, we may take this identity to hold
pointwise on \((0,1)\).

For (iii), use a unimodal version of $q_r$, and choose a point $m_r\ge0$ such
that $q_r$ is nonincreasing on $[m_r,\infty)$. Fix
$\delta\in(0,\min\{1/(2\beta),1\})$. For $n\ge1$, decompose
$U_r=U_r^{(1,n)}+U_r^{(2,n)}+U_r^{(3)}$ over the time intervals
$[0,\delta/n]$, $(\delta/n,1]$, and $(1,\infty)$ in the representation
$U_r=\int_0^\infty e^{-\beta t}\,dM_t$. Since
$U_r^{(3)}\stackrel d=e^{-\beta}U_r'$ with $U_r'$ an independent copy of
$U_r$, strict positivity of the density gives
$p_\delta:=\Pbb(U_r^{(3)}\in[\beta\delta,1])>0$. On the event
\[
A_n:=\{M([0,\delta/n])=n\},\quad
B_n:=\{M((\delta/n,1])=0\},\quad
C:=\{U_r^{(3)}\in[\beta\delta,1]\},
\]
one has $U_r\in[n,n+1]$: indeed
$U_r^{(1,n)}\ge n e^{-\beta\delta/n}\ge n-\beta\delta$, while
$U_r^{(2,n)}=0$ and $U_r^{(3)}\in[\beta\delta,1]$. Independence and Stirling's
upper bound give constants $c_1,c_2>0$, depending only on $(r,\beta)$, such
that
\[
\Pbb(U_r\in[n,n+1])
\ge e^{-r\delta/n}\frac{(r\delta/n)^n}{n!}\,e^{-r(1-\delta/n)}p_\delta
\ge c_1 e^{-c_2 n\log n}.
\]
For integers $n\ge\lceil m_r\rceil$, monotonicity gives
$q_r(n)\ge\int_n^{n+1}q_r(x)\,dx$, hence
$q_r(n)\ge c_1e^{-c_2n\log n}$. If $x\ge m_r+1$ and $n=\lceil x\rceil$, then
$q_r(x)\ge q_r(n)$, and therefore
\[
-\log q_r(x)\le C\{1+x\log(1+x)\},\qquad x\ge m_r+1.
\]
Near zero, part (ii) gives
$-\log q_r(x)\le C(1+|\log x|)$ for a.e. $x\in(0,1)$. 
As before, use a unimodal version so that \(q_r\) is nondecreasing on \((0,m_r]\)
and nonincreasing on \([m_r,\infty)\). On the remaining compact interval
\([1,m_r+1]\), this version is strictly positive, and by unimodality its
infimum is attained at one of the endpoints. Hence the infimum is positive.
Combining these three bounds proves \eqref{eq:app-global-log-q}.

Finally, (iv) follows from the Poisson shot-noise formula
\[
\log \E[e^{\eta U_r}]
= r\int_0^\infty (e^{\eta e^{-\beta t}}-1)\,dt
= \frac{r}{\beta}\int_0^1\frac{e^{\eta u}-1}{u}\,du<\infty,
\qquad \eta>0,
\]
because $e^{\eta u}-1\le u(e^\eta-1)$ on $[0,1]$. The exponential moments
imply $\E[U_r\log(1+U_r)]<\infty$ and $\E[\log_+ U_r]<\infty$. For the
negative logarithmic moment, (ii) gives
\[
\int_0^1 |\log x|q_r(x)\,dx
= c_r\int_0^1 |\log x|x^{\alpha-1}\,dx<\infty,
\]
so $\E[|\log U_r|]<\infty$.
Moreover, (ii) gives integrability of \(q_r|\log q_r|\) near zero. On compact
subintervals of \((0,\infty)\), the chosen unimodal version is locally bounded
by monotonicity on the two sides of its mode, after redefining the value at the
mode if necessary; hence \(q_r\log_+q_r\) is integrable there. On the tail
\([m_r+1,\infty)\), unimodality gives \(q_r(x)\le q_r(m_r+1)\), so the
positive part is bounded by a constant times \(q_r\). The negative part is
controlled by \eqref{eq:app-global-log-q} and the moment bounds above.
Therefore \(\int q_r|\log q_r|<\infty\).
\end{proof}

We now use the properties from Lemma~\ref{lem:app-1d-density} to control the KL of translating the target's baseline stationary law by a nonnegative random amount.

\begin{lemma}\label{lem:app-translation-kl}
Let $q_0:=q_{\bar\mu_\star}$ be the density of $\pi_{\bar\mu_\star}^{\mathrm P}$.
For $w\ge 0$, define the translated density
\[
q_{0,w}(z):=q_0(z-w)\,\mathbf 1_{\{z>w\}},
\qquad z>0,
\]
that is, the law of $U+w$ when $U\sim \pi_{\bar\mu_\star}^{\mathrm P}$.
Then there exists a finite constant $C_{\mathrm{tr}}$, depending only on
$(\beta,\bar\mu_\star)$, such that
\[
g(w):=\KL(q_{0,w}\,dz\|q_0\,dz)
\le C_{\mathrm{tr}}\bigl(1+w\log(1+w)\bigr),
\qquad w\ge 0.
\]
\end{lemma}

\begin{proof}
Let $U\sim\pi_{\bar\mu_\star}^{\mathrm P}$. By a change of variables,
\[
g(w)=\E[\log q_0(U)-\log q_0(U+w)].
\]
Lemma~\ref{lem:app-1d-density}(iv) gives
$J_0:=\int q_0|\log q_0|<\infty$. Since $U+w$ is absolutely continuous,
\eqref{eq:app-global-log-q} applies to it a.s., and
\[
g(w)\le J_0+C\E\{1+(U+w)\log(1+U+w)+|\log(U+w)|\}.
\]
The elementary bound $1+U+w\le(1+U)(1+w)$ gives
\[
(U+w)\log(1+U+w)
\le U\log(1+U)+w\log(1+U)+(U+w)\log(1+w),
\]
whose expectation is bounded by $C(1+w\log(1+w))$. Also,
\[
|\log(U+w)|\le |\log U|+\log(1+U)+\log(1+w),
\]
because the negative part decreases when $w$ is added and the positive part is
bounded by $\log(1+U+w)$. Lemma~\ref{lem:app-1d-density}(iv) makes all
constants finite, and the claim follows.
\end{proof}

\subsection{KL of the stationary initial law}\label{app:lower-initial}

We now control the KL contribution from the state $X(0)$.
The point of the lower-bound subclass is that all irrelevant coordinates factor
out, leaving a fixed-dimensional comparison involving only the target coordinate
and its $k$ active parents.
Let \(\bar\pi_k\) denote the stationary law of the \((k+1)\)-dimensional block
\((Z,\mathbf Y)\) in which \(Y^{(1)},\dots,Y^{(k)}\) are independent
Poisson shot-noise coordinates with rate \(\bar\mu\), and \(Z\) has intensity
\(\bar\mu_\star+\thetalo\sum_{m=1}^k Y^{(m)}_{t-}\).

\begin{proposition}
\label{prop:app-initKL-k}
For every $d\ge k+2$ and every $S\in\mathcal S_{d,k}(i_\star)$,
\[
\KL(\pi_S\|\pi_0)
=
\KL\!\bigl(\bar\pi_k\|\pi_{\bar\mu_\star}^{\mathrm P}\otimes
(\pi_{\bar\mu}^{\mathrm P})^{\otimes k}\bigr)
=:C_{\mathrm{init}}.
\]
The constant $C_{\mathrm{init}}$ is finite and depends only on
$(\beta,\bar\mu,\bar\mu_\star,\thetalo,k)$; in particular it is independent of
both $S$ and $d$.
\end{proposition}

\begin{proof}
Fix $S=\{s_1,\dots,s_k\}\in\mathcal S_{d,k}(i_\star)$ and order the remaining
coordinates in any fixed way. Let $R_S$ be the coordinate permutation that puts
$i_\star$ first, then the elements of $S$, and then all other coordinates. KL is
invariant under this bijection.

Under $\Theta^{(S)}$, all coordinates outside $\{i_\star\}\cup S$ are independent
Poisson shot-noise coordinates with stationary law $\pi_{\bar\mu}^{\mathrm P}$,
and they are independent of the target/source block. The nontrivial block
$(Z_t,\mathbf Y_t):=(X_{i_\star}(t),(X_{s_1}(t),\dots,X_{s_k}(t)))$ satisfies
\[
\begin{aligned}
dY_t^{(m)}&=-\beta Y_t^{(m)}\,dt+dN_t^{(m)},\qquad m=1,\dots,k,\\
dZ_t&=-\beta Z_t\,dt+dM_t,
\end{aligned}
\]
where $N^{(1)},\dots,N^{(k)}$ are independent Poisson processes of rate
$\bar\mu$ and $M$ has predictable intensity
$\bar\mu_\star+\thetalo\sum_{m=1}^kY_{t-}^{(m)}$. If $\bar\pi_k$ denotes the
stationary law of this block, then
\[
R_S\#\pi_S=\bar\pi_k\otimes(\pi_{\bar\mu}^{\mathrm P})^{\otimes(d-k-1)}.
\]
Under the baseline model,
\[
R_S\#\pi_0=\pi_{\bar\mu_\star}^{\mathrm P}\otimes
(\pi_{\bar\mu}^{\mathrm P})^{\otimes k}\otimes
(\pi_{\bar\mu}^{\mathrm P})^{\otimes(d-k-1)}.
\]
Product additivity of relative entropy gives
\[
\KL(\pi_S\|\pi_0)=
\KL\!\bigl(\bar\pi_k\|\pi_{\bar\mu_\star}^{\mathrm P}\otimes
(\pi_{\bar\mu}^{\mathrm P})^{\otimes k}\bigr),
\]
and the right-hand side is independent of the particular set $S$ by
exchangeability of the source coordinates.

It remains only to prove finiteness. Let
$\mathcal H=\sigma(Y_s^{(m)}:s\le0,1\le m\le k)$ be the stationary source
history. Conditional on $\mathcal H$, the target process on $(-\infty,0]$ is an
inhomogeneous Poisson process with intensity
$\bar\mu_\star+\thetalo\sum_mY_{t-}^{(m)}$. Decompose it as $M=M^{(0)}+M^{(1)}$,
where, conditional on $\mathcal H$, $M^{(0)}$ is homogeneous Poisson of rate
$\bar\mu_\star$, $M^{(1)}$ is Poisson with intensity
$\thetalo\sum_mY_{t-}^{(m)}\,dt$, and the two are independent. Then
\[
Z_0=U+W,
\qquad
U:=\int_{(-\infty,0]}e^{\beta s}\,dM_s^{(0)},\qquad
W:=\int_{(-\infty,0]}e^{\beta s}\,dM_s^{(1)}.
\]
Here $U\sim\pi_{\bar\mu_\star}^{\mathrm P}$ and $U$ is independent of
$(\mathcal H,W)$.

The $\mathbf Y_0$-marginal of $\bar\pi_k$ is
$(\pi_{\bar\mu}^{\mathrm P})^{\otimes k}$. Hence the entropy chain rule and
conditional Jensen's inequality give
\[
\KL\!\bigl(\bar\pi_k\|\pi_{\bar\mu_\star}^{\mathrm P}\otimes
(\pi_{\bar\mu}^{\mathrm P})^{\otimes k}\bigr)
= \E\bigl[\KL(P_{Z_0\mid \mathbf Y_0}\|\pi_{\bar\mu_\star}^{\mathrm P})\bigr]
\le \E\bigl[\KL(P_{Z_0\mid \mathcal H}\|\pi_{\bar\mu_\star}^{\mathrm P})\bigr].
\]
Indeed, since \(\mathbf Y_0\) is \(\mathcal H\)-measurable,
\(P_{Z_0\mid \mathbf Y_0}\) is obtained by averaging
\(P_{Z_0\mid \mathcal H}\) over the conditional law of \(\mathcal H\) given
\(\mathbf Y_0\), and relative entropy is convex in its first argument.
Given $\mathcal H$, the conditional law of $Z_0$ is the mixture of translates
$q_{0,w}(z)\,dz$ with mixing law $P_{W\mid\mathcal H}$. Convexity of KL and
Lemma~\ref{lem:app-translation-kl} therefore yield
\[
\KL(P_{Z_0\mid \mathcal H}\|\pi_{\bar\mu_\star}^{\mathrm P})
\le \E[g(W)\mid\mathcal H]
\le C\, \E[1+W\log(1+W)\mid\mathcal H].
\]
Thus finiteness follows once $\E[W\log(1+W)]<\infty$.

In fact $W$ has exponential moments of all orders. For $\eta>0$, the conditional
Laplace functional gives
\[
\E[e^{\eta W}\mid\mathcal H]
=
\exp\!\left(\thetalo\int_{-\infty}^0
(e^{\eta e^{\beta s}}-1)\sum_{m=1}^kY_{s-}^{(m)}\,ds\right)
\le
\exp\!\left(\thetalo(e^\eta-1)V\right),
\]
where
\[
V:=\int_{-\infty}^0 e^{\beta s}\sum_{m=1}^kY_{s-}^{(m)}\,ds.
\]
Using the stationary shot-noise representation of the sources and Tonelli's
theorem,
\[
V=\sum_{m=1}^k\int_{(-\infty,0]}(-u)e^{\beta u}\,dN_u^{(m)}
=\int_{(-\infty,0]}(-u)e^{\beta u}\,d\widetilde N_u,
\]
where $\widetilde N$ is homogeneous Poisson with rate $k\bar\mu$. The kernel
$u\mapsto(-u)e^{\beta u}\mathbf 1_{\{u\le0\}}$ is bounded and integrable, so
Campbell's formula implies $\E[e^{cV}]<\infty$ for every $c>0$. Consequently
$\E[e^{\eta W}]<\infty$ for every $\eta>0$, and in particular
$\E[W\log(1+W)]<\infty$. This proves $C_{\mathrm{init}}<\infty$ and completes
the proof.
\end{proof}

\subsection{Dynamic KL bound}\label{app:lower-dynamic}

It remains to control the entropy accumulated over the observation window.
We start by computing the stationary second moment of the aggregated source block.

\begin{lemma}\label{lem:app-moment-k}
For $S\in\mathcal S_{d,k}(i_\star)$, set
$\Sigma_S(t):=\sum_{j\in S}X_j(t)$. Under $\Theta^{(S)}$, $\Sigma_S$ is a
stationary Poisson shot-noise OU process driven by a homogeneous Poisson process
of rate $k\bar\mu$, and
\[
\E[\Sigma_S(t)^2]
=\frac{k^2\bar\mu^2}{\beta^2}+\frac{k\bar\mu}{2\beta}
=:C_{\mathrm{path}},
\qquad t\in\R.
\]
\end{lemma}

\begin{proof}
The $k$ parent coordinates are independent Poisson shot-noise OU processes with
rate $\bar\mu$, so their sum is driven by the superposed Poisson process of rate
$k\bar\mu$. Campbell's formula gives
$\E[\Sigma_S]=k\bar\mu/\beta$ and
$\operatorname{Var}(\Sigma_S)=k\bar\mu/(2\beta)$, which proves the statement for
$\E[\Sigma_S^2]$.
\end{proof}

In the following proposition, we derive the bound on the dynamic KL term.
\begin{proposition}\label{prop:app-dynamic-KL}
For every $S\in\mathcal S_{d,k}(i_\star)$,
\[
\int \KL\!\bigl(K_S^T(x,\cdot)\|K_0^T(x,\cdot)\bigr)\,\pi_S(dx)
\le
\frac{\thetalo^2}{\bar\mu_\star}\,C_{\mathrm{path}}\,T.
\]
Consequently,
\[
\KL(P_S^T\|P_0^T)\le C_{\mathrm{init}}
+\frac{\thetalo^2}{\bar\mu_\star}\,C_{\mathrm{path}}\,T.
\]
\end{proposition}

\begin{proof}
Fix $x\in\R^d$. Under model $S$, the intensity is
\[
\lambda_{i_\star}^{(S)}(t)=\bar\mu_\star+\thetalo\sum_{j\in S}X_j(t-),
\qquad
\lambda_i^{(S)}(t)\equiv\bar\mu \quad (i\neq i_\star),
\]
whereas under the baseline model
\[
\lambda_{i_\star}^{(0)}(t)\equiv\bar\mu_\star,
\qquad
\lambda_i^{(0)}(t)\equiv\bar\mu \quad (i\neq i_\star).
\]

For \(\pi_S\)-a.e. initial state \(x\), both point-process systems are
nonexplosive on \([0,T]\), and the only differing intensity is the
\(i_\star\)-coordinate. Since both target intensities are bounded below by
\(\bar\mu_\star>0\), the conditional laws satisfy
\(K_S^T(x,\cdot)\ll K_0^T(x,\cdot)\).
Hence, we can apply Jacod's \cite{jacod1975multivariate} Girsanov formula for multivariate point processes (\cite[Chapter~VI]{bremaud1981point} and \cite{kabanov1978capacity} for more details) to the single differing coordinate $i_\star$,
\begin{align}
\KL\!\bigl(K_S^T(x,\cdot)\|K_0^T(x,\cdot)\bigr)
&=
\E_S\!\left[
\int_0^T\!\left(
\lambda_{i_\star}^{(S)}(t)\log\frac{\lambda_{i_\star}^{(S)}(t)}{\bar\mu_\star}
-\lambda_{i_\star}^{(S)}(t)+\bar\mu_\star
\right)\!dt
\,\middle|\, X(0)=x
\right].
\label{eq:app-cond-KL}
\end{align}
For $a\ge 0$ and $b>0$, $a\log(a/b)-a+b\le (a-b)^2/b$ and the resulting entropy term is finite after averaging over $\pi_S$ with $b=\bar\mu_\star$ and
Lemma~\ref{lem:app-moment-k}. Applying this with $a=\bar\mu_\star+\thetalo\sum_{j\in S}X_j(t-)$
and  $b=\bar\mu_\star$ gives
\[
\KL\!\bigl(K_S^T(x,\cdot)\|K_0^T(x,\cdot)\bigr)
\le
\frac{\thetalo^2}{\bar\mu_\star}
\int_0^T \E_S\!\Bigl[\Bigl(\sum_{j\in S}X_j(t)\Bigr)^2\Bigm| X(0)=x\Bigr]\,dt.
\]
Integrating over $\pi_S(dx)$ and applying stationarity yields
\[
\int \KL\!\bigl(K_S^T(x,\cdot)\|K_0^T(x,\cdot)\bigr)\,\pi_S(dx)
\le
\frac{\thetalo^2}{\bar\mu_\star}\int_0^T\!\E_S[\Sigma_S(t)^2]\,dt
=\frac{\thetalo^2}{\bar\mu_\star}\,C_{\mathrm{path}}\,T.
\]
The last claim follows from Lemma~\ref{lem:app-chain} and Proposition~\ref{prop:app-initKL-k}.
\end{proof}

\subsection{Proof of Theorem~\ref{thm:lower-bound}}\label{app:lower-fano}
We now turn the KL bound into a lower bound over the hard subclass using Fano's inequality\cite{gerchinovitz2020fano}.
Let
\[
M:=|\mathcal S_{d,k}(i_\star)|=\binom{d-1}{k},
\]
and let $S^\star$ be uniform on $\mathcal S_{d,k}(i_\star)$. The following
finite-$d$ bound is the quantitative form used in the main text.
\begin{proposition}\label{prop:app-finite-fano}
For every $d\ge k+2$ and every $T>0$,
\[
\inf_{\widehat S}\sup_{S\in\mathcal S_{d,k}(i_\star)}
\Pbb_S(\widehat S(Y_T)\neq S)
\ge
1-
\frac{C_{\mathrm{init}}+\frac{\thetalo^2}{\bar\mu_\star}C_{\mathrm{path}}T+
\log 2}
{\log\binom{d-1}{k}}.
\]
\end{proposition}

\begin{proof}
For any estimator $\widehat S=\widehat S(Y_T)$, Fano's inequality gives
\[
\sup_{S\in\mathcal S_{d,k}(i_\star)}\Pbb_S(\widehat S\neq S)
\ge
1-\frac{I(S^\star;Y_T)+\log 2}{\log M}.
\]
The baseline mutual-information inequality in Lemma~\ref{lem:app-mi-baseline},
with $Q=P_0^T$, yields
\[
I(S^\star;Y_T)
\le \frac1M\sum_{S\in\mathcal S_{d,k}(i_\star)}\KL(P_S^T\|P_0^T).
\]
By Lemma~\ref{lem:app-chain}, Proposition~\ref{prop:app-initKL-k}, and
Proposition~\ref{prop:app-dynamic-KL}, each summand satisfies
\[
\KL(P_S^T\|P_0^T)
\le C_{\mathrm{init}}+\frac{\thetalo^2}{\bar\mu_\star}C_{\mathrm{path}}T.
\]
Substituting this into Fano's inequality proves the claim.

\end{proof}
We now prove the information-theoretic lower bound.
\begin{proof}[Proof of Theorem~\ref{thm:lower-bound}]

Let \(C_{\mathrm{path}}\) be the constant from Lemma~\ref{lem:app-moment-k}, and let
\(C_{\mathrm{init}}\) be the constant from Proposition~\ref{prop:app-initKL-k}.
Since $k$ is fixed,
\[
\log\binom{d-1}{k}=k\log d+O(1),\qquad d\to\infty.
\]
Hence there are constants $a_k>0$ and $d_1\ge k+2$, depending only on $k$, such
that
\[
\log\binom{d-1}{k}\ge a_k\log d,\qquad d\ge d_1.
\]
Choose $d_0\ge d_1$ large enough that
\[
\frac{C_{\mathrm{init}}+\log 2}{\log\binom{d-1}{k}}\le \frac14,
\qquad d\ge d_0,
\]
and set
\[
c:=\frac{a_k\bar\mu_\star}{4\thetalo^2 C_{\mathrm{path}}},
\qquad \varepsilon:=\frac12.
\]
If $d\ge d_0$ and $T\le c\log d$, then
\[
\frac{\thetalo^2}{\bar\mu_\star}C_{\mathrm{path}}T
\le \frac14\log\binom{d-1}{k}.
\]
The finite bound in Proposition~\ref{prop:app-finite-fano} therefore gives
\[
\inf_{\widehat S}\sup_{S\in\mathcal S_{d,k}(i_\star)}
\Pbb_S(\widehat S(Y_T)\neq S)
\ge \frac12.
\]
This proves the theorem with the above constants. 
The same argument also gives the
corresponding full-class impossibility statement, since $\mathcal G_{d,k}^{\mathrm{sub}}(i_\star)\subseteq\mathcal G_{d,k}$. Indeed, any full support estimator $\widehat G(Y_T)$ induces
a row estimator $\widehat S(Y_T)$ by taking its estimated support in row
$i_\star$. On the subclass, the event $\{ \widehat S(Y_T)\ne S \}$ is contained
in the event that the full support is estimated incorrectly, and therefore
\[
\inf_{\widehat G}
\sup_{(\mu,\Theta)\in\mathcal G_{d,k}}
\Pbb_{\mu,\Theta}\!\left(\widehat G(Y_T)\ne \supp(\Theta)\right)
\ge
\inf_{\widehat S}
\sup_{S\in\mathcal S_{d,k}(i_\star)}
\Pbb_S\!\left(\widehat S(Y_T)\ne S\right).
\]
\end{proof}


\end{document}